\documentclass[11pt]{amsart}
\usepackage[margin=0.9in]{geometry}
\usepackage{amsmath}
\usepackage{amsfonts}
\usepackage{amssymb}
\usepackage{graphicx}
\usepackage{hyperref}
\usepackage{mathrsfs}
\usepackage{graphicx,color}
\usepackage[mathscr]{eucal}
\usepackage{caption}
\usepackage{subcaption}
\usepackage{color}
\usepackage[all]{xy}
\usepackage{comment}
\usepackage{tikz-cd}
\usepackage{graphicx,color}
\usepackage{mathrsfs}
\usepackage{marvosym}
\usepackage{mathtools,tikz-cd}
\usepackage{xcolor}
\usepackage{enumitem}
\usepackage{graphicx}

\usepackage{tikz}
\usetikzlibrary{shapes.geometric, arrows}
\tikzstyle{startstop} = [rectangle, rounded corners, minimum width=2.5cm, minimum height=0.5cm,text centered, text width=2cm, draw=black, fill=white!30]
\tikzstyle{startstop2} = [rectangle, rounded corners, minimum width=2.5cm, minimum height=0.5cm,text centered, text width=2cm, draw=black, fill=white!30]
\tikzstyle{startstop3} = [rectangle, rounded corners, minimum width=2.5cm, minimum height=0.5cm,text centered, text width=2.5cm, draw=black, fill=white!30]
\tikzstyle{arrow} = [thick,->,>=stealth]

\newtheorem{theorem}{Theorem}[section]
\newtheorem{lemma}[theorem]{Lemma}

\newtheorem{corollary}[theorem]{Corollary}
\newtheorem{proposition}[theorem]{Proposition}
\newtheorem{definition}[theorem]{Definition}
\newtheorem{example}[theorem]{Example}
\newtheorem{remark}[theorem]{Remark}
\newtheorem{conjecture}[theorem]{Conjecture}

\numberwithin{equation}{section}
\numberwithin{figure}{section}

\newcommand{\defeq}{\vcentcolon=}

\newcommand{\ul}[1]{\underline{#1}}

\newcommand{\ZZ}{\mathbb{Z}}

\newcommand{\RR}{\mathbb{R}}

\newcommand{\Z}{\mathbb{Z}}

\newcommand{\kg}{\mathfrak{g}}

\newcommand{\FF}{\mathbb{F}}

\newcommand{\CC}{\mathbb{C}}

\newcommand{\cY}{\mathcal{Y}}

\newcommand{\Cone}{\mathrm{Cone}}

\newcommand{\cL}{\mathcal{L}}

\newcommand{\Ham}{\mathrm{Ham}}
\newcommand{\Symp}{\mathrm{Symp}}

\newcommand{\crit}{\mathrm{crit}}

\newcommand{\cl}{\operatorname{cl}}
\newcommand{\Conf}{\operatorname{Conf}}
\newcommand{\Hilb}{\operatorname{Hilb}}
\newcommand{\Sym}{\operatorname{Sym}}

\newcommand{\length}{\operatorname{length}}

\newcommand{\Br}{\operatorname{Br}}
\newcommand{\AKh}{\operatorname{AKh}}
\newcommand{\Kh}{\operatorname{Kh}}
\newcommand{\CKh}{\operatorname{CKh}}
\newcommand{\Fix}{\operatorname{Fix}}
\newcommand{\id}{\operatorname{id}}

\def\eL{\EuScript{L}}

\newcommand{\CFhat}{\widehat{\mathit{CF}}}
\newcommand{\HFhat}{\widehat{\mathit{HF}}}
\newcommand{\CFKhat}{\widehat{\mathit{CFK}}}
\newcommand{\HFKhat}{\widehat{\mathit{HFK}}}
\newcommand{\HFLhat}{\widehat{\mathit{HFL}}}

\newcommand{\alphas}{{\boldsymbol \alpha}}
\newcommand{\betas}{{\boldsymbol \beta}}

\newcommand{\HFhatDual}{\HFhat^{\raisebox{-3pt}{$\scriptstyle *$}}}

\newcommand{\HFKhatDual}{\HFKhat^{\raisebox{-3pt}{$\scriptstyle *$}}}
\newcommand{\HFLhatDual}{\HFLhat^{\raisebox{-3pt}{$\scriptstyle *$}}}

\newcommand{\fix}{\mathrm{fix}}

\newcommand{\co}{\nobreak\mskip2mu\mathpunct{}\nonscript
  \mkern-\thinmuskip{:}\penalty300\mskip6muplus1mu\relax}

\usepackage{placeins}

\begin{document}

\title
[Symplectic annular Khovanov homology and fixed point localizations]
{Symplectic annular Khovanov homology and fixed point localizations}

\author{Kristen Hendricks}
\address{Department of Mathematics\\Rutgers University\\  New Brunswick, NJ 08854}
\email{kristen.hendricks@rutgers.edu}
\thanks{KH was partially supported by NSF CAREER grant DMS-2019396 and a Sloan Research Fellowship.}

\author{Cheuk Yu Mak}
\address{School of Mathematical and Physical Sciences\\ University of Sheffield\\ Sheffield S10 2TN, UK\\}
\email{c.mak@sheffield.ac.uk}
\thanks{CYM was partially supported by the Royal Society University Research Fellowship and the University of Southampton FSS ECR Career Development Fund}

\author{Sriram Raghunath}
\address{Department of Mathematics\\Rutgers University\\  New Brunswick, NJ 08854}
\email{sr1508@math.rutgers.edu}
\thanks{SR was partially supported by NSF CAREER grant DMS-2019396.}

\begin{abstract}
We introduce a new version  of symplectic annular Khovanov homology and establish spectral sequences from (i) the symplectic annular Khovanov homology of a knot to the link Floer homology of the lift of the annular axis in the double branched cover; (ii) the symplectic Khovanov homology of a two-periodic knot to the symplectic annular Khovanov homology of its quotient; and (iii) the symplectic Khovanov homology of a strongly invertible knot to the cone of the axis-moving map between the symplectic annular Khovanov homology of the two resolutions of its quotient.
\end{abstract}

\maketitle

\tableofcontents

\section{Introduction}

\subsection{Background}

Khovanov homology is a link invariant first defined by Khovanov in 2000 \cite{Khovanov:Jones} which associates to $L \subset S^3$ and a choice of ring $R$ a bigraded module $\Kh(L; R)$, the Euler characteristic of which recovers the Jones polynomial. In its original form, it is defined combinatorially from a diagram of the link in a straightforward manner; however, despite this simplicity, it has proved to contain a wealth of geometric information and has been used for numerous deep applications, e.g. \cite{Rasmussen:Milnor, Ng:tb, LZ:ribbon, Olga:transverse, Piccirillo:Conway}.

Khovanov homology has connections to invariants from Floer theory, and in particular to Heegaard Floer homology. The first of these was discovered in 2004, when Ozsv{\'a}th and Szab{\'o} \cite{OS:branched} showed the existence of a spectral sequence
\begin{equation} \label{eq:ozsz} \widetilde{\Kh}(mL; \FF_2) \rightrightarrows \HFhat(\Sigma(L); \FF_2).\end{equation}
Here $\widetilde{\Kh}$ is a variant called the \emph{reduced} Khovanov homology, which has half the dimension of unreduced Khovanov homology, and $mL$ is the mirror of $L$. The $E_{\infty}$ page is isomorphic to $\HFhat(\Sigma(L); \FF_2)$, the simplest Heegaard Floer invariant of the double branched cover $\Sigma(L)$ of $L$.\footnote{This spectral sequence also has a lift to integral coefficients, with the following caveat. Khovanov homology has two non-equivalent sign assignments, the ``even'' and ``odd'' variations. The Oszv{\'a}th-Szab{\'o} spectral sequence has an integral lift to odd Khovanov homology \cite{ORS:odd}. Symplectic Khovanov homology corresponds to even Khovanov homology.}. For more on the Heegaard Floer theories, see Section~\ref{subsec:knotfloer}.

Inspired by connections between Khovanov homology and representation theory, in 2006 Seidel and Smith defined a conjectural Floer-theoretic analog of Khovanov homology, called the symplectic Khovanov homology $\Kh_{symp}(L)$ \cite{SS06}. Given a braid representation $b \in Br_n$ for a link $L$ and its Plat closure $b'=b\times 1^n$ in $Br_{2n}$, they consider a symplectic fibration over the configuration space of $2n$ distinct points in the plane. The fibre $\mathcal Y_n$ of the fibration contains a standard Lagrangian $\eL$; the authors construct a parallel transport action $\phi_{b'}$ on the fibration which lifts the action on the configuration space, and then let $\Kh_{symp}(L)$ be the Lagrangian Floer cohomology of $\eL$ and $\phi_{b'}(\eL)$ in $\mathcal Y_n$. (For more details on the construction, see Section \ref{s:SympAKh}.) Although it is harder to compute than combinatorial Khovanov homology, symplectic Khovanov homology has the advantage that its connections to Heegaard Floer homology and its behavior with respect to symmetries of knots and links can be understood in terms of the geometry of the theory. For example, shortly after its introduction Manolescu embedded $\mathcal Y_{n}$ as an open subset of a Hilbert scheme of an affine algebraic surface in a way which gave a clear connection to the Heegaard Floer homology of the double branched cover of $L$ \cite{Manolescu:nilpotent}. Using this and technical work in Lagrangian Floer cohomology, Seidel and Smith were able to produce an analog of the Ozsv{\'a}th-Szab{\'o} spectral sequence \cite{SS10}. (Since this time, many further analogs of the Ozsv{\'a}th-Szab{\'o} spectral sequence and generalizations have been extensively studied \cite{KM:KhUnknot, Scaduto:instantons, Daemi:gauge, SSS:geometric, BHL:functoriality, Dowlin:KhtoHFK}.)

In 2019, Abouzaid and Smith proved that over fields of characteristic zero, Khovanov homology and symplectic Khovanov homology are isomorphic in an appropriate sense \cite{AS16, AS19}. In particular, symplectic Khovanov homology is a singly-graded theory whereas Khovanov homology is a bigraded theory; Abouzaid and Smith show that
\[ \Kh_{symp}^i(L; \mathbb Q) \simeq \oplus_{j-q=i} \Kh^{j,q}(L; \mathbb Q) \]
(Symplectic Khovanov homology was later given a relative second grading over fields of characteristic zero by Cheng \cite{Cheng23}; his techniques do not extend to fields of finite characteristic.) It is also known that the theories agree over integer coefficients for quasi-alternating links, since both satisfy an unoriented skein exact triangle, tautologically in the case of Khovanov homology and by work of Abouzaid and Smith for symplectic Khovanov homology \cite[Section 7 and Appendix]{AS19}. The general relationship of the two theories, however, remains unknown.

Khovanov homology has an annular variation, as follows. Recall that an \emph{annular link} is the union of a link $L$ and an additional unknot component $A$ linked with $L$, called the annular axis (and usually suppressed from the notation). \emph{Annular Khovanov homology}, introduced by Asaeda, Przytycki, and Sikora \cite{APS:annular} and elaborated by L.~Roberts \cite{Roberts:knotfloer}, uses the data of the annular axis to induce an annular filtration on the Khovanov chain complex. The homology of the associated graded chain complex with respect to the filtration is the annular Khovanov homology $\AKh(L) = \oplus_{j, q, k} \AKh^{j,q,k}(K)$, a triply-graded theory. As with Khovanov homology, this complex may be computed from a diagram. The theory also has an interpretation in terms of the Hochschild homology of braid bimodules over extended arc algebras, established by Beliakova, Putyra, and Wehrli \cite{BPW:trace} following a conjecture of Auroux, Grigsby, and Wehrli \cite{AGW:sutured}. Annular Khovanov homology plays a key role in the understanding of equivariant versions of Khovanov homology; we discuss the details in Section~\ref{subsec:equivariant}. 

In 2019, the second author and Smith constructed an annular variant of symplectic Khovanov homology, defined over fields of characteristic zero, and showed that it agrees with combinatorial annular Khovanov homology over such fields \cite{MS22}. Their theory is inspired by the Hochschild homology interpretation of annular Khovanov homology; despite its technical elegance, understanding group actions on it geometrically is difficult. Moreover, while it is believed to be extendable to finite characteristic, a new suite of transversality arguments would be required in order to carry out such an extension \cite[Section 2]{MS22}.

In this paper, we construct a second, conceptually simple version of annular symplectic Khovanov homology $\AKh_{symp}(L)$, definable over any field. Our construction involves inducing a filtration on the Lagrangian Floer cochain complex $CKh_{symp}(L)$ via counting intersections of the pseudoholomorphic disks appearing in the differential with a holomorphic divisor $D_o \subseteq \mathcal Y_n$. This induces an annular filtration on the chain complex; the homology of the associated graded with respect to this filtration is the symplectic annular Khovanov homology $\AKh_{symp}(L)$. Invariance of our theory follows quickly from similar arguments to symplectic Khovanov homology.

\begin{theorem} \label{thm:invariant} The theory $\AKh_{symp}(L)$ is an annular link invariant. \end{theorem}

The second author and Smith establish a geometrically-constructed symplectic analog of the spectral sequence from annular Khovanov homology to ordinary Khovanov homology for their construction \cite[Section 12]{MS22}. For our theory, the corresponding theorem is immediate from the algebraic definitions.

\begin{theorem} \label{thm:akh-to-kh} There is a spectral sequence with $E_1$ page $\AKh_{symp}(L)$ and $E_{\infty}$ page isomorphic to $\Kh_{symp}(L)$, every page of which is an invariant of the annular link.
 \end{theorem}

We expect that our theory is isomorphic to the Mak-Smith formulation of symplectic annular Khovanov homology where both are defined, which is to say over characteristic zero, and thus is isomorphic to combinatorial annular Khovanov homology, but do not attempt to prove as much in the present paper. Section~\ref{sec:motivation} is devoted to an explanation of our motivation and the precise statement of this conjecture as Conjecture~\ref{con:isom}; we plan to investigate it in this paper's sequel.

Our primary motivation for this construction involves equivariant versions of Khovanov and symplectic Khovanov homology with respect to various intrinsic and extrinsic symmetries of the theories, which we now discuss.

\subsection{Equivariant symplectic Khovanov homologies} \label{subsec:equivariant} The symplectic manifolds and Lagrangians used to define symplectic Khovanov homology and the annular symplectic Khovanov homology of the present paper admit an $O(2)$-action. Considering the equivariant cohomology with respect to various subsets of this intrinsic action, and with respect to various actions induced by extrinsic symmetries on the link, has produced a variety of interesting relationships between theories and inspired others in combinatorial Khovanov homology.

We begin by reminding the reader of some details of Manolescu's reformulation of symplectic Khovanov homology \cite{Manolescu:nilpotent}, which is reviewed at greater length in Sections~\ref{subsec:nilpotent} and \ref{subsec:bridge}. Given a link $L$, one chooses a bridge diagram for $L$ in $\mathbb C$. If this bridge diagram has $n$ bridges, let the set of endpoints of the bridges be $\tau = [\tau_1, \dots, \tau_{2n}]$. Letting $p(z) = \prod_{i=1}^{2n} (z-\tau_i)$, one considers a Milnor fibre
\[ A_{\tau} = \{ (u,v,z) \subset \mathbb C^3 : u^2+v^2 = p(z)\}\]
together with the map $\pi \co A_{\tau} \rightarrow \mathbb C$ given by projecting onto the third coordinate. The symplectic manifold $\mathcal Y_n$ used to define $\Kh_{symp}(L)$ is a subset of $\Hilb^n(A_{\tau})$. The action of $O(2; \mathbb C)$ on $(u,v) \in \mathbb C^2$ induces a symplectic action of $O(2; \mathbb R)$ on $(\mathcal Y_n, \eL_0, \eL_1)$ the symplectic manifold and Lagrangians used in defining symplectic Khovanov homology. (The restriction to $O(2;\mathbb R)$ is necessary to preserve the Lagrangians). This action preserves the complement of the annular divisor, and therefore also induces an action on the symplectic manifold and (unchanged) Lagrangians used in defining our annular symplectic Khovanov homology.

The spectral sequences discussed in this section use a fixed point localization theorem for $\mathbb Z/2\mathbb Z$ actions on Lagrangian Floer cohomology due to Seidel and Smith \cite{SS10}. Given an exact symplectic manifold $M$ containing exact Lagrangians $\eL_0$ and $\eL_1$, and a symplectic involution $\iota \co M \rightarrow M$ fixing the Lagrangians setwise, under restrictive algebraic-topological conditions they show that there is a spectral sequence whose $E_1$ page is isomorphic to
\[ HF(M, \eL_0, \eL_1) \otimes \FF_2[\theta, \theta^{-1}]\]
where $\theta$ is a degree one generator of $H^*(\mathbb{RP}^{\infty})$ and  $\FF[\theta, \theta^{-1}] = \theta^{-1}H_*(\mathbb{RP}^{\infty})$, and whose $E_{\infty}$ page is isomorphic to
\[ HF(M^{fix}, \eL_0^{fix}, \eL_1^{fix}) \otimes \FF_2[\theta, \theta^{-1}]\]
where $(M^{fix}, \eL_0^{fix}, \eL_1^{fix})$ denote the fixed sets under $\iota$, which form a new exact symplectic manifold with exact Lagrangians if all three are connected. Their theory was later generalized by \cite{Large}. The spectral sequences in symplectic and annular symplectic Khovanov homology which follow all arise from Seidel-Smith localization; we specify the involutions as we go along, and leave precise discussion of the fixed sets and the hypotheses for the body of the paper.

\subsubsection{Intrinsic symmetries} As discussed above, Seidel and Smith \cite{SS10} established an analog of the Ozsv{\'a}th-Szab{\'o} spectral sequence from the Khovanov homology of a link $L$ to the Heegaard Floer homology of the branched double cover of (the mirror of) the link \cite{OS:branched} in symplectic Khovanov homology, and used it to prove that
\[\dim(\Kh_{symp}(L; \FF_2)) \geq 2\dim(\widehat{HF}(\Sigma(L);\FF_2)).\]
This spectral sequence uses the action induced by $(u,v,z) \mapsto (u,-v,z)$ on $\mathcal Y_n \subset \Hilb^n(A_{\tau})$. The target of their spectral sequence is the Lagrangian Floer cohomology of a triple $(M \setminus \nabla, \eL_0, \eL_1)$ where $(M, \eL_0, \eL_1)$ is a suitable symplectic manifold and Lagrangians for computing Heegaard Floer homology, and $\nabla$ is the \emph{anti-diagonal} divisor. For more on this divisor, see Section \ref{sec:knotfloer}.

We prove the following refinement in the annular theory, again using the action $(u,v,z) \mapsto (u,-v,z)$. In the discussion below, if $K$ is a nullhomologous knot in a three-manifold $Y$, then $\widehat{HFK}(Y,K)$ denotes the Heegaard Floer knot homology of the pair, due to Ozsv{\'a}th-Szab{\'o} and J. Rasmussen, whose definition is reviewed in Section~\ref{sec:knotfloer}, and similarly if $L$ is a nullhomologous link in $Y$ for $\widehat{HFL}(Y,L)$. 

\begin{theorem} \label{thm:knotfloer} Let $L$ be an annular link in $S^3$, and $A$ its annular axis. If the linking number of the annular axis with $L$ is odd, such that the lift $\widetilde{A}$ of $A$ in the double branched cover $\Sigma(L)$ is a knot, there is a dimension inequality
\[\dim (\AKh_{symp}(L; \FF_2)) \geq 2\dim (\widehat{HFK}(\Sigma(L), \widetilde{A}; \FF_2))\]
and if the linking number of the annular axis with the link is even, such that $\widetilde{A}$ is a two-component link, there is a dimension inequality
\[\dim (\AKh_{symp}(L; \FF_2)) \geq \dim (\widehat{HFL}(\Sigma(L), \widetilde{A}; \FF_2)).\]
 \end{theorem}

This is the analog of a spectral sequence in (annular) Khovanov homology first described by L. Roberts \cite{Roberts:knotfloer} and further investigated by E. Grigsby and S. Wehrli \cite{GW:sutured}. As with Seidel and Smith's spectral sequence, the target of our spectral sequence is the Lagrangian Floer cohomology of a triple $(M \backslash \nabla_o, \eL_0, \eL_1)$, where $(M, \eL_0, \eL_1)$ are suitable manifolds for defining a version of link Floer homology, and $\nabla_o$ is again an anti-diagonal divisor. Note that these spectral sequences in fact relate the symplectic Khovanov homologies of the mirror $mL$ of the link $L$ with the Heegaard Floer homologies of the branched double cover $\Sigma(L)$ of the link, as in the Ozsv{\'a}th-Szab{\'o} spectral sequence, but one nevertheless has the rank inequalities above.

Theorem \ref{thm:knotfloer} has the following special case. Recall that any two-bridge knot is in particular two-periodic. If the annular axis $A$ is chosen to be the axis of periodicity, then the lifts $\widetilde{A}$ and $\widetilde{K}$ are isotopic in $\Sigma(K)$. (For a quick proof of this well-known fact, see Lemma~\ref{lemma:interchange}). We obtain the following.

\begin{corollary} \label{cor:two-bridge} Let $K$ be a two-bridge knot with annular axis $A$ given by the fixed set of its two-periodic symmetry. Then there is a spectral sequence with $E_1$ page isomorphic to
\[ \AKh_{symp}(mK; \FF_2) \otimes \FF_2[\theta, \theta^{-1}] \]
and with $E_{\infty}$ page isomorphic to
\[ \widehat{HFK}(\Sigma(K), \widetilde{K}; \FF_2) \otimes V \otimes \FF_2[\theta, \theta^{-1}],\]
where $\widetilde{K}$ denotes the lift of the knot in the branched double cover and $V$ is a two-dimensional vector space. It follows that there is a dimension inequality
\[\dim(\AKh_{symp}(K); \FF_2) \geq 2\dim(\widehat{HFK}(\Sigma(K), \widetilde{K}; \FF_2)).\]
\end{corollary}

This is the analog of a result of Tweedy for two-bridge knots in the context of Seidel and Smith's original spectral sequence \cite[Theorem 1.0.14]{Tweedy:anti-diagonal-reduced}; for more detail, see Section~\ref{sec:knotfloer}.

\subsubsection{Extrinsic symmetries} We now turn to interactions of our theory with various extrinsic symmetries of knots and links. Recall that a link $L\subseteq S^3$ is said to be \emph{two-periodic} if it is preserved setwise by an orientation-preserving $\mathbb Z/ 2\mathbb Z$-action on $S^3$ whose fixed set is an unknotted axis $A$ disjoint from the knot. A two-periodic link has a symmetric bridge diagram on $\mathbb C$ which is preserved by the action $z\mapsto -z$. In 2010 Seidel and Smith \cite{SS10} proved there is a spectral sequence
\begin{equation} \label{eq:ss} \Kh_{symp}(L; \FF_2) \otimes \mathbb F_2[\theta, \theta^{-1}] \rightrightarrows \Kh_{symp}(\overline{L}; \FF_2)\otimes \mathbb F_2[\theta, \theta^{-1}].\end{equation}

\noindent This spectral sequence uses the action induced by $(u,v,z) \mapsto (u,v,-z)$ on $\mathcal Y_n \subseteq \Hilb^n(A_{\tau})$ for a symmetric bridge diagram for $L$. There is no analog of this spectral sequence in combinatorial Khovanov homology; indeed, Seidel and Cornish \cite[Remark 2.5.2]{Cornish} independently observed that no such analog can preserve the quantum grading. We elaborate on Seidel's unpublished example in Section~\ref{subsec:Hopf}. However, in 2018 Stoffregen and Zhang proved the existence of a remarkable spectral sequence from the ordinary Khovanov homology of a periodic link to the \textit{annular} Khovanov homology \cite{SZ:localization}. To wit, they show that for a two-periodic link $L$ with quotient $\overline{L}$ there is a spectral sequence
\begin{equation} \label{eq:sz}
\Kh(L; \FF_2) \otimes \mathbb F_2[\theta, \theta^{-1}] \rightrightarrows \AKh(\overline{L}; \FF_2)\otimes \mathbb F_2[\theta, \theta^{-1}].
\end{equation}

\noindent This spectral sequence splits along the quantum grading in an appropriate sense. More generally, they show that a suitable version of the above spectral sequence holds for $p$-periodic links. The proof uses classical Smith theory on the Khovanov homotopy type of Lipshitz and Sarkar \cite{LS:homotopy, LLS:cube, LLS:products} (see also \cite{HKK:Khovanov}). Constructions of equivariant Khovanov homotopy types have also been done by Borodzik, Politarczyk, and Silvero \cite{BPS}, whose construction recovers the above spectral sequence using significantly different methods \cite[Theorems 1.2 and 1.3]{BPS}, and by Musyt \cite{Musyt:thesis}.

The apparent difference between the equivariant theories with respect to the action induced by a doubly-periodic link on the symplectic and combinatorial side has occasioned speculation that combinatorial and symplectic Khovanov homology might fail to agree over $\mathbb F_2$ coefficients. In this paper we show that we can in fact recover the analog of the Stoffregen-Zhang spectral sequence on the symplectic side, as follows, using the action induced by $(u,v,z) \mapsto (-u,-v,-z)$ on $\mathcal Y_n \subset \Hilb^n(A_{\tau})$ for a symmetric bridge diagram for $L$.

\begin{theorem}\label{thm:2periodic}
 Let $L$ be a 2-periodic link with quotient link $\overline{L}$. There is a spectral sequence whose $E_1$ page is isomorphic to
\[ \Kh_{symp}(L; \FF_2) \otimes \mathbb F_2[\theta, \theta^{-1}]\]
and whose $E_{\infty}$ page is isomorphic to \[\AKh_{symp}(\overline{L}; \FF_2)\otimes \mathbb F_2[\theta, \theta^{-1}].\]
\noindent It follows that 
\[ \dim \Kh_{symp}(L; \FF_2) \geq \dim \AKh_{symp}(\overline{L}; \FF_2).\] \end{theorem}

Most of the difficulty of proving Theorem ~\ref{thm:2periodic} lies in identifying the Floer cohomology of the fixed set. In the case of the involution $(u,v,z) \mapsto (u,v,-z)$ and the spectral sequence \eqref{eq:ss}, this step is straightforward; it is easy to see that the fixed symplectic manifold and fixed Lagrangians are those used in computing the symplectic Khovanov homology of the quotient link. In the case of the involution $(u,v,z) \mapsto (-u,-v,-z)$ and the spectral sequence of Theorem~\ref{thm:2periodic}, there is no such immediate identification, and we must analyze the geometry of the situation more carefully. We carry out this analysis of the fixed set in Section~\ref{ss:quotient}.

The reader may wonder why we recover only the analog of the spectral sequence (\ref{eq:sz}) for 2-periodic links, and not the more general results of \cite{SZ:localization} for $p$-periodic links. This is because Seidel and Smith's fixed point localization theorem for Lagrangian Floer cohomology \cite{SS10, Large} is presently only known for actions of the group $\mathbb Z/2\mathbb Z$. We expect that the techniques of Rezchikov's recent work for Hamiltonian Floer theory \cite{Rezchikov}, which recover a parallel localization result for fixed point Floer cohomology, could be combined with a Floer homotopy type for Lagrangian Floer cohomology to recover a $\mathbb Z/p\mathbb Z$-localization theorem for Lagrangian Floer cohomology. In that case, one would begin with a bridge diagram for a $p$-periodic link $L$ preserved under $z \mapsto \xi z$, where $\xi$ is a primitive $p$th root of unity. Using the change of coordinates 
\[x=u-iv \qquad \qquad  y=u+iv \qquad \qquad xy = u^2+v^2 = p(z),\] 
one would then consider the action $\mathcal Y_n$ induced by the action $(x,y,z) \mapsto (\xi x, \xi^{-1} y, \xi z)$ on $A_{\tau}$ in order to obtain the following conjecture. Recall that $H^*(B\mathbb Z_p; \mathbb F_p)$ is the ring $\mathbb F_p[\alpha, \theta]/(\alpha^2)$ where $\deg(\alpha) = 1$ and $\deg(\theta) = 2$.
\begin{conjecture} Let $L$ be a $p$-periodic link for $p$ an odd prime and $\overline{L}$ its quotient link, treated as an annular link with the axis of symmetry as the annular axis. There is a spectral sequence with $E_1$ page isomorphic to
\[ \Kh_{symp}(L; \mathbb F_p) \otimes \theta^{-1}H^*(B\mathbb Z_p; \mathbb F_p) \]
and $E_{\infty}$ page isomorphic to 
\[ \AKh_{symp}(\overline{L}; \mathbb F_p) \otimes \theta^{-1}H^*(B\mathbb Z_p; \mathbb F_p). \]
\end{conjecture}

One would also immediately recover an analog of Seidel and Smith's original spectral sequence (\ref{eq:ss}) for $p$-periodic links using the action induced by $(u,v,z) \mapsto (u,v, \xi z)$.

We also consider the case of \emph{strongly invertible} knots. A knot $K \subseteq S^3$ is said to be \emph{strongly invertible} if it is preserved setwise by an orientation-preserving $\mathbb Z/2\mathbb Z$ action of $S^3$ fixing an unknotted axis $A$ which reverses the orientation of the knot; the intersection of the knot with the fixed axis consists of two points. Given a strongly invertible knot, we may associate to it a choice of half-axis, which is to say one of the two arcs into which $A$ is divided by its intersection with the knot. Given a choice of half-axis, one may form a \emph{quotient knot} $\overline{K}$ which is the union of the quotient of $K$ under the action together with the half-axis. In general, the two quotient knots associated to a strong inversion on a knot are non-isotopic. Recently, following a similar strategy to Stoffregen and Zhang, Lipshitz and Sarkar used the Khovanov homotopy type to construct a localization spectral sequence for the Khovanov homology of a strongly invertible knot \cite{LS:strongly_invertible}. For a strongly invertible knot $K$, the $E_1$ page of their spectral sequence is isomorphic to $Kh(K; \FF_2) \otimes \FF_2[\theta, \theta^{-1}]$. The final page is somewhat more intricate, as follows: in the language introduced by Boyle and Issa \cite[Definition 3.3]{BI:equivariant-genera}, consider an \emph{intravergent} diagram for the strongly invertible knot, that is, a diagram such that the axis of rotation is orthogonal to the page and passes through a crossing at the center of the diagram. An intravergent diagram naturally distinguishes a preferred half-axis by choosing the short segment connecting the two strands of the crossing passing through the plane of the diagram. Moreover, the two possible resolutions of the center crossing induce two presentations of the quotient knot as an annular knot, which we will denote $\overline{K}_0$ and $\overline{K}_1$, both of whose underlying non-annular knot is $\overline{K}$. We describe these annular resolutions in more detail in Section~\ref{sec:symmetry}. The annular Khovanov homologies of these annular knots are related by an \emph{axis-moving map}, here denoted $f$. The $E_{\infty}$ page of the Lipshitz-Sarkar spectral sequence is 
\[ \Cone \left(\AKh(\overline{K}_1; \FF_2) \xrightarrow{f} \AKh(\overline{K}_0; \FF_2)\right) \otimes \mathbb F_2[\theta, \theta^{-1}].\]

A strongly invertible knot has a symmetric bridge diagram on $\mathbb C$ fixed by the action $z \mapsto -z$, with an odd number of bridges, one of which is sent to itself by the action. The analog of the axis-moving map in symplectic annular Khovanov homology is the $U=0$ truncation of the filtered continuation map associated to a Hamiltonian isotopy over the annular divisor. Because the Lipshitz-Sarkar conventions for Khovanov homology are dual to the standard conventions used in this paper \cite[Section 2]{LS:strongly_invertible}, our axis-moving map goes from $\AKh_{symp}(\overline{K}_0; \FF_2) \xrightarrow{f} \AKh_{symp}(\overline{K}_1; \FF_2)$. We prove the existence of the following spectral sequence, using the action induced by $(u,v,z) \mapsto (u,v,-z)$ on $\mathcal Y_n \subset \Hilb^n(A_{\tau})$ for a symmetric bridge diagram for $K$.

\begin{theorem}\label{thm:stronglyinvertible} There is a spectral sequence whose $E_1$ page is isomorphic to 
\[\Kh_{symp}(K; \FF_2) \otimes \mathbb F_2[\theta, \theta^{-1}]\]
and whose $E_{\infty}$ page is isomorphic to
\[\mathrm{Cone}\left(\AKh_{symp}(\overline{K}_0; \FF_2) \xrightarrow{f} \AKh_{symp}(\overline{K}_1; \FF_2)\right)\otimes \mathbb F_2[\theta, \theta^{-1}]\]
\noindent along with a corresponding dimension inequality. 
\end{theorem}

Indeed, in our case we are able to show that the kernel and cokernel of $f$ are canonically chain homotopy equivalent, and therefore the homology of the mapping cone decomposes as a direct sum of the homology of two isomorphic chain complexes offset by a grading shift. As far as we are aware this result is not known for combinatorial Khovanov homology (see Remark \ref{r:new}).

In some respects, our spectral sequence is somewhat simpler to construct than the Lipshitz-Sarkar sequence, as it uses properties of the continuation map associated to an isotopy in Lagrangian Floer cohomology which are well-understood. However, the identification of the fixed set as a mapping cone requires some delicate curve-counting, carried out in Section \ref{s:strongly}.

\begin{remark} \label{rem:asymmetry} The reader familiar with equivariant Khovanov homology may find the difference in involutions used in constructing our analogs of the Stoffregen-Zhang and Lipshitz-Sarkar spectral sequences surprising, as the constructions of the involutions on the Khovanov homotopy type are very similar. This is an artifact of the fact that a symmetric diagram for a 2-periodic knot necessarily has an even number of bridges and a symmetric diagram for a strongly invertible knot necessarily has an odd number of bridges; if one writes the involutions in terms of the coordinates used in Seidel and Smith's original formulation of symplectic Khovanov homology, the two maps more clearly agree. We discuss this coordinate system in the Appendix.
\end{remark}

\begin{remark} We expect that it should be straightforward to use the techniques of \cite{HLS:flexible} to show that the spectral sequence of Theorems~\ref{thm:knotfloer}, \ref{thm:2periodic}, and \ref{thm:stronglyinvertible} are invariants of the underlying knot or link, in the second two cases together with the data of the extrinsic symmetry.
\end{remark}

\subsubsection{A summary table}

The three involutions discussed above will soon acquire the following names.
\[ \iota_D(u,v,z) = (u,-v,z) \qquad \qquad \iota_A(u,v,z)=(u,v,-z) \qquad \qquad \iota_B(u,v,z)=(-u,-v,-z).\]
The motivation behind this notation is as follows: $A$ and $B$ are a pair of options which exist for any two-periodic link or strongly-invertible knot. The choice of $D$ for the remaining action, which exists for an arbitrary link, is meant to emphasize its connection to the double branched cover.

With this in mind, we collect the spectral sequences discussed in this paper into the following table. The term $\FF_2[\theta, \theta^{-1}]$ is omitted throughout. If $H(X)$ is the homology of some chain complex, we let $gH(X)$ indicate the homology of the associated graded of some filtration on the complex, in both cases below the appropriate version of the anti-diagonal filtration. For the cases where Heegaard Floer homology appears, we omit possible factors of two-dimensional vector spaces. The abbreviation ``s.i.'' stands for strongly invertible. More precise references to the combinatorial Khovanov analogs of our spectral sequences are given in the discussion surrounding the relevant theorems. Finally, for discussion of the unusual-seeming entry in the fourth row, see Remark~\ref{rem:either}.

\begin{center}
\resizebox{\textwidth}{!}{%
\def\arraystretch{1.5}
\begin{tabular}{|c|c|c|c|c|c|} \hline
Involution & Symmetry & First Page & Final Page & Reference & Analog in Kh \\ \hline
$\iota_D$ & - & $\Kh_{symp}(mL)$ & $g\HFhat(\Sigma(L))$ &  \eqref{eq:ssdbc}; \cite{SS10} & \eqref{eq:ozsz}; \cite{OS:branched} \\ \hline
$\iota_D$ & - & $\AKh_{symp}(mL)$ & $g\HFLhat(\Sigma(L), \widetilde{A})$ &  Theorem \ref{thm:knotfloer} & \cite{Roberts:knotfloer} \\ \hline
$\iota_A$ & 2-periodic & $\Kh_{symp}(L)$ & $\Kh_{symp}(\overline{L})$ & \eqref{eq:ss}; \cite{SS10} & - \\ \hline
$\iota_A$ or $\iota_B$ & 2-periodic & $\AKh_{symp}(L)$ & $\AKh_{symp}(\overline{L})$ & Proposition \ref{p:2periodicEasy} & \cite{Zhang:annular} \\ \hline
$\iota_B$ & 2-periodic & $\Kh_{symp}(L)$ & $\AKh_{symp}(\overline{L})$ & Theorem~\ref{thm:2periodic} & \eqref{eq:sz}; \cite{SZ:localization, BPS} \\ \hline 
$\iota_A$ & s.i. & $\Kh_{symp}(K)$ & $Cone(\AKh_{symp}(\overline{K}_0 \rightarrow \overline{K}_1))$ & Theorem~\ref{thm:stronglyinvertible} & \cite{LS:strongly_invertible} \\ \hline
\end{tabular}}

\end{center}

\subsection*{Organization} This paper is organized as follows. In Section \ref{sec:background} we introduce certain necessary background topics; specifically, Section \ref{sec:braid} reviews the theory of annular braids, Section \ref{subsec:knotfloer} contains a brief rundown of Heegaard Floer homology for three-manifolds and knots within them, Section \ref{sec:localization} reviews Seidel and Smith's localization theorem for $\mathbb Z / 2 \mathbb Z$-equivariant Lagrangian Floer cohomology, and Section \ref{sec:symmetry} briefly recalls relevant aspects of knot symmetry. In Section \ref{s:SympAKh} we introduce our new annular symplectic Khovanov homology. Specifically, in Section~\ref{sec:geometry} we discuss the geometric inputs of the theory, and in Section~\ref{sec:floer} we touch on some necessary aspects of Floer theory and define our $\AKh_{symp}$. In Section~\ref{sec:invariance} we prove invariance of the theory and discuss its basic properties, including defining an absolute winding grading and proving Theorem~\ref{thm:akh-to-kh}. In Section~\ref{ss:quotient} we introduce two additional link invariants which are useful for studying equivariance. Finally, in Section~\ref{subsec:bridge} we review important aspects of symplectic Khovanov homology for bridge diagrams and their application to our setting. We then proceed to investigate the relationship of our theory to various order-two actions. In particular, we set up and prove Theorems \ref{thm:knotfloer}, \ref{thm:2periodic} and \ref{thm:stronglyinvertible} in Sections \ref{sec:knotfloer}, \ref{s:2periodic} and \ref{s:strongly} respectively, in all cases modulo checking the existence of stable normal trivializations, which we carry out for all of the relevant manifolds in Section \ref{sec:normal-trivs}. Note that Section~\ref{s:2periodic} concludes with discussing the example of the periodic Hopf link in Section~\ref{subsec:Hopf}, and Section~\ref{s:strongly} concludes in Section~\ref{subsec:examples} with several computations of annular symplectic Khovanov homology and the mapping cone of the axis-moving map from Theorem~\ref{thm:stronglyinvertible}. Finally, in the Appendix we compare the actions on symplectic Khovanov homology considered in this paper as they appear on the coordinates of Seidel and Smith's original formulation of the theory and Manolescu's reformulation, with (inter alia) the goal of explaining the apparent asymmetry in the constructions of this paper.

\subsection*{Conventions} For the reader's convenience, we gather here the following conventions. 

\begin{itemize}
\item We use the ``cohomological'' convention for Lagrangian Floer theory, counting strips as described in Equation \eqref{eq:strips} in Section~\ref{sec:floer}. The reader familiar with Heegaard Floer theory, which ordinarily uses the opposite convention, may wish to skim Section \ref{subsec:knotfloer} for the implications of this convention.

\item Throughout, the terms ``ordinary'' and ``annular'' are antonyms, and the terms ``combinatorial'' and ``symplectic'' are antonyms. 

\item For notational simplicity, we refer to the version of annular symplectic Khovanov homology introduced in this paper as $\AKh_{symp}(L)$ and the original Mak-Smith formulation as $\AKh^{HH}_{symp}(L)$, referencing its roots in Hochschild homology.

\item Throughout, $mK$ denotes the mirror of a knot $K$, and $\overline{K}$ indicates the quotient of a knot $K$ under some action, and similarly for links.

\item We reserve $L$ to denote links, whereas Lagrangian submanifolds are denoted by $\eL$ if not otherwise named.

\item From this point forward, most theories are taken with coefficients in $\mathbb F_2$, and this field is suppressed from the notation where it is used.
\end{itemize}

\subsection*{Acknowledgments} We thank Mohammed Abouzaid, Mina Aganagic, Keegan Boyle, Robert Lipshitz, Aakash Parikh, Ivan Smith, Matt Stoffregen, Catharina Stroppel, Ben Webster, Melissa Zhang, and Peng Zhou for helpful conversations, in some cases over many years. This work was inspired by discussions at an AIM Workshop on ``Floer theory of symmetric products and Hilbert schemes'' in December 2022 (of which the first two authors were co-organizers), and was partially carried out at the Clay Mathematics Institute during the conference ``Gauge Theory and Topology: In Celebration of Peter Kronheimer’s 60th Birthday'' in July 2023. We are grateful to both institutions for their hospitality. Finally, we thank the anonymous referees for their careful reading and helpful comments.

\section{Background} \label{sec:background}

In this section we review some background topics useful to our constructions.

\subsection{Annular braids and the Markov theorem} \label{sec:braid}

Recall from the introduction that an \emph{annular link} $L$ is a link together with an unknot component linked with $L$, the \emph{annular axis}. Equivalently, if $H=(\mathbb{R}^2 \setminus \{0\}) \times \mathbb{R}$, an \emph{annular link} $L$ is a smooth embedding of a disjoint union of circles into $H$. One has analogs of many of the basic theorems of knot theory in $S^3$ for annular links. We here review certain of these theorems which we will need.

Let $\vec{n} = \{1,2, \dots, n\} \subseteq \mathbb R$. An \emph{annular $n$-strand braid} in $(\mathbb{R}^2 \setminus \{0\}) \times [0,1]$ is a smooth embedding \[f \co \coprod_{i=1}^n [0,1] \to (\mathbb{R}^2 \setminus \{0\}) \times [0,1]\]  such that the image of $f$ contains $\vec{n} \times \{0,1\}$, and after post-composing $f$ with projection $p_2$ to the second factor, the map \[p_2 \circ f \co \coprod_{i=1}^n [0,1] \to [0,1]\] is a covering map. Given this set-up, we say that the annular $n$-strand braid starts at $\vec{n} \times \{0\}$ and ends at $\vec{n} \times \{1\}$.
An \emph{isotopy} of an annular $n$-strand braid indicates an isotopy relative to the endpoints $\vec{n} \times \{0,1\}$. An \emph{annular braid} is an annular $n$-strand braid for some $n \in \mathbb{N}$.

As in the case of non-annular braids, $n$-strand annular braids have an obvious composition operation via stacking, which preserves isotopy equivalence. If $\Br_{1,n} =\pi_1(\Conf^n(\mathbb{R}^2 \setminus \{0\}),\vec{n})$ is the fundamental group of the configuration space of $n$ unordered distinct points in $\mathbb{R}^2 \setminus \{0\}$ based at the point $\vec{n}$, then
elements $\beta \in \Br_{1,n}$ are in natural bijective correspondence with isotopy classes of $n$-strand braids in $(\mathbb{R}^2 \setminus \{0\}) \times [0,1]$ which start at $\vec{n} \times \{0\}$ and end at $\vec{n} \times \{1\}$, and the composition laws also agree. We henceforth treat these two groups as interchangeable.

The closure of an annular braid is the annular link, well-defined up to isotopy, obtained by connecting $\{i\} \times \{0\}$ with $\{i\} \times \{1\}$ for all $i \in \vec{n}$ in the standard way depicted in Figure \ref{fig:closure}.
For any $\beta \in \Br_{1,n}$, we use $\cl(\beta)$ to denote the isotopy class of the annular link obtained by taking closure of the annular braid corresponding to $\beta$.

\begin{figure}[ht]
\scalebox{.5}{
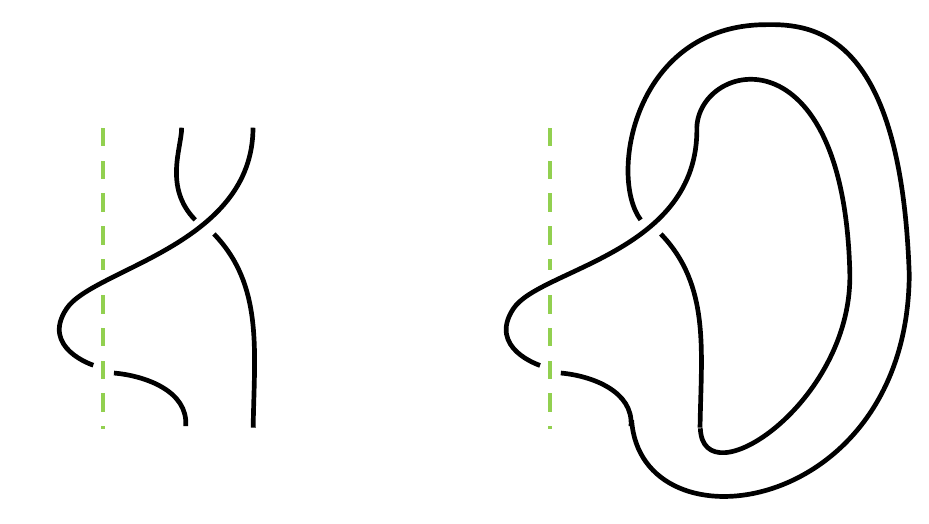}
\caption{Left: A $2$-strand annular braid $\beta$, with strands drawn in black, and the axis $\{0\} \times \mathbb{R}$, drawn dashed and in green. Right: The closure $\cl(\beta)$ of $\beta$.}
\label{fig:closure}
\end{figure}

One has the following analog of the Alexander theorem for the case of annular links.

\begin{theorem}\cite[Theorem 2]{HOL} \label{t:Alex}
Every annular link is isotopic to the closure of an annular braid.
\end{theorem}

A complete description of $\Br_{1,n}$ in terms of generators and relations, similar to those of the non-annular braid group, can be found in \cite[p.13]{HOL}. Among the generators, the ones that are most relevant to us are $\{\sigma_i | i \in \{1,\dots,n-1\}\}$, where $\sigma_i$, as in the non-annular case, is the (annular) braid with a single crossing in which the the $i^{th}$ strand overcrosses the $(i+1)^{st}$ strand, as depicted in Figure \ref{fig:simple}.

\begin{figure}[ht]
\scalebox{.5}{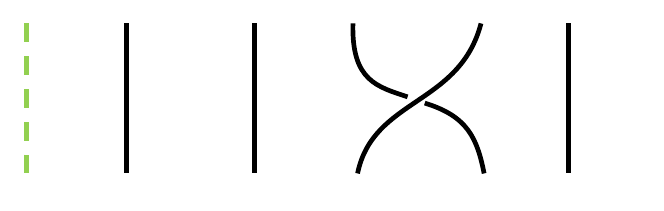}
\caption{The generator $\sigma_3$ in $\Br_{1,5}$}
\label{fig:simple}
\end{figure}

With this notation in mind, we may recall the analog of the Markov theorem for annular braids. Below, composition is written left to right, running upward along the annular braid, and in the same order it would be in the fundamental group.

\begin{theorem}\cite[Theorem 5]{HOL}\label{t:Markov}
Two annular braid closures are isotopic to each other if and only if the annular braids are related by a sequence of the following equivalences
\begin{itemize}
\item (Markov conjugation) $\sigma_j \beta \sigma_j^{-1} \sim \beta$ for any $\beta \in \Br_{1,n}$.
\item (Markov stablization) $ \beta_1 \sigma_n^{\pm 1}\beta_2 \sim \beta_1 \beta_2$ for any $\beta_1, \beta_2 \in \Br_{1,n} \subseteq \Br_{1,n+1}$, where the subset relation is via inclusion on the first $n$ strands.
\end{itemize}
\end{theorem}

As a consequence of Theorem \ref{t:Markov}, any invariant of annular braids that is furthermore invariant under Markov conjugation and Markov stablization defines an invariant of annular links.

\subsection{Heegaard Floer homology basics} \label{subsec:knotfloer}

Heegaard Floer homology is an invariant of 3-manifolds and knots and links within them introduced by Ozsv{\'a}th and Szab{\'o}  in the early 2000s \cite{OS_Disks, OS:properties, OS:knots}, and in the knot case independently by J.~Rasmussen \cite{Rasmussen:knots}. In this paper we will require the simplest version of both the 3-manifold and knot invariants, which we now briefly recall. In the 3-manifold case, the theory associates to a 3-manifold $Y$ an $\mathbb F_2$-chain complex $\CFhat(Y)$ and its homology $\HFhat(Y)$. If $Y$ is a rational homology sphere, the complex and its homology admit an absolute homological grading. In the knot case, the theory associates to a nullhomologous knot $K\subseteq Y$ an $\mathbb F_2$-chain complex $\CFKhat(Y,K)$ and its homology $\HFKhat(Y,K)$. As in the 3-manifold case, if $Y$ is a rational homology sphere, this complex admits an absolute homological grading; there is also an additional Alexander grading, which is well-defined if $Y$ is an integer homology sphere and well-defined up to a choice of Seifert surface if $Y$ is a rational homology sphere.

Producing these chain complexes involves taking Lagrangian Floer cohomology in a choice of auxiliary symplectic manifold, the construction of which we now explain. A \emph{multi-pointed Heegaard diagram} for $Y$ consists of

\begin{itemize}
\item A surface $\Sigma$ of genus $g$.
\item Two sets of attaching curves $\alphas = \{\alpha_1, \dots, \alpha_{g+n-1}\}$ and $\betas = \{\beta_1, \dots, \beta_{g+n-1}\}$ for handlebodies $\mathcal H_{\alpha}$ and $\mathcal H_{\beta}$ such that the union  $\mathcal H_{\alpha} \cup_{\Sigma} \mathcal H_{\beta}$ is the 3-manifold Y. We further impose the condition that $\alpha_i$ and $\beta_j$ intersect transversely.
\item A set of basepoints ${\bf w}= \{w_1,\dots, w_n\}$ on $\Sigma \setminus \{\boldsymbol{\alpha} \cup \boldsymbol{\beta}\}$ with the property that every component of $\Sigma \setminus \alphas$ contains a single basepoint, and likewise for $\Sigma \setminus \betas$.
\end{itemize}

One may obtain such data by choosing a self-indexing Morse function $f \co Y \rightarrow [0,3]$. Then $\Sigma = f^{-1}\left( \frac{3}{2} \right)$ is a surface of genus $g$. If there are $n$ critical points each of index zero and three, there must be $g+n-1$ critical points of index one and two respectively. The curves $\alphas$ are then the intersections of the ascending manifolds of the critical points of index one with $\Sigma$, and the curves $\betas$ are the intersections of the descending manifolds of the critical points of index two with $\Sigma$. The basepoints $\bf{w}$ are the intersections with the surface of a set of flowlines of $f$ connecting the index three and index zero critical points of $f$.

We further require that the diagram be \emph{weakly admissible}. Consider the closure of $\Sigma \setminus \{ \alphas \cup \betas\}$, which consists of a set of oriented surfaces with boundary $D_1, \dots, D_n$. Let $D_1, \dots, D_k$ be those not containing any basepoint. A Heegaard diagram is said to be weakly admissible if any 2-chain on the Heegaard surface consisting of a sum of the surfaces $D_1, \dots, D_k$ whose boundary consists of either a union of $\alpha$ curves or a union of $\beta$ curves has both positive and negative multiplicities.

The  simplest version of the Heegaard Floer homology of $Y$ is defined as follows. We form the symmetric product $\Sym^{g+n-1}(\Sigma \setminus {\bf w}) = (\Sigma \setminus {\bf w})^{\times (g+n-1)}/S_{g+n-1}$, where the quotient is via the action of the symmetric group on the product. This is a smooth manifold; moreover, any choice of complex structure on $\Sigma$ induces a complex structure on the symmetric product. Furthermore, this symmetric product contains two half-dimensional totally real submanifolds $\mathbb T_{\alpha} = \alpha_1 \times \dots \alpha_{g+n-1}$ and $\mathbb T_{\beta} = \beta_1 \times \dots \times \beta_{g+n-1}$ which intersect transversely. Given a choice of exact symplectic form on $\Sigma \setminus \{\bf w\}$, Perutz has shown that the symmetric product admits a symplectic form which agrees with the product symplectic form outside a neighborhood of the diagonal \cite{Perutz:handleslide}. In particular, the two tori $\mathbb T_{\alpha}$ and $\mathbb T_{\beta}$ are Lagrangian with respect to this symplectic form. Weak admissibility of the Heegaard diagram implies that $\Sym^{g+n-1}(\Sigma \setminus \{\mathbf{w}\}, \mathbb T_{\alpha}, \mathbb T_{\beta})$ may be taken to be exact and such that the Lagrangians $\mathbb T_{\alpha}$ and $\mathbb T_{\beta}$ are themselves exact. (This fact appears implicitly in \cite[Lemma 4.12]{OS_Disks}; for more detail, the reader may consult \cite[Proposition 4.2]{HLL:rank}.) Given this, using the conventions of Section \ref{sec:floer}, the Heegaard Floer homology of $Y$ is specified by
\[ HF(\Sym^{g+n-1}(\Sigma \setminus \{{\bf w}\})), \mathbb T_{\beta}, \mathbb T_{\alpha}) \simeq \HFhat(Y) \otimes V^{\otimes {(n-1)}}.\]
Here $V$ is a two-dimensional vector space over $\mathbb F_2$ which has generators in homological gradings $0$ and $-1$. This construction is invariant of the choice of symplectic form and complex structure on $\Sigma \setminus\{\bf w\}$, and invariant of the choice of Heegaard diagram for $Y$ except for the dependence on the number $n$ of pairs of basepoints. Correspondingly, the Heegaard Floer cohomology of $Y$ is specified by the homology of the dual chain complex
\[ HF(\Sym^{g+n-1}(\Sigma \setminus \{{\bf w}\})), \mathbb T_{\alpha}, \mathbb T_{\beta}) \simeq \HFhatDual(Y) \otimes H^*(S^1)^{\otimes (n-1)}.\]
Here $H^*(S^1)$ may equivalently be replaced with a vector space $V^*$ of dimension two over $\mathbb F_2$ with generators having homological gradings in degrees $0$ and $1$. Note that, as we are working over a field, $\HFhatDual(Y) \simeq \HFhat(Y)$ non-canonically; the isomorphism reverses the sign of the homological grading. Perhaps more helpfully, there is a canonical isomorphism $\HFhatDual(Y) \simeq \HFhat(-Y)$, where $-Y$ denotes the manifold with its orientation reversed \cite[Section 5]{OS:triangles}.

Given a pseudoholomorphic strip $u \co \RR \times [0,1] \rightarrow \Sym^{g+n-1}(\Sigma)$ counted in the Lagrangian Floer cohomology differential, we may choose a point $x_i \in D_i$ for each of the closures of the regions of $\Sigma \setminus \{\alphas \cup \betas\}$ not containing a basepoint, as defined above. Let $V_{x_i}$ be the holomorphic divisor $\{x_i\} \times \Sym^n(\Sigma)$. Then if $m_i = u \cdot V_{x_i}$ is the algebraic intersection number, which is necessarily positive, we say the \emph{shadow} of the strip is $D = \sum m_i D_i$. The coefficient of a domain containing a basepoint in the shadow is necessarily zero.

We now turn our attention to the knot theory. Given a nullhomologous knot $K$ in a three-manifold $Y$, a multi-pointed Heegaard diagram for $(Y,K)$ consists of data $H = (\Sigma, \alphas, \betas, {\bf w}, {\bf z})$ satisfying the following.

\begin{itemize}
\item $(\Sigma, \alphas, \betas, {\bf w})$ are the data of a multi-pointed Heegaard diagram for $Y$ as above.
\item There is an additional set of basepoints ${\bf z} = \{z_1, \dots, z_{g+n-1}\}$ with the property that each component of $\Sigma \setminus \alphas$ or $\Sigma \setminus \betas$ contains a single basepoint of each type $w$ and $z$, such that $K$ is the union of a set of arcs connecting $w$ basepoints to $z$ basepoints in $\mathcal H_{\beta}$ and a set of arcs connecting $z$ basepoints to $w$ basepoints in $H_{\alpha}$.
\end{itemize}

In the Morse theory picture, one chooses the function $f$ such that $K$ is the union of a set of flowlines connecting the index three and index zero critical points; the $z$ basepoints are the intersections of $K$ with $\Sigma$ for those arcs in which the orientation of the flowline matches the orientation of $K$, and the $w$ basepoints those for which it disagrees. For practical purposes, one may reconstruct the knot on the Heegaard surface by first connecting $w$ basepoints to $z$ basepoints in the complement of the $\beta$ curves and then $z$ to $w$ in the complement of the $\alpha$ curves, letting the second set of arcs undercross the first.

With this in mind, the knot Floer homology of $(Y,K)$ is specified by
\[ HF(\Sym^{g+n-1}(\Sigma \setminus \{{\bf w}\cup {\bf z}\}), \mathbb T_{\beta}, \mathbb T_{\alpha}) \simeq \HFKhat(Y,K) \otimes V^{\otimes (n-1)}.\]
Here $V$ is again a two-dimensional vector space over $\mathbb F_2$ with suitable gradings; in particular, if $(M,A)$ denotes the homological and Alexander gradings, then the gradings on the generators of $V$ are $(0,0)$ and $(-1,-1)$. Correspondingly, the Heegaard Floer cohomology of $(Y,K)$ is specified by the homology of the dual chain complex
\[ HF(\Sym^{g+n-1}(\Sigma \setminus \{{\bf w}\cup{\bf z}\}), \mathbb T_{\alpha}, \mathbb T_{\beta}) \simeq \HFKhatDual(Y,K) \otimes (V^*)^{\otimes (n-1)}.\]
Analogously to the 3-manifold case there is a canonical isomorphism $\HFKhatDual(Y,K) \simeq \HFKhat(-Y, mK)$, where $mK$ denotes the mirror of the knot.

The theory for links is similar. We say that $L$ is a nullhomologous link in $Y$ if its components jointly bound a surface. If $L$ has $\ell$ components, a Heegaard diagram $\mathcal H = (\Sigma, \alphas, \betas, {\bf w}, {\bf z})$ is defined as in the case of a knot except that the union of flowlines connecting the index zero and index three critical points forms the link $L$. For our purposes we may restrict ourselves to the case that $|{\bf w}| = |{\bf z}| = \ell$, so that each component of $L$ contains exactly one $w$ basepoint and one $z$ basepoint. In this case we have
\[ HF(\Sym^{g+\ell-1}(\Sigma \setminus( \{{\bf w}\cup {\bf z}\}), \mathbb T_{\beta}, \mathbb T_{\alpha}) \simeq \HFLhat(Y,L)\]
and correspondingly, the Heegaard Floer cohomology of $(Y,L)$ is specified by the homology of the dual chain complex
\[ HF(\Sym^{g+\ell-1}(\Sigma \setminus \{{\bf w} \cup {\bf z}\}), \mathbb T_{\alpha}, \mathbb T_{\beta}) \simeq \HFLhatDual(Y,L).\]
As previously there is an isomorphism $\HFLhatDual(Y,L)\simeq \HFLhat(-Y, mL)$.

\subsection{Seidel-Smith's localization theorem} \label{sec:localization}

In this subsection we introduce Seidel and Smith's localization theorem for order-two actions on Lagrangian Floer cohomology. We begin by introducing some terminology necessary to stating their hypotheses. The following two definitions are drawn from Large \cite{Large}.

\begin{definition} Let $M$ be a symplectic manifold containing Lagrangians $\eL_0$ and $\eL_1$. A \emph{set of polarization data} for $(M, \eL_0, \eL_1)$ is a triple $\mathfrak p = (E, F_0, F_1)$ such that
\begin{itemize}
\item $E$ is a symplectic vector bundle over $M$, and
\item $F_i$ is a Lagrangian subbundle of $E|_{\eL_i}$ for $i=0,1$.
\end{itemize}
\end{definition}

Given $\mathfrak p = (E, F_0, F_1)$ a set of polarization data for $(M, \eL_0, \eL_1)$, one may stabilize by a trivial bundle to obtain $\mathfrak p \oplus \underline{\CC}^N = (E \oplus  \underline{\CC}^N, F_0 \oplus \underline{\RR}^N, F_1 \oplus i \underline{\RR}^N)$.

\begin{definition} Let $\mathfrak p = (E, F_0, F_1)$ and $\mathfrak p' = (E', F_0', F_1')$ be two sets of polarization data for $(M, \eL_0, \eL_1)$. An \emph{isomorphism of polarization data} between $\mathfrak p$ and $\mathfrak p'$ is an isomorphism of symplectic vector bundles 
\[\psi \co E \rightarrow E'\]
such that there are homotopies through Lagrangian subbundles of $(E')|_{\eL_i}$ between $\psi(\eL_i)$ and $\eL_i'$ for $i=0,1$. A \emph{stable isomorphism of polarization data} between $\mathfrak p$ and $\mathfrak p'$ is an isomorphism of polarization data between $\mathfrak p \oplus \underline{\CC}^{N_1}$ and $\mathfrak p' \oplus \underline{\CC}^{N_2}$ for some $N_1$ and $N_2$.\end{definition}

We now turn our attention to the situation of involutions on Lagrangian Floer cohomology. Let $M$ be an exact symplectic manifold which is convex at infinity and let $\eL_0$ and $\eL_1$ be two exact compact Lagrangians within it. Suppose further that $M$ admits a symplectic involution $\iota$ which fixes $\eL_0$ and $\eL_1$ setwise, so that $\iota(\eL_i) = \eL_i$ for $i=0,1$. Let $(M^{fix}, \eL_0^{fix}, \eL_1^{fix})$ denote the fixed sets under the involution. In the case that either $M^{fix}$ is connected or that all of the connected components of $M^{fix}$ have the same dimension, it is easy to see that $M^{fix}$ is itself an exact symplectic manifold which is convex at infinity and $\eL_0^{fix}$ and $\eL_1^{fix}$ are two exact compact Lagrangians within it. The \emph{normal polarization} is the set of polarization data $(NM^{fix}, N\eL_0^{fix}, N\eL_1^{fix})$ consisting of the normal bundle to $M^{fix}$ inside $M$, which is a symplectic vector bundle, together with the normal bundles $N\eL_i^{fix}$ to $\eL_i^{fix}$ inside $\eL_i$ for $i=0,1$, which are Lagrangian subbundles of $(NM^{fix})|_{\eL_i}$. This brings us to the following definition of Seidel and Smith \cite{SS10}, as rephrased by Large \cite{Large}.

\begin{definition} \label{def:stable-normal} A \emph{stable normal trivialization} is a stable isomorphism of polarization data between the normal polarization $(NM^{fix}, N\eL_0^{fix}, N\eL_1^{fix})$ and the trivial polarization $(\underline{\CC}, \underline{\RR}, i\underline{\RR}).$ \end{definition}

We remind the reader that, since the symplectic group deformation retracts onto the unitary group, the theory of symplectic vector bundles is isomorphic to the theory of complex vector bundles. Therefore, the definitions above could be equivalently rephrased in terms of complex bundles with totally real subbundles.

Seidel and Smith prove the following.

\begin{theorem}\cite[Theorem 20]{SS10} \label{thm:localization}
Suppose that $(M,\eL_0,\eL_1)$ and $\iota$ satisfy the hypotheses above and $(M^{fix}, \eL_0^{fix}, \eL_1^{fix})$ admits a stable normal trivialization. Then for a generic $\mathbb{Z}/2\mathbb{Z}$-invariant $\omega$-compatible almost complex structure $J$ on $M$, after a suitable stabilization and equivariant isotopy of $\eL_0$ and $\eL_1$ which is trivial on the fixed sets, there is a localization map \[HF_{\ZZ/2\ZZ}(M, \eL_0,\eL_1) \rightarrow HF(M^{fix},\eL_0^{fix},\eL_1^{fix})[\theta]\] which becomes an isomorphism after inverting $\theta$.
\end{theorem}

\noindent Since localization is an exact functor, this implies the following.

\begin{corollary}\label{c:localization}\cite[Theorem 1]{SS10} Suppose that $(M,\eL_0,\eL_1)$ and $\iota$ satisfy the hypotheses above and $(M^{fix}, \eL_0^{fix}, \eL_1^{fix})$ admits a stable normal trivialization. There is a spectral sequence whose $E_1$ page is isomorphic to \[ HF(M, \eL_0,\eL_1)\otimes \mathbb F[\theta, \theta^{-1}]\] and whose $E_{\infty}$ page is isomorphic to \[ HF(M^{fix}, \eL_0^{fix}, \eL_1^{fix})\otimes \mathbb F[\theta, \theta^{-1}],\] and a corresponding dimension inequality
\[ \dim \left( HF(M, \eL_0, \eL_1)\right) \geq \dim \left( HF\left(M^{fix}, \eL_0^{fix}, \eL_1^{fix}\right)\right).\]
\end{corollary}

Large has shown that this theorem further holds in the case that there is a stable isomorphism of polarization data between the normal polarization and the tangent polarization $(TM^{fix}, T\eL_0^{fix}, T\eL_1^{fix})$ \cite{Large}, but we shall require only the original version in this paper.

\begin{remark}\label{rem:caveat} Caveat lector: Achieving equivariant transversality is delicate, and the isotopy of the Lagrangians in Theorem~\ref{thm:localization} is a critical step. In particular, suppose $(M, \eL_0, \eL_1)$ and $\iota$ are as above. Further suppose that there is an almost complex structure $J$ with respect to which the map on the Floer chain complex $CF(M, \eL_0, \eL_1)$ induced by the permutation of the generators of $\eL_0 \pitchfork \eL_1$ given by $\iota$ is a chain map, perhaps denoted $\iota_{naive}$. It is not necessarily the case that $\iota_{naive}$ is chain homotopy equivalent to the Seidel-Smith involution, nor that the equivariant cohomology of $(CF(M, \eL_0, \eL_1), \iota_{naive})$ localizes to the homology in the fixed set. Indeed, it is straightforward to produce examples where these involutions differ. This means it is in general difficult to compute the Seidel-Smith action and spectral sequence by hand even in simple examples.

In the setting of Heegaard Floer homology, the first author, Lipshitz, and Sarkar \cite{HLS:flexible} showed one may evade this difficulty by working with diagrams which are ``nice'' in the sense of Sarkar and Wang \cite{SW:nice}, in which case the naively-computed involution is the Seidel-Smith involution; in the symplectic Khovanov case, no such combinatorial workaround is known.

In this paper, we largely do not discuss the induced involution on the chain complex, focusing on the involution on the manifold, which is sufficient to apply Thereom~\ref{thm:localization}. The exceptions are in Section~\ref{subsec:Hopf}, in which case we can determine the Seidel-Smith involution on the homology of the chain complex from constraints on gradings and the known equivariant cohomology, and in Section~\ref{s:strongly}, where we discuss exclusively the action on the chain complex in the fixed set, which is unchanged by the isotopy.
\end{remark}

\subsection{Symmetries of knots and links} \label{sec:symmetry} In this section we briefly recap some relevant facts about symmetric knots and links, partially redundantly with the introduction. For a classification of all knot symmetries, we refer the reader to \cite{Boyle:symmetries}.

\subsubsection{Periodic links and their quotients} A link $L \subseteq S^3$ is said to be \emph{$q$-periodic} if there is an orientation-preserving action of $\mathbb Z/q\mathbb Z$ on $S^3$ which preserves $L$ setwise and whose fixed set is an unknotted axis $A$ disjoint from the knot. Taking the quotient under the action gives a quotient map $p \co (S^3, L) \rightarrow (S^3, \overline{L})$ which is a $q$-fold branched covering map over $A$. The link $\overline{L} = p(L)$ is the \emph{quotient link}. In this paper we will be primarily concerned with $2$-periodic links, also called doubly-periodic links.

Any $q$-periodic link admits a diagram preserved by a rotation of $2\pi/q$ radians in the plane; an example of the trefoil as a $2$-periodic knot is shown in Figure~\ref{fig:periodic}. The unknotted axis intersects the plane of the diagram at the origin, and is marked by a star. Such a diagram may be modified to produce a bridge diagram preserved by the same rotation of the plane in which the bridges are exchanged by the action in pairs.

\begin{figure}[ht]
\scalebox{.5}{
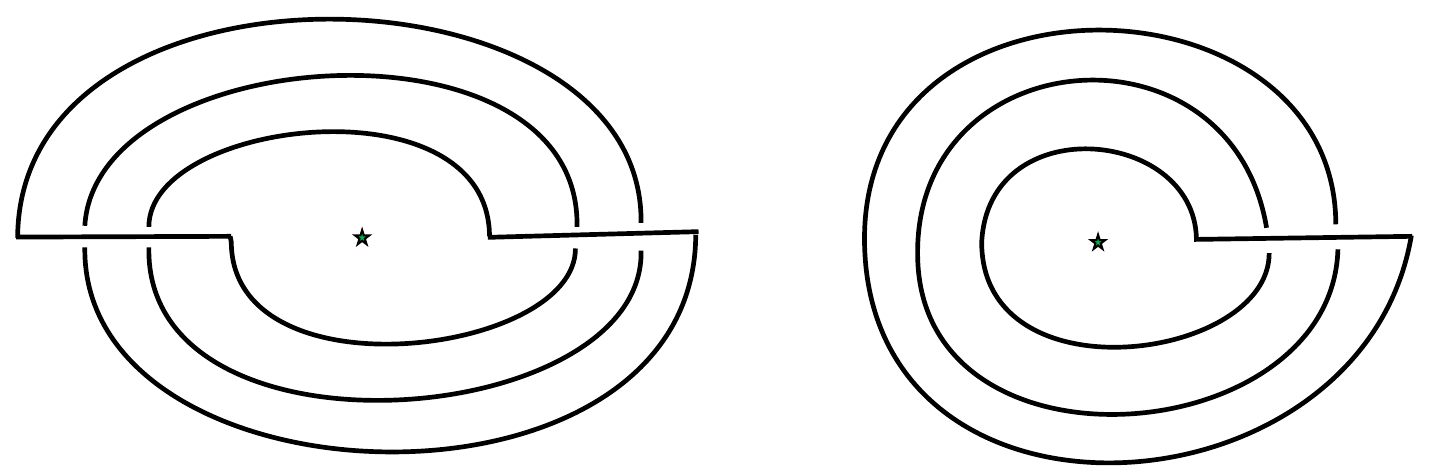}
\caption{Left: A diagram for the left-handed trefoil as a two-periodic knot. Right: The quotient knot (an unknot). The intersection of the axis of symmetry with the plane is marked with a star.}
\label{fig:periodic}
\end{figure}

Finally, note that a $q$-periodic link $L$ can naturally be made into an annular link $L$ by treating the axis of periodicity $A$ as the annular axis; the quotient link may then similarly be made into an annular link with annular axis $\overline{A} = p(A)$.

\subsubsection{Strongly invertible knots and their quotients} A knot $K$ is said to be \emph{strongly invertible} if there is an orientation-preserving action of $\mathbb Z/2\mathbb Z$ on $S^3$ which preserves $K$ setwise and whose fixed set is an unknotted axis $A$ intersecting $K$ in two points; the action therefore reverses the orientation of $K$. There are two natural choices of symmetric diagram for a strongly invertible knot. We will use \emph{intravergent} diagrams, in which the axis of symmetry is perpendicular to the plane of the diagram and the action rotates the diagram by $\pi$ radians. Such diagrams necessarily always have a crossing at the origin, and may be modified to produce a bridge diagram fixed by the same rotation of the plane, in which one bridge is preserved setwise and the remaining bridges are exchanged in pairs. One also has \emph{transvergent} diagrams, in which the axis of symmetry lies in the plane of the diagram. 

\begin{figure}[ht]
\scalebox{.5}{
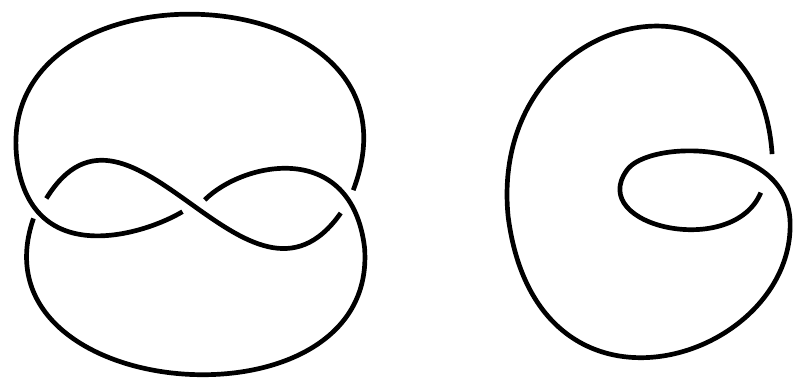}
\caption{Left: An intravergent diagram for the right-handed trefoil as a strongly invertible knot. Right: The quotient knot constructed using the choice of half-axis specified by the diagram.}
\label{fig:strongly_invertible_one}
\end{figure}

Given a strongly invertible knot $K$, the two points $K \cap A$ divide $A$ into half-axes, the closures of which we refer to as $A_1$ and $A_2$. Strongly invertible knots are generally studied together with a choice of (oriented) half-axis, after which there is a well-defined notion of equivariant connected sum \cite{Sakuma:si}. Given a choice of half-axis $A_i$, the \emph{quotient knot} is constructed as follows. If $p \co S^3 \rightarrow S^3$ is the quotient by the action, then $p(K)$ and $p(A_i)$ are line segments with the same endpoints; their union is the quotient knot $\overline{K}$. The two choices of half-axis give two typically non-isotopic quotient knots. A choice of intravergent diagram for the knot fixes a preferred half-axis, to which, the short segment connecting the two strands of the central crossing, and therefore a choice of quotient knot, as recalled in Figure~\ref{fig:strongly_invertible_one}. 

\begin{figure}[ht]
\scalebox{.5}{
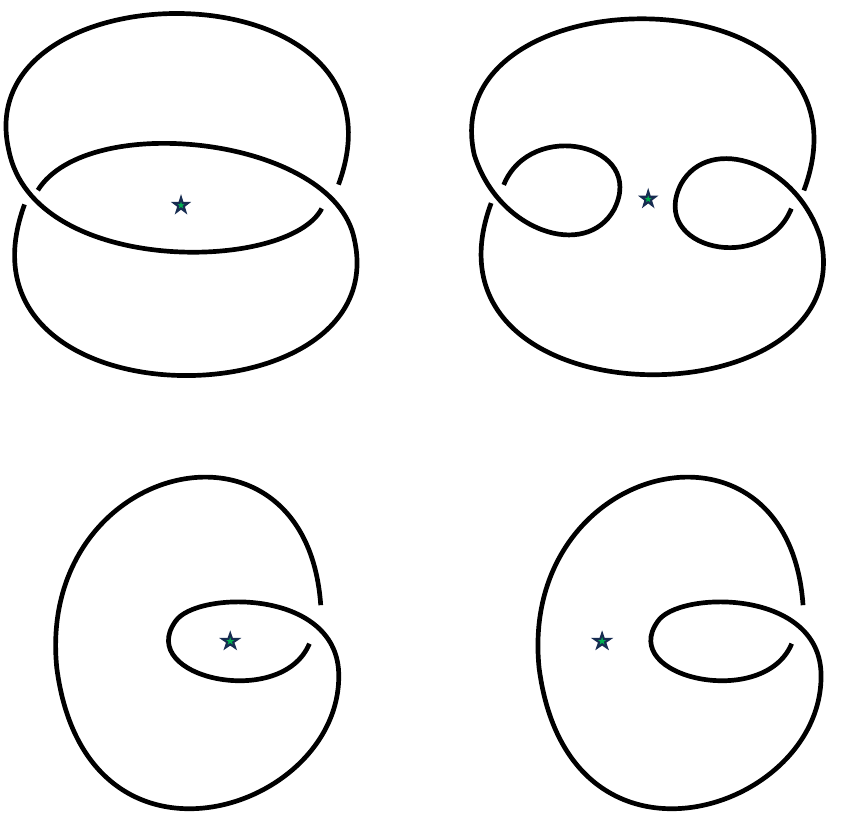}
\caption{Resolutions and quotients of the intravergent diagram of Figure~\ref{fig:strongly_invertible_two}. The top row shows the resolutions $K_0$ and $K_1$ on the left and right respectively. The bottom row shows their quotients $\overline{K}_0$ and $\overline{K}_1$. In all cases the intersection of the axis of symmetry, which is also the annular axis, with the plane is marked with a star.}
\label{fig:strongly_invertible_two}
\end{figure}

Given an intravergent diagram for a strongly invertible knot, taking the zero and one resolutions of the central crossing gives two periodic links $K_0$ and $K_1$, each with axis of symmetry $A$; these periodic links depend on the intravergent diagram in question. The quotients $\overline{K}_0$ and $\overline{K}_1$ are each isotopic to $\overline{K}$ as knots in $S^3$, but differ as annular knots with annular axis $\overline{A}$, as in Figure~\ref{fig:strongly_invertible_two}.

\section{A new symplectic annular Khovanov homology}\label{s:SympAKh}

In this section, we modify Seidel and Smith's construction of symplectic Khovanov homology \cite{SS06} to obtain a new annular link invariant, which we call the symplectic annular Khovanov homology. Conjecturally, this invariant is isomorphic to combinatorial annular Khovanov homology over a field of characteristic zero, and to the Mak-Smith formulation of symplectic annular Khovanov homology likewise; for the precise conjectural relationship, see Conjecture \ref{con:isom}.

\subsection{Geometric setup} \label{sec:geometry}

\subsubsection{Milnor fibres}

Let $\Sym^m(\mathbb{C})$ be the $m$-fold symmetric product of $\mathbb{C}$
and
$\Conf^m(\mathbb{C}) \subset \Sym^m(\mathbb{C})$ be the configuration space of unordered $m$-tuples of pairwise distinct points in $\mathbb{C}$.
For any $\tau=[\tau_1,\dots,\tau_m] \in \Sym^m(\mathbb{C})$, let $p(z) = \prod_{i=1}^m (z-\tau_i)$. We have the associated $A_{m-1}$-surface
\[
A_{\tau}:=\{u^2+v^2=p(z)\}\subset \mathbb{C}^3.
\]
This is a smooth algebraic surface if and only if $\tau \in \Conf^m(\mathbb{C})$, in which case we equip with it the K\"ahler form $\omega_A$  restricted from the standard K\"ahler form in $\mathbb{C}^3$.
For any loop $\gamma \co S^1 \to \Conf^m(\mathbb{C})$, there is an induced monodromy symplectomorphism $\phi_{\gamma} \co A_{\gamma(0)} \to A_{\gamma(1)}=A_{\gamma(0)}$. Two loops which are isotopic relative to the basepoint induce Hamiltonian isotopic symplectomorphisms. Therefore, we have an induced group homomorphism
\[
\Br_m=\pi_1(\Conf^m(\mathbb{C}),\tau) \to \Symp(A_{\tau})/\Ham(A_{\tau})
\]
from the braid group of $m$ strands to the quotient of the symplectomorphism group by the Hamiltonian diffeomorphism group, which is to say the symplectic mapping class group.
Seidel and Smith's symplectic Khovanov homology uses the case $m=2n$ in its construction. In particular, one needs the action induced by the copy of $\Br_n \subseteq \Br_{2n}$ given by the inclusion $b \mapsto b \times 1^n$, or equivalently as the fundamental group of the subspace
\[\{\tau=[\tau_1,\dots, \tau_{2n}] \in \Conf^{2n}(\mathbb{C})| \tau_i=i \text{ for }i=1,\dots,n, \text{ and } Re(\tau_j)<0.5 \text{ for }j>n\}.\]
In the annular setting, we are interested in the subspace of the configuration space consisting of unordered tuples of disjoint points not containing the origin, to wit
\[
\Conf^{2n,0}(\mathbb{C}):=\{\tau=[\tau_1,\dots, \tau_{2n}] \in \Conf^{2n}(\mathbb{C})| \tau_{i} \neq 0 \text{ for all }i\}.
\]
Fix $\tau_o:=[1,\dots,n,-1,\dots,-n] \in \Conf^{2n,0}(\mathbb{C})$, we have the isomorphism
\[
\pi_1(\Conf^{2n,0}(\mathbb{C}), \tau_o) \simeq \Br_{1,2n}
\]
and hence an associated group homomorphism
\[
\Br_{1,2n} \to \Symp(A_{\tau_o})/\Ham(A_{\tau_o}).
\]
Analogously to the above, let $\Br_{1,n}$ denote the copy of $\Br_{1,n} \subseteq \Br_{1,2n}$ included along the map $b \mapsto b \times 1^n$ or equivalently as the fundamental group of the subspace
\[
\{\tau=[\tau_1,\dots, \tau_{2n}] \in \Conf^{2n,0}(\mathbb{C})| \tau_i=i \text{ for }i=1,\dots,n, \text{ and } Re(\tau_j)<0.5 \text{ for }j>n\}.
\]
The induced group homomorphism
\[
\Br_{1,n} \to \Symp(A_{\tau_o})/\Ham(A_{\tau_o})
\]
will be used to define the annular link invariant in Section~\ref{sec:floer}.

Continuing our review of the geometry of $A_{\tau}$, let $\pi_A \co A_{\tau} \to \mathbb{C}_z$ be the projection to the $z$-coordinate. This map is Lefschetz, and the singular values of $\pi_A$ are precisely the points in $\tau$. We say that a path $c$ in $\mathbb{C}$ is a \emph{matching path} for $\tau \in \Conf^{m}(\mathbb C)$ if it is an embedding $c\co [0,1] \to \mathbb{C}$ such that $c(0)$ and $c(1)$ are distinct elements of  $\tau$ and furthermore $c(t) \notin \tau$ for all $t$. To a matching path for $\tau$, one may associate an embedded Lagrangian matching sphere $\eL_c \subseteq A_{\tau}$, a schematic of which is shown in Figure \ref{fig:matching}. Similarly, we say that $c$ is an \emph{annular matching path} if $\tau \in \Conf^{m,0}(\mathbb C)$ and $c$ is 
a matching path for $\tau$ which further satisfies $c(t)\neq 0$ for all $t$, which likewise comes with an embedded matching sphere in $A_{\tau}$.

\begin{figure}[ht]
\begingroup%
  \makeatletter%
  \providecommand\color[2][]{%
    \errmessage{(Inkscape) Color is used for the text in Inkscape, but the package 'color.sty' is not loaded}%
    \renewcommand\color[2][]{}%
  }%
  \providecommand\transparent[1]{%
    \errmessage{(Inkscape) Transparency is used (non-zero) for the text in Inkscape, but the package 'transparent.sty' is not loaded}%
    \renewcommand\transparent[1]{}%
  }%
  \providecommand\rotatebox[2]{#2}%
  \newcommand*\fsize{\dimexpr\f@size pt\relax}%
  \newcommand*\lineheight[1]{\fontsize{\fsize}{#1\fsize}\selectfont}%
  \ifx\svgwidth\undefined%
    \setlength{\unitlength}{140.89285663bp}%
    \ifx\svgscale\undefined%
      \relax%
    \else%
      \setlength{\unitlength}{\unitlength * \real{\svgscale}}%
    \fi%
  \else%
    \setlength{\unitlength}{\svgwidth}%
  \fi%
  \global\let\svgwidth\undefined%
  \global\let\svgscale\undefined%
  \makeatother%
  \begin{picture}(1,0.78326992)%
    \lineheight{1}%
    \setlength\tabcolsep{0pt}%
    \put(0,0){\includegraphics[width=\unitlength,page=1]{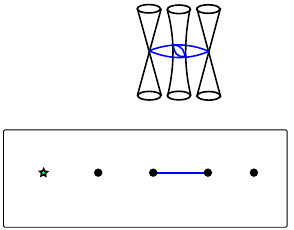}}%
    \put(0.02281365,0.04182502){\color[rgb]{0,0,0}\makebox(0,0)[lt]{\lineheight{1.25}\smash{\begin{tabular}[t]{l}$\mathbb C$\end{tabular}}}}%
    \put(0,0){\includegraphics[width=\unitlength,page=2]{matching_sphere.pdf}}%
  \end{picture}%
\endgroup%

\caption{A matching sphere $\eL_c$. Here the black dots are elements of $\tau \in \Conf^{2m,0}(\mathbb C)$, and the green star is the origin. The blue line segment is an annular matching path connecting two elements of $\tau$, which are critical values of the Lefschetz fibration $A_{\pi}$, and avoiding the origin. The matching sphere $\eL_c$ lies over $c$ in $A_{\tau}$ as shown.}
\label{fig:matching}
\end{figure}

One way to define $\eL_c$ is to use symplectic parallel transport with respect to the fibration $\pi_A$ to construct Lefschetz thimbles emanating from $c(0)$ and $c(1)$ respectively. The resulting vanishing cycle over $z=c(t)$, for $0<t<1$, consists of points where both $u$ and $v$ are real, and the two Lefschetz thimbles can therefore be glued together to give a Lagrangian sphere $\eL_c$.
Alternatively, $\eL_c$ is itself a vanishing cycle: consider the fibration with base $\Sym^{2n}(\mathbb C)$ and fibre $A_{\tau}$ above each $\tau \in \Sym^{2n}(\mathbb{C})$. There is a Lefschetz thimble above the path 
\[
\tilde{c} \co [0,0.5] \to \Sym^{2n}(\mathbb{C}), \quad t \mapsto (\tau \setminus \{c(0),c(1)\})\cup \{c(t),c(1-t)\},
\]
 and its vanishing cycle is $\eL_c$, possibly up to Hamiltonian isotopy.

\subsubsection{Nilpotent slices} \label{subsec:nilpotent}

Let $G=GL_{2n}(\CC)$ and $ \mathfrak{g}$ be its Lie algebra. Let $S_{n} \subseteq \kg$ be the affine subspace consisting of matrices
\[  \left( \begin{array}{cccccccccccccccccccccccc}
a_1   & 1 &         &  & b_1  &    &  & \\
a_2   & 0 & 1       &  & b_2  &    &  & \\
\dots &   & \dots   &   & \dots &  &    \\
a_{n-1}& &         & 1 & b_{n-1}&  &  & \\
a_n    & &         & 0 & b_n    &  &  & \\
c_1      &   &      &    & d_1     & 1  &   & \\
c_2 &     &     &    & d_2    &  0 &1    & \\
\dots & &      &     & \dots & & &\dots & 1 \\
c_n   & &       &     & d_{n}& &  & &0
\end{array} 
\right)\] 
This is a nilpotent slice of the nilpotent element of Jordan type $(n,n)$ (see \cite[Lemma 23]{SS06}).
The set of regular values of the adjoint quotient map $\chi \co S_n \to \Sym^{2n}(\CC)$, which sends a matrix to its set of coefficients of its characteristic polynomial, is precisely $\Conf^{2n}(\CC)$.
For $ \tau \in \Conf^{2n}(\CC)$, we define
\begin{align}
\cY_{n,\tau}:= \chi^{-1}(\tau)
\end{align}
equipped with the K\"ahler form $\omega$ arising from restriction of the standard K\"ahler form on $\mathbb{C}^{4n}$ with coordinates $\{a_i,b_i,c_i,d_i \in \CC: 1 \leq i \leq n\}$.

In \cite{SS06}, Seidel and Smith prove that ``rescaled'' symplectic parallel transport is sufficiently well-behaved on the total space of the adjoint quotient map to define monodromy symplectomorphisms. In particular, for $\tau \in \Conf^{2n}(\mathbb C)$ there is a map
\[\Br_{n} \subseteq \Br_{2n} \to \Symp(\cY_{n,\tau})/\Ham(\cY_{n,\tau})
\]
as in the case of Milnor fibres, and restricting this map we see that for the specific case of $\tau_o \in \Conf^{2n,0}(\mathbb C)$ we have
\[
\Br_{1,n} \subseteq \Br_{1,2n} \to \Symp(\cY_{n,\tau_o})/\Ham(\cY_{n,\tau_o}).
\]

\begin{definition}
A crossingless matching for $\tau \in \Conf^{2m}(\mathbb C)$ consists of $n$ matching paths $\ul{c}=\{c_1,\dots,c_n\}$ for $\tau$ with pairwise disjoint images. Likewise, a crossingless annular matching for $\tau \in \Conf^{2m,0}(\mathbb C)$ consists of $n$ annular matching paths $\ul{c}=\{c_1,\dots,c_n\}$ for $\tau$ with pairwise disjoint images.
\end{definition}

Given a crossingless matching for $\tau_o$, we can use the Seidel-Smith iterated vanishing cycle construction \cite[Section 4(B)]{SS06} to associate to it a Lagrangian submanifold $\eL_{\ul{c}} \subset \cY_{n,\tau_o}$ as follows. Begin with a matching path $c_1$ for $\tau_o$, and use it to define a path 
\[\tilde{c}_1 \co [0,0.5] \to \Symp^{2}(\mathbb{C}), \qquad \tilde{c}_1(t)= \{c_1(t),c_1(1-t)\}.\] 
The vanishing cycle over this path gives a Lagrangian $\eL_{c_1}$ inside of the Milnor fibre $A_{[c_1(0),c_1(1)]} = \cY_{1, [c_1(0), c_1(1)]}$. Similarly, we  define a path 
\[\tilde{c}_2 \co [0,0.5] \to \Symp^{4}(\mathbb{C}), \qquad \tilde{c}_2(t)= \{c_1(0),c_1(1),c_2(t),c_2(1-t)\}.\] 
There is a Morse-Bott degeneration for $\cY_{2,[c_1(0), c_1(1), c_2(t), c_2(1-t)]}$ as $t$ goes to $0.5$, which has critical locus precisely $A_{[c_1(0),c_1(1)]}$. Seidel and Smith consider the vanishing cycles $\eL_{c_1,c_2}$ consisting of points which converge to $\eL_{c_1}$ under this Morse-Bott degeneration. By inductively applying this procedure, they produce a Lagrangian submanifold $\eL_{\ul{c}} \subset \cY_{n,\tau_o}$, which they show that is independent of the ordering of the matching paths $c_i$ up to Hamiltonian isotopy.

Moreover, given a crossingless matching $\ul{c}$ for some $\tau \in \Conf^{2n}(\CC)$, let $\beta$ be an $n$-stranded braid regarded as an element of $\Br_{2n}$. Then there is a symplectomorphism induced by $\beta$ on $\cY_{n,\tau}$, also denoted $\beta$, with the property that $\beta \eL_{\ul{c}}$ is Hamiltonian isotopic to $L_{\beta \ul{c}}$. Likewise, if $\ul{c}$ is a crossingless annular matching for $\tau_o$, if $\beta$ is an $n$-strand annular braid considered as an element of $\Br_{1,2n}$, using the induced symplectomorphism on $\cY_{n,\tau_o}$ gives us $\beta \eL_{\ul{c}}$ Hamiltonian isotopic to $L_{\beta \ul{c}}$.

There is a helpful alternate description of the Lagrangian $\eL_{\ul{c}}$ associated to a crossingless matching due to Manolescu \cite[Theorem 1.2]{Manolescu:nilpotent}, which we now recall. For $\tau \in \Conf^{2n}(\mathbb{C})$, let $\Hilb^n(A_{\tau})$ be the Hilbert scheme of zero dimensional length $n$ subschemes of $A_{\tau}$. This is a smooth algebraic variety. Then Manolescu proves that $\cY_{n,\tau}$ is precisely the horizontal Hilbert scheme of $A_{\tau}$ with respect to the projection $\pi_A \co A_{\tau} \to \mathbb{C}_z$, to wit:
\begin{align}\label{eq:horHilb}
\Hilb^{n,hor}(A_{\tau})=\{ I \in \Hilb^n(A_{\tau}) | \length((\pi_A)_*I)=n\}.
\end{align}
In particular, $\Hilb^{n,hor}(A_{\tau})$ is the open subvariety of $\Hilb^n(A_{\tau})$ consisting of those subschemes $I$ of $A_{\tau}$ whose projection to $\mathbb{C}_z$ remains of length $n$.
We denote the divisor $\Hilb^n(A_{\tau}) \setminus \cY_{n,\tau}$ by $D_r$. 
There is another distinguished divisor $D_{HC}$ in $\Hilb^n(A_{\tau})$, called the Hilbert-Chow divisor, which is the exceptional  divisor of the Hilbert-Chow morphism $\Hilb^n(A_{\tau}) \to \Sym^n(A_{\tau})$.
By \cite[Lemma 5.5]{AS16}, we may equip $\cY_{n,\tau}$ with a K\"ahler form $\omega_{\cY}$ which is the restriction of the product K\"ahler form arising from the K\"ahler structure on $A_{\tau}$ away from a small neighborhood of $D_{HC}$.
By \cite[Theorem 1.2]{Manolescu:nilpotent}, we may deform the K\"ahler form $\omega$ arising from the restriction of the canonical form on $\mathbb{C}^{4n}$ to $\omega_{\cY}$ such that under this deformation, $\eL_{\ul{c}}$ is deformed to a Lagrangian $\eL_{\ul{c}}'$ with respect to $\omega_{\cY}$ which is Hamiltonian isotopic to the product of the matching spheres $\eL_{c_i} \subseteq A_{\tau}$. In other words, $\eL_{\ul{c}}'$ is Hamiltonian isotopic to the Lagrangian consisting of length $n$ subschemes whose support consists of $n$ disjoint points, one on each $\eL_{c_i}$. In coordinates, given a matching path $c$ for $\tau$, one may set $\Sigma_c$ to be the Lagrangian sphere
\begin{equation} \label{eqn:sphere} \Sigma_c = \{(u,v,z): (u,v,z) \in A_{\tau}, z \in \mathrm{Im}(c) u,v \in \sqrt{-p(z)}\mathbb R \}.\end{equation}
Let $\Delta$ be the fat diagonal inside of $\Sym^n{A_{\tau}}$, so that $\Sym^n(A_{\tau})\setminus \Delta$ embeds as an open subset of $\Hilb^n(A_{\tau})$. Then for a crossingless matching $\{c_1, \dots, c_n\}$, after deforming $\omega$ to $\omega_{\cY}$, the product 
\begin{equation}\label{eqn:productL}\Sigma_{\ul{c}} = \Sigma_{c_1} \times \dots \times \Sigma_{c_n} \subseteq (\Sym^n(A_{\tau}) \setminus \Delta) \cap \mathcal Y_{n, \tau_o} \end{equation}
is Hamiltonian isotopic to the Lagrangian $\eL_{\ul{c}}'$. From now on we use this product Lagrangian when working with the K\"ahler form $\omega_{\cY}$.

In the annular setting, we will also consider another distinguished divisor $D_o$.

\begin{definition} Let $D_{axis}:=\{z=0\} \subseteq A_{\tau_o}$, and call this divisor the annular divisor. Then we define
\[
D_o:=\{ I \in \Hilb^{n,hor}(A_{\tau_o})=\cY_{n,\tau_o}: supp(I) \cap D_{axis} \neq \emptyset \}
\]
where $supp(I)$ is the support of $I$ as a subscheme of $A_{\tau_o}$.
\end{definition}

\begin{lemma}
The set $D_o$ is a divisor of $\Hilb^{n,hor}(A_{\tau_o})$.
\end{lemma}

\begin{proof}
The sum $\sum_{k=1}^n A_{\tau_o}^{k-1} \times D_{axis}  \times A_{\tau_o}^{n-k}$ is a $\Sym_n$-invariant divisor of $A_{\tau_o}^{n}$. Therefore, it is locally cutout by a $\Sym_n$-invariant algebraic function.
As a result, it descends to a Cartier divisor 
\[
D:=\{[x_1,\dots, x_n] \in \Sym^n(A_{\tau_o}) : x_i \in D_{axis} \text{ for some }i\}
\]
in $\Sym^n(A_{\tau_o})$. In other words, $D$ is locally cut out by a single element in the ring of algebraic functions.
The pull-back of $D$ along the algebraic morphism
\[
\Hilb^{n,hor}(A_{\tau_o}) \to \Hilb^{n}(A_{\tau_o}) \to \Sym^n(A_{\tau_o})
\]
is $D_o$, by definition.
The pull-back of a Cartier divisor is a Cartier divisor so $D_o$ is a Cartier divisor. 
Since $\Hilb^{n,hor}(A_{\tau_o})$ is smooth, divisors and Cartier divisors are the same, so $D_o$ is a divisor.
\end{proof}

The divisor $D_o$ is preserved by the monodromy action on $\cY_{n,\tau_o}$ induced by the braid group $\Br_{1, 2n}$. We further note that when $\ul{c}$ is a crossingless annular matching, we may choose $\omega_{\cY}$ such that the isotopy from 
$\eL_{\ul{c}}$ to $\eL_{\ul{c}}'$ is disjoint from $D_o$.
We will follow Seidel and Smith's presentation in \cite{SS06} and use the symplectic form $\omega$ arising from the restriction of the form on $\CC^{4n}$ to show that our new definition of symplectic annular Khovanov homology is an annular link invariant. Subsequently, we will employ the deformation from $\omega$ to $\omega_{\cY}$, which does not change the symplectic annular Khovanov homology, and consider various actions on $\cY_{n,\tau}$ using appropriately averaged versions of $\omega_{\cY}$. This is because it is somewhat simpler to work with the theory under the perspective that the Lagrangians are products of spheres $\Sigma_{c_i}$.

\subsection{Floer theory details and definitions} \label{sec:floer}
We now discuss the Floer-theoretic input to Seidel and Smith's symplectic Khovanov homology, and the extensions necessary for symplectic annular Khovanov homology. 

Recall that the manifold $\mathcal Y_{n,\tau}$ with the symplectic form $\omega$ restricted from $\CC^{4n}$ is an exact symplectic manifold which is convex at infinity. Furthermore, for any crossingless matching $\ul{c}$, the associated Lagrangian $\eL_{\ul{c}}$ is topologically a product of spheres, and therefore is an exact Lagrangian. Ergo, given two crossingless matchings $\ul{c}_0$ and $\ul{c}_1$, the Lagrangian Floer cohomology of $(\cY_{n,\tau}, \eL_{\ul{c}_0}, \eL_{\ul{c}_1})$ can be constructed using standard methods. Let $L$ a link in $S^3$ be the closure of a braid $b$ identified with an element $\beta \in \Br_{n}\subseteq \Br_{2n}$. Let $\ul{c}_{up}$ be a crossingless matching of $\tau_o = [-n, \dots, -1, 1, \dots n]$ in the upper half-plane which matches $-k$ to $k$, as in Figure~\ref{fig:upper}. Seidel and Smith's symplectic Khovanov homology is 
\[ \Kh_{symp}^*(L) \vcentcolon = HF^{*+n+w}(\cY_{n, \tau_o}, \eL_{\ul{c}_{up}},\beta \eL_{\ul{c}_{up}}). \]
where $w$ is the writhe of $\beta$.

In the annular setting, we wish to adapt this construction to track intersections of pseudoholomorphic disks with the divisor $D_o$. We now recall how to go about making this precise, using the general setting of Lagrangians $\eL_{\ul{c}_0}$ and $\eL_{\ul{c}_1}$ for two crossingless annular matchings $\ul{c}_0$ and $\ul{c}_1$.

First, we choose a Hamiltonian function $H$ with time-$1$-map $\phi^1_H$ such that $\phi_H(\eL_{\ul{c}_0})$ and $\eL_{\ul{c}_1}$ intersect transversely. We further require that $\phi^t_H(\eL_{\ul{c}_0}) \cap D_o=\emptyset$ for all $t \in [0,1]$.
We next choose a grading datum for each of $\eL_{\ul{c}_0}$ and $\eL_{\ul{c}_1}$ with respect to a holomorphic volume form\footnote{Since $H^1(\cY_{n,\tau_o})=0$, any two choices of holomorphic volume form are homotopic through continuous sections of the canonical line bundle, so there is effectively only one choice of holomorphic volume form from the perspective of gradings.}  on $\cY_{n,\tau_o}$. (See \cite{SeidelGraded} for a general discussion of graded Lagrangians.) 
We let the Floer cochain group be
\[
CF(\cY_{n,\tau_o},\eL_{\ul{c}_0}, \eL_{\ul{c}_1};H; \mathbb F[U]):=\oplus_{x \in \phi_H(\eL_{\ul{c}_0}) \pitchfork \eL_{\ul{c}_1}}\mathbb F[U]x
\]
where $\mathbb F$ is a field of characteristic two, the grading $|x|$ is the Maslov grading, and $U$ is a formal variable of grading zero. 

In addition to the Maslov grading, we can define a relative topological winding grading as follows. Given a crossingless annular matching $\ul{c}$, let $\Sym(\ul{c}) \subseteq  \Conf^n(\mathbb{C})  \subseteq \Sym^n(\mathbb{C})$ be the product of the components of $\ul{c}$.
In other words, a point in $\Sym(\ul{c})$ is an $n$-tuple of points in $\mathbb{C}$ containing one point on each component of $\ul{c}$.
Let $d_o$ be the divisor of $\Sym^n(\mathbb{C})$ consisting of points that meet the origin of $\mathbb{C}$.
If $\ul{c}_0$ and $\ul{c}_1$ are crossingless annular matchings, then the intersection pairing with $d_o$ defines a map \[\cdot d_o \co \pi_2(\Sym^n(\mathbb{C}), \Sym(\ul{c}_0) \cup \Sym(\ul{c}_1)) \to \mathbb{Z}, \qquad u \mapsto u \cdot d_o.\]
Now, the symmetric product of the map $\pi_A \co A_{\tau_o} \to \mathbb{C}_z$ gives us a map $\Sym^n(A_{\tau_o}) \to \Sym^n(\mathbb{C})$.
Precomposing this map with the Hilbert-Chow morphism and restricting to $\cY_{n,\tau_o}$ defines a map $\pi_{\Delta} \co \cY_{n,\tau_o} \to \Sym^n(\mathbb{C})$. It is clear from the definitions that $\pi_{\Delta}(L_{\ul{c}_i})=\Sym(\ul{c}_i)$ for $i=0,1$, and we may choose our Hamiltonian perturbation $H$ to be compatible with $\pi_{\Delta}$ so that this remains true after perturbation.

\begin{definition}\label{d:winding}
For  $x_0,x_1 \in \phi_H(\eL_{\ul{c}_0}) \pitchfork \eL_{\ul{c}_1}$, 
we define the relative winding grading $w(x_0,x_1)$ to be $2(u \cdot d_o)$ where $u \co \mathbb{R} \times [0,1] \to \Sym^n(\mathbb{C})$
is a topological strip such that $\lim_{s \to -\infty} u(s,t)=\pi_{\Delta}(x_0)$ and $\lim_{s \to \infty} u(s,t)=\pi_{\Delta}(x_1)$, and furthermore $u(s,0) \in \Sym(\ul{c}_0)$ and $u(s,1) \in \Sym(\ul{c}_1)$.
\end{definition}
Here the factor of two is chosen in order to match the conventions of combinatorial annular Khovanov homology. The fact that the relative winding grading is independent of the choice of the topological strip $u$ follows from the observation that $\Sym^n(\mathbb{C})=\mathbb{C}^n$, $\Sym(\ul{c}_0)$ and $\Sym(\ul{c}_1)$ are all contractible. In Section~\ref{sec:invariance} we discuss how to promote this relative grading to an absolute grading on the annular symplectic Khovanov homology.

Continuing with Floer theory, let $\mathcal{J}$ be the space of $\omega$-tamed families of almost complex structures which agree with the complex structure of $\cY_{n,\tau_o}$ near infinity, near the intersection $D_{HC} \cap \cY_{n,\tau_o}$, and near the annular divisor $D_o$.
For a generic choice of $J=J_t$ for $t \in [0,1]$ in $\in \mathcal{J}$, the moduli space of Floer solutions $u:\mathbb{R} \times [0,1] \to \cY_{n,\tau_o}$ between intersection points $x_0$ and $x_1$ in $\phi_H(\eL_{\ul{c}_0}) \pitchfork \eL_{\ul{c}_1}$ 
\begin{align}
 \mathcal{M}(x_0,x_1;J):=\{u : du^{0,1}=0, \lim_{s \to -\infty} u(s,t)=x_0, \lim_{s \to \infty} u(s,t)=x_1, u(s,0) \in \phi_H(\eL_{\ul{c}_0}), u(s,1) \in \eL_{\ul{c}_1}\}/\mathbb{R} \label{eq:strips}
\end{align}
is transversely cut out and can be compactified to a topological manifold with corners.
The condition that $J$ agrees with the complex structure of $\cY_{n,\tau_o}$ near infinity guarantees that we can run Gromov compactness to compactify the moduli spaces.

The condition that $J$ agrees with the complex structure of $\cY_{n,\tau_o}$ near $D_o$, together with the fact that $D_o \cap \cY_{n,\tau_o}$ does not contain any $J$-holomorphic sphere, guarantees that every intersection between $u$ and $D_o$ is positive. This positivity in turn implies that we can define the differential to be
\begin{align} \label{eq:differential}
\partial x_1=\sum_{x_0, |x_0|-|x_1|=1} \#\mathcal{M}(x_0,x_1;J) U^{u \cdot D_o}x_0 
\end{align}
where $u \cdot D_o$ is the algebraic intersection number between $u \in \mathcal{M}(x_0,x_1;J)$ and $D_o$. This number is independent of the choice of $u \in \mathcal{M}(x_0,x_1;J)$, and moreover $2(u \cdot D_o)$ is equal to the relative winding grading $w(x_0,x_1)$. 
We denote the homology with respect to the differential $d$ by $HF(\cY_{n, \tau_o},\eL_0,\eL_1;H)$, and let the \emph{annular homology} $HF_{ann}(\cY_{n, \tau_o},\eL_0,\eL_1;H)$ be the homology of the truncated complex
\[CF(\cY_{n, \tau_o},\eL_0,\eL_1;H; \mathbb F[U])/(U=0):=CF(\cY_{n, \tau_o},\eL_0,\eL_1;H) \otimes_{\mathbb{F}[U]} \mathbb{F}.\]
In particular, the annular homology does not count disks passing through the divisor $D_o$. Both $HF(\cY_{n, \tau_o},\eL_0,\eL_1;H; \mathbb F[U])$ and $HF_{ann}(\cY_{n, \tau_o},\eL_0,\eL_1;H)$ are independent of the choice of Hamiltonian $H$; we therefore drop $H$ from the notation when the explicit choice is not important. Note furthermore that the differential of $CF_{ann}(\cY_{n, \tau_o}, L_{\ul{c}_0}, L_{\ul{c}_1})$ preserves the relative winding grading, and therefore $HF_{ann}(\cY_{n, \tau_o},L_{\ul{c}_0}, L_{\ul{c}_1})$ splits with respect to the relative winding grading.

With this in mind, we are ready to define our annular link invariant. Let $L$ an annular link in $S^3$ be the closure of a braid $b$ identified with an element $\beta \in \Br_{1,n} \subseteq \Br_{1,2n}$. As previously, let $\ul{c}_{up}$ be a crossingless annular matching of $\tau_o = [-n, \dots, -1, 1, \dots n]$ in the upper half-plane which matches $-k$ to $k$, as in Figure~\ref{fig:upper}.
\begin{definition} \label{def:AKh} The annular symplectic Khovanov homology of $L$ over $\FF$ is
\[ \AKh_{symp}^*(L) \vcentcolon = HF_{ann}^{*+n+w}(\cY_{n, \tau_o}, \eL_{\ul{c}_{up}}, \beta\eL_{ \ul{c}_{up}}). \]
\end{definition}

\begin{remark}\label{r:char0} For notational simplicity, we have introduced the annular symplectic Khovanov homology over fields of characteristic two; and indeed in this paper we will be almost exclusively concerned with the theory over $\mathbb F_2$ the field with two elements. However, the Lagrangian $\eL_{\ul{c}}$ associated to a crossingless matching is diffeomorphic to a product of spheres, and is thus spin with a unique spin structure. Therefore, it is also straightforward to give coherent orientations to the moduli spaces of Floer solutions to define the symplectic annular Khovanov homology over the integers, and thus any field. The invariance proofs of Section~\ref{sec:invariance} go through equally well in this case.
\end{remark}

\begin{remark}\label{r:tauto} Note that the condition that $J$ agrees with the complex structure of $\cY_{n,\tau_o}$ near $D_{HC} \cap \cY_{n,\tau_o}$ allows us to apply the tautological correspondence to a pseudoholomorphic strip $u$ to get a pseudo-holomorphic map $(v,\pi_{\Sigma})\co\Sigma \to A_{\tau_o} \times (\mathbb{R} \times [0,1])$ such that $\pi_{\Sigma} \co \Sigma \to \mathbb{R} \times [0,1]$ is an $n$-fold branched covering. See \cite[Section 5.8]{AS16} and \cite{DonSmi, Smith03}, as well as \cite[Lemma 3.6]{OS_Disks} and \cite[Section 13]{Lip06} for further discussion of the tautological correspondence. This observation will be helpful for our applications. \end{remark}

\subsection{Invariance under Markov moves and other properties} \label{sec:invariance}

In this section we establish invariance of our annular symplectic Khovanov homology under Markov moves, and discuss how to promote the relative winding grading constructed in the previous section to an annular grading. We then conclude by proving Theorem~\ref{thm:akh-to-kh}. 

Our invariance arguments closely follow those of \cite{SS06} for ordinary symplectic Khovanov homology. We begin with the following modification of \cite[Lemma 49]{SS06}, showing invariance for handleslides among the components of a crossingless annular matching.

\begin{lemma}[Handleslide invariance]\label{l:slide}
Let $\ul{c}=\{c_1,\dots,c_n\}$ be a crossingless annular matching. Suppose that $c_1'$ is an annular matching obtained from handlesliding $c_1$ across $c_2$ such that the handleslide region does not intersect the origin in $\in \mathbb{C}$. Define $\ul{c}':=\{c_1',c_2,\dots,c_n\}$.
Then the associated Lagrangian $\eL_{\ul{c}}$ is Hamiltonian isotopic to $\eL_{\ul{c}'}$ through an isotopy in the complement of $D_o$. In particular, for any other crossingless annular matching $\ul{d}$, we have  $HF(\cY_{n, \tau_o}, \eL_{\ul{d}},\eL_{\ul{c}}; \mathbb F[U]) \simeq HF(\cY_{n,\tau_o}, \eL_{\ul{d}},\eL_{\ul{c}'}; \mathbb F[U])$.
\end{lemma}

An example of such a handleslide is shown in Figure~\ref{fig:sliding}.

\begin{proof}
The proof of \cite[Lemma 49]{SS06} goes as follows. 
Recall that the two endpoints of $c_2$ are elements of $2n$-tuple $\tau_0$.
By moving the two endpoints towards each other along $c_2$ until they coincide and adjoining the endpoints of $c_j$ for $j \neq 2$, one obtains a path $(\tau_t)_{t \in [0,1]}$ of $2n$-tuples starting at $\tau_o$ and ending at a point in $\Sym^{2n}(\CC) \setminus \Conf^{2n}(\CC)$.
The corresponding family $\cY_{n, \tau_t}$ gives a Morse-Bott degeneration from $\cY_{n, \tau_0}$ to the singular space $\cY_{n, \tau_1}$, and $\eL_{\ul{c}}$ is the vanishing cycle of the Lagrangian $\eL_{\{c_1,c_3,\dots,c_n\}}$ in the critical locus of 
$\cY_{n, \tau_1}$.
In the critical locus, $\eL_{\{c_1,c_3,\dots,c_n\}}$ is Hamiltonian isotopic to $\eL_{\{c_1',c_3,\dots,c_n\}}$ because $c_1$ is homotopic to $c_1'$ when $c_2$ (including the two endpoints of $c_2$) is absent.
The vanishing cycle of $\eL_{\{c_1',c_3,\dots,c_n\}}$ is $\eL_{\ul{c}'}$.
Since $\eL_{\{c_1,c_3,\dots,c_n\}}$ is Hamiltonian isotopic to $\eL_{\{c_1',c_3,\dots,c_n\}}$ and $\eL_{\ul{c}}$ and $L_{\ul{c}'}$ are the respective vanishing cycles, $\eL_{\ul{c}}$ is Hamiltonian isotopic to $\eL_{\ul{c}'}$.

Since the path $(\tau_t)_{t \in [0,1]}$ and the isotopy from $c_1$ to $c_1'$ when $c_2$ is absent are both away from the origin, one can show that the Hamiltonian isotopy from  $L_{\ul{c}}$ to $L_{\ul{c}'}$ is away from $D_o$. \end{proof}

\begin{remark}
For the remainder of the paper, we only need the weaker statement that $\eL_{\ul{c}}$ is Floer theoretically isomorphic to $\eL_{\ul{c}'}$, in the sense that $HF(\cY_{n, \tau_o}, \eL_{\ul{c}},\eL_{\ul{c}''}; \mathbb F[U])=HF(\cY_{n, \tau_o}, \eL_{\ul{c}'},\eL_{\ul{c}''}; \mathbb F[U])$ for any crossingless annular matching $\ul{c}''$.
One way to prove this is to show that there is an element $a \in HF(\cY_{n, \tau_o}, \eL_{\ul{c}},\eL_{\ul{c}'}; \mathbb F[U])$ and $a' \in HF(\cY_{n, \tau_o}, \eL_{\ul{c}'},\eL_{\ul{c}}; \mathbb F[U])$ such that the Floer product $\mu_2(a,a')$ and $\mu_2(a',a)$ are the respective identity elements.
This can be proved by combining (i)  $\eL_{\ul{c}}$ is Floer theoretically isomorphic to $\eL_{\ul{c}'}$ when $U=1$, which is exactly \cite[Lemma 49]{SS06}, and (ii) for any pseudoholomorphic map contributing to the Floer product, 
the projection of the corresponding pseudo-holomorphic map $\pi_A \circ v \co \Sigma \to \CC$, as in Remark \ref{r:tauto}, misses the origin because of the open mapping theorem.
\end{remark}

\begin{figure}[ht]
\fontsize{8pt}{12pt}
\begingroup%
  \makeatletter%
  \providecommand\color[2][]{%
    \errmessage{(Inkscape) Color is used for the text in Inkscape, but the package 'color.sty' is not loaded}%
    \renewcommand\color[2][]{}%
  }%
  \providecommand\transparent[1]{%
    \errmessage{(Inkscape) Transparency is used (non-zero) for the text in Inkscape, but the package 'transparent.sty' is not loaded}%
    \renewcommand\transparent[1]{}%
  }%
  \providecommand\rotatebox[2]{#2}%
  \newcommand*\fsize{\dimexpr\f@size pt\relax}%
  \newcommand*\lineheight[1]{\fontsize{\fsize}{#1\fsize}\selectfont}%
  \ifx\svgwidth\undefined%
    \setlength{\unitlength}{209.62502307bp}%
    \ifx\svgscale\undefined%
      \relax%
    \else%
      \setlength{\unitlength}{\unitlength * \real{\svgscale}}%
    \fi%
  \else%
    \setlength{\unitlength}{\svgwidth}%
  \fi%
  \global\let\svgwidth\undefined%
  \global\let\svgscale\undefined%
  \makeatother%
  \begin{picture}(1,0.24508043)%
    \lineheight{1}%
    \setlength\tabcolsep{0pt}%
    \put(0,0){\includegraphics[width=\unitlength,page=1]{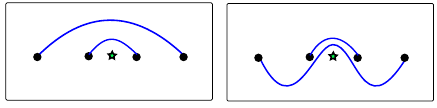}}%
    \put(0.35957058,0.18425753){\color[rgb]{0,0,0}\makebox(0,0)[lt]{\lineheight{1.25}\smash{\begin{tabular}[t]{l}$c_1$\end{tabular}}}}%
    \put(0.29695864,0.14669044){\color[rgb]{0,0,0}\makebox(0,0)[lt]{\lineheight{1.25}\smash{\begin{tabular}[t]{l}$c_2$\end{tabular}}}}%
    \put(0.88729847,0.03756704){\color[rgb]{0,0,0}\makebox(0,0)[lt]{\lineheight{1.25}\smash{\begin{tabular}[t]{l}$c_1'$\end{tabular}}}}%
    \put(0.77280842,0.16457949){\color[rgb]{0,0,0}\makebox(0,0)[lt]{\lineheight{1.25}\smash{\begin{tabular}[t]{l}$c_2$\end{tabular}}}}%
    \put(0,0){\includegraphics[width=\unitlength,page=2]{handleslide.pdf}}%
  \end{picture}%
\endgroup%

\caption{Two crossingless annular matchings $\ul{c}=(c_1,c_2)$ and $\ul{c}'=(c_1', c_2)$, with matching paths marked in blue, each avoiding the origin, marked with a green star. The path $c_1'$ is the result of sliding $c_1$ over $c_2$, with handleslide region not intersecting the origin. The associated Lagrangians $\eL_{\ul{c}}$ and $\eL_{\ul{c}'}$ are Lagrangian isotopic in $\cY_{n,\tau_o} \setminus D_o$.}
\label{fig:sliding}
\end{figure}

We now confirm invariance of the annular symplectic Khovanov homology under Markov conjugation. As in Section~\ref{sec:floer}, let $\ul{c}_{up}$ be a crossingless annular matching for $\tau_o$ in the upper-half plane which matches $-k$ with $k$, as in Figure \ref{fig:upper}. Our proof closely mimics \cite[Proposition 54]{SS06}.

\begin{figure}[ht]
\includegraphics{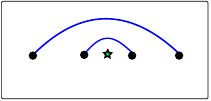}
\caption{A crossingless annular matching $\ul{c}_{up}$ in the upper-half plane, with matching paths drawn in blue, avoiding the origin, marked with a green star.}
\label{fig:upper}
\end{figure}

\begin{lemma}[Invariance under Markov conjugation]\label{l:M1}
For any braid $\beta \in \Br_{1,n}\subseteq \Br_{1,2n}$
\[
HF(\cY_{n, \tau_o},\eL_{\ul{c}_{up}},\sigma_j \beta \sigma_j^{-1} \eL_{\ul{c}_{up}}; \mathbb F[U])\simeq HF(\cY_{n,\tau_o},   \eL_{\ul{c}_{up}},\beta \eL_{\ul{c}_{up}}; \mathbb F[U])
\]
for all $j=1,\dots,n-1$.
\end{lemma}

\begin{proof} We recall from \cite[Lemma 53]{SS06} that for $1 \leq j \leq n$, the crossingless annular matching obtained by acting by $\sigma_j^{-1} \sigma_{2n-j}$ on $\ul{c}_{up}$ differs from $\ul{c}_{up}$ by a handleslide away from the origin, as in Figure~\ref{fig:double_twist}. With this in mind, we have the following string of isomorphisms:
\begin{align*}
HF(\cY_{n, \tau_o}, \eL_{\ul{c}_{up}}, \sigma_j \beta \sigma_j^{-1}\eL_{\ul{c}_{up}}; \mathbb F[U])&\simeq HF(\cY_{n, \tau_o},  \sigma_j^{-1} \eL_{\ul{c}_{up}}, \beta \sigma_j^{-1} \eL_{\ul{c}_{up}}; \mathbb F[U])\\
&\simeq HF(\cY_{n, \tau_o} , \sigma_{2n-j}^{-1} \eL_{\ul{c}_{up}},\beta \sigma_j^{-1} \eL_{\ul{c}_{up}}; \mathbb F[U])\\
&\simeq HF(\cY_{n, \tau_o}, \eL_{\ul{c}_{up}}, \sigma_{2n-j} \beta \sigma_j^{-1} \eL_{\ul{c}_{up}}; \mathbb F[U])\\
& \simeq HF(\cY_{n, \tau_o}, \eL_{\ul{c}_{up}},\beta  \sigma_{2n-j} \sigma_j^{-1} \eL_{\ul{c}_{up}}; \mathbb F[U])\\
& \simeq HF(\cY_{n, \tau_o},   \eL_{\ul{c}_{up}}, \beta \eL_{\ul{c}_{up}}; \mathbb F[U]).
\end{align*}
Here the first and third isomorphisms are applications of invariance under symplectomorphisms preserving the divisor $D_o$, the second and fifth isomorphisms are applications of Lemma \ref{l:slide}, and the third isomorphism uses the fact that if $\beta \in \Br_{1,n} \subseteq \Br_{1,2n}$ and $1 \leq j \leq n-1$ then $\beta$ commutes with $\sigma_{2n-j}$ in $\Br_{1,2n}$.
\end{proof}

\begin{figure}[ht]
\includegraphics{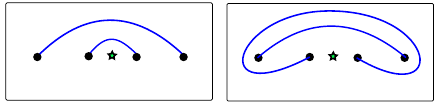}
\caption{The effect of acting on $\ul{c}_{up}$ by $\sigma_{2n-j}\sigma_{j}^{-1}$ is shown on the right in the case $j=1$ and $n=2$. We observe that this is handleslide equivalent to $\ul{c}_{up}$, shown on the left.}
\label{fig:double_twist}
\end{figure}

We now prove invariance under Markov stabilization, following \cite[Proposition 55, Lemma 57]{SS06}. To distinguish $\tau_o$ in different dimensions, let $\tau_{o,n}=\{-n,\dots,n\}$ for $n \in \mathbb{N}$.

\begin{lemma}[Invariance under Markov stabilization]\label{l:M2}
For any $\beta_1,\beta_2 \in \Br_{1,n}$
\begin{align*}
HF^*(\cY_{n+1, \tau_{o,n+1}},\eL_{\ul{c}_{up}},\beta_1 \sigma_{n} \beta_2 \eL_{\ul{c}_{up}}; \mathbb F[U])\simeq HF^*(\cY_{n, \tau_{o,n}},\eL_{\ul{c}_{up}},\beta_1 \beta_2  \eL_{\ul{c}_{up}}; \mathbb F[U])\\
HF^{*+2}(\cY_{n+1, \tau_{o,n+1}},\eL_{\ul{c}_{up}},\beta_1 \sigma_{n}^{-1} \beta_2 \eL_{\ul{c}_{up}}; \mathbb F[U])\simeq HF^*(\cY_{n, \tau_{o,n}},\eL_{\ul{c}_{up}},\beta_1 \beta_2  \eL_{\ul{c}_{up}}; \mathbb F[U])
\end{align*}
\end{lemma}

\begin{proof}By symplectomorphism invariance, it suffices to prove that 
\begin{align*}
HF^*(\cY_{n+1, \tau_{o,n+1}},\beta_1^{-1}\eL_{\ul{c}_{up}}, \sigma_{n} \beta_2 \eL_{\ul{c}_{up}}; \mathbb F[U])\simeq HF^*(\cY_{n, \tau_{o,n}},\beta_1^{-1}\eL_{\ul{c}_{up}}, \beta_2  \eL_{\ul{c}_{up}}; \mathbb F[U])\\
HF^{*+2}(\cY_{n+1, \tau_{o,n+1}},\beta_1^{-1}\eL_{\ul{c}_{up}}, \sigma_{n}^{-1} \beta_2 \eL_{\ul{c}_{up}}; \mathbb F[U])\simeq HF^*(\cY_{n, \tau_{o,n}},\beta_1^{-1}\eL_{\ul{c}_{up}}, \beta_2  \eL_{\ul{c}_{up}}; \mathbb F[U])
\end{align*}
Following the proof of \cite[Proposition 55]{SS06}, we wish to consider a Morse-Bott degeneration arising from colliding $\{-2,-1,1\}$. Toward this end, we can choose a path $\gamma\co[0,1] \to \Sym^{2(n+1)}(\mathbb{C})$ such that $\gamma(0)=\tau_{o,n+1}$, for all $t \in [0,1)$ we have $\{-(n+1),\dots,-2,2,\dots,(n+1)\} \subseteq \gamma(t) \in \Conf^{2(n+1),0}(\mathbb{C})$, and finally 
\[
\gamma(1)=(-(n+1),\dots,-3,-2,-2,-2,2,3,\dots,n+1) \in \Sym^{2(n+1)}(\mathbb{C}).
\]
In particular this path avoids the origin. The remainder of the argument goes through as in \cite[Proposition 55]{SS06} without alteration. More precisely, the critical locus of this degeneration is 
$$\cY_{n, (-(n+1),\dots,-3,-2,2,3,\dots,n+1)} \simeq \cY_{n, \tau_{o,n}}.$$
and furthermore this identification sends the annular divisor of one side to the annular divisor of the other. Moreover, under this identification, the Lagrangian $\sigma_{n}^{\pm 1} \beta_2 L_{\ul{c}_{up}}$ and $\beta_1 ^{-1}\eL_{\ul{c}_{up}}$ are vanishing cycles over
$\beta_2  \eL_{\ul{c}_{up}}$ and  $\beta_1 ^{-1} \eL_{\ul{c}_{up}}$, respectively, with their fibred spheres transversely intersecting at $1$ point.
 This show that $HF(\cY_{n+1, \tau_{o,n+1}}, \beta_1 ^{-1}\eL_{\ul{c}_{up}},\sigma_{n}^{\pm 1} \beta_2 \eL_{\ul{c}_{up}}; \mathbb F[U])$ is identified with the Floer cohomology inside the critical locus, which is $HF(\cY_{n, \tau_{o,n}}, \beta_1 ^{-1} \eL_{\ul{c}_{up}}, \beta_2  \eL_{\ul{c}_{up}}; \mathbb F[U])$, up to a grading shift by $0$ and $2$ respectively for $\sigma_{n}$ and $\sigma_{n}^{-1}$.
\end{proof}

We may now briefly complete the proof of Theorem~\ref{thm:invariant}, establishing that $\AKh_{symp}(L)$ is a well-defined link invariant.

\begin{proof}[Proof of Theorem \ref{thm:invariant}]
Recall from Theorem~\ref{t:Markov} that two annular braids $\beta_1$ and $\beta_2$ have the same annular link closures if and only if they differ by a sequence of Markov conjugations and Markov stabilizations. Therefore invariance follows from  Lemmas \ref{l:M1} and \ref{l:M2}. Moreover, since the isomorphisms of Lemmas \ref{l:M1} and \ref{l:M2}) preserve the relative winding grading, there is a well-defined relative winding grading on $\AKh_{symp}(L)$; the theory therefore decomposes into a direct sum along the relative winding grading, every direct summand of which is an invariant of $L$.
\end{proof}

We conclude the discussion of invariance with the proof of Theorem~\ref{thm:akh-to-kh}.

\begin{proof} We observe that $\AKh_{symp}(L)$ is the homology of the associated graded of the filtration induced by the (relative) winding grading on $CF(\cY_{n, \tau_o}, \eL_{\ul{c}_{up}}, \beta \eL_{\ul{c}_{up}})$, the chain complex computing $\Kh_{symp}(L)$. The existence of the spectral sequence of Theorem~\ref{thm:akh-to-kh} therefore follows from standard homological algebra. Moreover, the maps of Lemmas \ref{l:M1} and \ref{l:M2} induce filtered chain homotopy equivalences on $CF(\cY_{n, \tau_o}, \eL_{\ul{c}_{up}}, \beta \eL_{\ul{c}_{up}})$, from which follows invariance of every page of the spectral sequence.
\end{proof}

\begin{remark} We remind the reader that so far we have used the K\"ahler form restricted from $\CC^{4n}$.
As noted at the end of Section~\ref{sec:geometry}, a direct generalization of \cite[Theorem 1.2]{Manolescu:nilpotent} show that we can compute the Floer cohomology and hence the symplectic annular Khovanov homology using $\omega_{\cY}$ instead, with respect to which the Lagrangians are isotopic to products of spheres.
\end{remark}

We now discuss a useful alternate construction of $HF_{ann}$, in which we delete the divisor $D_o$ instead of counting its intersections with pseudoholomorphic strips. Denote $A_{\tau} \setminus \{z=0\}$ by $A_{\tau}^*$, which is a hypersurface in $\CC^2 \times \CC^*$.
The standard K\"ahler form on  $\CC^2 \times \CC^*$ is $-dd^c(|u|^2+|v|^2+\log |z|)$, which restricts to a K\"ahler form $\omega_{A^*}$ on $A_{\tau}^*$. There exists a K\"ahler form $\omega_{\cY_o}$ on $\Hilb^{n, hor}(A_{\tau}^*)$ which agrees with the product form of $\omega_{A^*}$ away from a small neighborhood of $D_{HC}$ by the same argument as \cite[Lemma 5.5]{AS19}. With respect to this form, we can consider the Lagrangian Floer cohomology
\begin{equation} HF(\Hilb^{n, hor}(A_{\tau}^*), \Sigma_{\ul{c}}, \Sigma_{\ul{d}})  \label{eqn:deleted}\end{equation}
where $\Sigma_{\ul{c}}$ is the product of the matching spheres for the matching paths $c_i \in \ul{c}$ as in (\ref{eqn:productL}) and likewise $\Sigma_{\ul{d}}$.

As a preparation for the following lemma, we recall that: 
\begin{definition}\label{d:stein}
A Stein homotopy on a smooth manifold  $X$ is a family of Stein structures $(J_t, \phi_t)_{t \in [0,1]}$, where $J_t$ is a complex structure and $\phi_t$ is an  exhausting plurisubharmonic function, such that the critical points of $\phi_t$ do not go to the infinity during the homotopy.
In this case, we call $(J_0, \phi_0)$ Stein homotopy equivalent to $(J_1, \phi_1)$.
A Stein deformation equivalence between $(X_0,J_0,\phi_0)$  and $(X_1,J_1,\phi_1)$ consists of a diffeomorphism $f:X_0 \to X_1$ and a Stein homotopy equivalence from $(J_0,\phi_0)$ to $(f^*J_1,f^*\phi_1)$.
\end{definition}
A Stein homotopy induces a symplectomorphism between the associated K\"ahler forms $-d (d\phi_0 \circ J_0)$ and $-d (d\phi_1 \circ J_1)$. More discussion of Liouville/Weinstein/Stein homotopy can be found in \cite[Chapter 11]{CE}.

\begin{lemma}\label{l:AKHvariant}
The Lagrangian Floer cohomology in the horizontal Hilbert scheme with the divisor deleted \eqref{eqn:deleted} agrees with $HF_{ann}(\Hilb^{n, hor}(A_{\tau}^*), \Sigma_{\ul{c}}, \Sigma_{\ul{d}})$. Therefore, given an annular link $L$ which is the closure of a braid $\beta \in \Br_{1,n}$, we have that $HF(\Hilb^{n, hor}(A_{\tau}^*), \Sigma_{\ul{c}_{up}}, \Sigma_{\beta\ul{c}_{up}})$ agrees with $\AKh_{symp}(L)$ as defined in Definition~\ref{def:AKh}.
\end{lemma}

\begin{proof}[Sketch of proof]
Recall that $\omega_A$ is the standard K\"ahler form on $A_{\tau}$.
For any small neighborhood $N$ of $\pi_A^{-1}(0)=\{z=0\} \subseteq A_{\tau}$, we can construct a K\"ahler form $\omega'$ on  $A_{\tau}^*$ which agrees with $\omega_A|_{A_{\tau}^*}$ outside $N$ such that $(A_{\tau}^*,\omega')$ is Stein and Stein homotopy equivalent to $(A_{\tau}^*,\omega_{A^*})$. Note that we do not change the complex structure during the homotopy. For example, $\omega'$ can be constructed using a K\"ahler potential of the form 
$|u|^2+|v|^2+|z|^2+ \rho(|z|)\log |z|$ for some cutoff function $\rho$ supported in a small neighborhood of $0$. One may then use this to produce a K\"ahler form $\omega'_{\cY}$ on $\Hilb^{n, hor}(A_{\tau}^*)$ which agrees with $\omega_{\cY}$ away from an arbitrarily small neighborhood $N_o$ of $D_o$. Now, for any pair of product Lagrangians associated to matchings $\Sigma_{\ul{c}}$ and $\Sigma_{\ul{d}}$, there is such an arbitrarily small neighborhood $N_o$ of $D_o$ such that all the Floer solutions contributing to $HF_{ann}(\Hilb^{n,hor}(A_{\tau}),\Sigma_{\ul{c}},\Sigma_{\ul{d}})$ do not pass through $N_o$. Therefore, $HF(\Hilb^{n, hor}(A_{\tau}^*), \Sigma_{\ul{c}}, \Sigma_{\ul{d}})$ computed using $\omega'_{\cY}$ agrees with $HF_{ann}(\Hilb^{n, hor}(A_{\tau}), \Sigma_{\ul{c}}, \Sigma_{\ul{d}})$. Since $(A_{\tau}^*,\omega')$ is Stein homotopy equivalent to $(A_{\tau}^*,\omega_{A^*})$, it follows that the same is true if we replace $\omega'_{\cY}$ with $\omega_{\cY_o}$.
\end{proof}

As a consequence of Lemma \ref{l:AKHvariant}, we will use the two constructions of symplectic annular Khovanov homology interchangeably.

\subsubsection{The absolute winding grading}\label{subsec:absolutewinding}

The first application of this new variant of symplectic annular Khovanov homology, in which we delete the divisor $D_o$, is to promote the relative winding grading on symplectic Khovanov homology to an absolute grading. The following discussion is inspired by \cite[Lemma 2]{GLW:Schur}. We can compactify $\mathbb{C}$ to $\mathbb{CP}^1$ by adding a point $\{\infty\}$ at infinity.
Let $d_{\infty}$ be the corresponding divisor of $\Sym^n(\mathbb{CP}^1)$, that is, let $d_{\infty} = \{\infty\} \times \Sym^{n-1}(\mathbb{CP}^1)$. As previously, we have an intersection pairing \[\cdot d_\infty \co \pi_2(\Sym^n(\mathbb{CP}^1), \Sym(\ul{c}_0) \cup \Sym(\ul{c}_1)) \to \mathbb{Z}.\]
We can define the relative winding grading at infinity in exactly the same way as Definition \ref{d:winding} by 
considering $u_{\infty}\co \mathbb{R} \times [0,1] \to \Sym^n(\mathbb{CP}^1 \setminus \{0\})$ with the same boundary and asymptotic conditions and replacing $d_0$ by $d_{\infty}$. Given two intersection points $x_0,x_1 \in \phi_H(\eL_{\ul{c}_0}) \pitchfork \eL_{\ul{c}_1}$, this relative grading is written $w_{\infty}(x_0, x_1)$.

\begin{lemma}\label{l:negativeWinding}
For  any $x_0,x_1 \in \phi_H(\eL_{\ul{c}_0}) \pitchfork \eL_{\ul{c}_1}$, the relative winding grading at infinity is the opposite of the relative winding grading; that is, $w(x_0,x_1) = -w_{\infty}(x_0,x_1)$.
\end{lemma}

\begin{proof}
Let $u_{\infty} \co \mathbb{R} \times [0,1] \to \Sym^n(\mathbb{CP}^1 \setminus \{0\})$ be as above and $u \co \mathbb{R} \times [0,1] \to \Sym^n(\mathbb{C})$ be as in Definition \ref{d:winding}.
Precompose $u_{\infty}$ with multiplication by $-1$ on $\mathbb R$ so that the result has the opposite orientation and call it $\bar{u}_{\infty}$.
Since $\Sym(\ul{c}_1)$ is contractible, we can homotope $\bar{u}_{\infty}$ such that its boundary on $\Sym(\ul{c}_1)$ agrees with that of $u$.
We can then glue $\bar{u}_{\infty}$ with $u$ along the boundary to obtain a strip $\bar{u}_{\infty} \#u$ which represents a class in $\pi_2(\Sym^n(\mathbb{CP}^1), \Sym(\ul{c}_0))$.
Note that $\pi_2(\Sym^n(\mathbb{CP}^1), \Sym(\ul{c}_0))=\pi_2(\Sym^n(\mathbb{CP}^1))=\pi_2(\mathbb{CP}^n)=\mathbb{Z}$.
Since the divisor $d_0$ is homotopic to $d_{\infty}$ away from $\Sym(\ul{c}_0)$, the intersection pairing
\[
\cdot (d_o -d_{\infty}):\pi_2(\Sym^n(\mathbb{CP}^1), \Sym(\ul{c}_0)) \to \mathbb{Z}
\]
is the zero map.
Therefore, we have \[0=[\bar{u}_{\infty} \#u] \cdot (d_o -d_{\infty})=[u_{\infty}] \cdot d_{\infty} +[u] \cdot d_o\] from which the result follows.
\end{proof}

\begin{corollary} \label{cor:windingbalanced}
There is a vector space isomorphism $R \co \AKh_{symp}(L) \to \AKh_{symp}(L)$ which negates the relative winding grading.
\end{corollary}

\begin{proof}
Let $\bar{f}\co \mathbb{C}^* \to \mathbb{C}^*$ be the map $z \mapsto \frac{1}{z}$.
Let $\tau \in \Conf^{2n}(\mathbb{C})$ be chosen such that it is preserved by $\bar{f}$.
Let $f \co A_{\tau}^* \to A_{\tau}^*$ be given by $(u,v,z) \mapsto (u,v,\frac{1}{z})$. Observe that this is a symplectomorphism with respect to the restriction of $-dd^c(|u|^2+|v|^2+\log |z|)$ to $A_{\tau}^*$.
In particular, it induces an isomorphism 
\[
CF_{ann}(\cY_{n, \tau},\eL_{\ul{c}_0}, \eL_{\ul{c}_1}) \to CF_{ann}(\cY_{n,\tau}, \eL_{\bar{f}(\ul{c}_0)}, \eL_{\bar{f}(\ul{c}_1)})
\]
for appropriate choices of almost complex structures and Hamiltonian perturbations.

If we choose $\eL_{\ul{c}_0}$ and $\eL_{\ul{c}_1}$ such that $HF_{ann}(\cY_{n,\tau}, \eL_{\ul{c}_0}, \eL_{\ul{c}_1})$ computes symplectic annular Khovanov homology $\AKh_{symp}(L)$, then 
$HF_{ann}(\cY_{n, \tau}, \eL_{\bar{f}(\ul{c}_0)}, \eL_{\bar{f}(\ul{c}_1)})$ also computes $\AKh_{symp}(L)$ because the annular link does not change under applying $f$.
However, the relative winding grading of two elements in $HF_{ann}(\cY_{n, \tau}, \eL_{\ul{c}_0}, \eL_{\ul{c}_1})$ is negated under the map to $HF_{ann}(\cY_{n, \tau}, \eL_{\bar{f}(\ul{c}_0)}, \eL_{\bar{f}(\ul{c}_1)})$ by Lemma \ref{l:negativeWinding}.
\end{proof}

We may use Corollary~\ref{cor:windingbalanced} to establish an absolute winding grading: if $w(x_0, R(x_0))=2k$, then we set the absolute grading of $x_0$ to be $k$. This gives us a unique way to write the annular symplectic Khovanov homology symmetrically as a direct sum decomposition
\[\AKh_{symp}(L)\simeq\oplus_{k \in \mathbb{Z}} \AKh_{symp}(L;k)\]
such that $w(x_0,x_1)=k_0-k_1$ for $x_i \in \AKh_{symp}(L;k_i)$. 

\begin{definition}
The absolute winding grading $w(x)$ for $x \in \AKh_{symp}(L;k)$ is the integer $k$.
\end{definition}

\subsubsection{Relationship to other theories} \label{sec:motivation}

In this section we summarize the expected relationship of our annular symplectic Khovanov homology to other theories, and explain the motivation for its definition. Recall from Remark~\ref{r:char0} that, although we have taken coefficients in $\mathbb F_2$ for convenience, $\AKh_{symp}(L)$ is definable over any field.

\begin{conjecture}\label{con:isom} Over a field $\mathbb F$ of characteristic zero, there is an isomorphism
\[ \AKh_{symp}^i(L; k; \mathbb F) \simeq \oplus_{j-q=i} \AKh^{j,q}(L;k; \mathbb F) \]
in analogy with that of Abouzaid and Smith \cite{AS16, AS19} for the non-annular case. Here $i$ denotes the absolute homological grading on the symplectic side, which we will refer to as the Khovanov grading where there is any possibility of ambiguity. Meanwhile $(j,q)$ are the homological and quantum gradings on the combinatorial side, and $k$ is the winding or annular grading. If $\AKh_{symp}(L)$ is equipped with a relative quantum (equivariant) grading from a non-commutative vector field as in \cite{Cheng23}, then this isomorphism may be upgraded to an isomorphism of (relatively) trigraded groups.
\end{conjecture}

In particular, Conjecture~\ref{con:isom} implies that our theory agrees with the Mak-Smith formulation $\AKh_{symp}^{HH}(L)$ over a field of characteristic zero, which is to say wherever both are defined.

The techniques used in Sections~\ref{sec:floer} and \ref{sec:invariance} to introduce annular symplectic Khovanov homology and prove its invariance are essentially standard; it is well-known to the experts that one could in principle define a link invariant by following our recipe. However, a priori there would be no reason to believe that the result was related to combinatorial annular Khovanov homology. We now explain the motivation for Conjecture~\ref{con:isom}. Many annular link invariants including the annular Khovanov homology are governed by cylindrical Khovanov-Lauda-Rouquier-Webster algebras. It is expected that the Koszul dual of the KLRW algebra  agrees with the Fukaya algebra of a collection of basic Lagrangians, including $L_{\ul{c}_{up}}$, in $\cY_{n,\tau_o}\setminus D_o$, and the way that the Koszul dual algebra governs the annular Khovanov homology agrees with the way that the Fukaya algebra governs the symplectic annular Khovanov homology. (We are grateful to Ben Webster for early explanations helpful to understanding the conjectural correspondence.) For more on the non-cylindrical KLRW algebras, we direct the reader to \cite{KL:diagrammatic1, KL:diagrammatic2, Webster:KLRW, Rouquier:KM}.  For more on their relation to link invariants and specifically the annular case, see \cite{Webster:book}. Finally, \cite{QRS:annular} discusses the categorification of annular link invariants, in terms of an algebra expected to be Morita equivalent to the KLRW algebra. 

A related conjecture has also appeared in recent work of Aganagic, Danilenko, Li, Shende, Zhou \cite[Conjecture 1.8]{ADLSZ}. We expect that $\cY_{n,\tau_o}\setminus D_o$ is symplectomorphic to the symplectic manifold $\mathcal{M}^{\times}(\Gamma, \vec{d})$ considered in their paper. If Conjecture 1.8 of \cite{ADLSZ} 
is true, then the Fukaya algebra of the collection of Lagrangians they consider will be Koszul dual to the Fukaya algebra of the collection of basic Lagrangians we use, giving an approach towards proving our symplectic annular Khovanov homology agrees with the combinatorial version.

Alternatively, one could try to relate $\AKh_{symp}^{HH}(L)$ and $\AKh_{symp}(L)$ directly without passing through the combinatorial version. When $L$ is the annular braid closure of a braid with $n$ strands in $\mathbb{R}^2 \times [0,1]$, the symplectic manifold involved in defining $\AKh_{symp}(L)$ is $\Hilb^{n,hor}(A^*_{\tau})$ with $\tau \in \Conf^{2n}(\mathbb{C})$. 
On the other hand, the symplectic manifolds involved in defining $\AKh_{symp}^{HH}(L)$ are $\Hilb^{k,hor}(A_{\tau'})$ with  $\tau' \in \Conf^{n}(\mathbb{C})$ and $k=1,\dots,n$.
In \cite[Section 12]{MS22}, the authors define a full and faithful functor from the Fukaya-Seidel category of 
$\Hilb^{k,hor}(A_{\tau'}) \times \Hilb^{n-k,hor}(A_{\tau'})$ to the Fukaya-Seidel category of $\Hilb^{n,hor}(A_{\tau})$.
The same construction gives us a  full and faithful functor to the Fukaya-Seidel category of $\Hilb^{n,hor}(A_{\tau}^*)$.
Applying the reasoning of  \cite[Section 12]{MS22}, we obtain a spectral sequence which converges to $\AKh_{symp}(L)$ whose second page is a direct sum of Hochschild homology of bimodules. We expect that the second page is in fact exactly $\AKh_{symp}^{HH}(L)$, similarly to how the spectral sequence from $\AKh_{symp}^{HH}(L)$ to $\Kh_{symp}(L)$ is obtained in  \cite[Section 12]{MS22}, and we conjecture that the spectral sequence collapses such that the second page, so that  $\AKh_{symp}^{HH}(L)$ agrees with the last page, $\AKh_{symp}(L)$.

\subsection{Two $\mathbb{Z}/2\mathbb{Z}$-quotients}\label{ss:quotient} 

In this section, we introduce two additional symplectic annular link invariants, each of which is isomorphic to symplectic annular Khovanov homology. These will play crucial roles in the proofs of Theorems~\ref{thm:2periodic} and \ref{thm:stronglyinvertible} in Sections~\ref{s:2periodic} and \ref{s:strongly} respectively.

Let $\iota \co \mathbb{C} \to \mathbb{C}$ be the involution given by $z \mapsto -z$.
Let $\iota_A,\iota_B \co \CC^3 \to \CC^3$ be the maps
\[
\iota_A(u,v,z)=(u,v,-z) \qquad \qquad \iota_B(u,v,z)=(-u,-v,-z).
\]
Given $\tau \in \Conf^{4n,0}$ which is fixed by the action induced by $\iota$ on the configuration space, the involutions $\iota_A, \iota_B$ restrict to involutions on $A_{\tau}$. Moreover, these involutions further restrict to two free involutions on $A^*_{\tau}:=A_{\tau} \setminus \{z=0\}$. We denote the quotients of $A^*_{\tau}$ with respect to these two actions by $Q_{A,\tau}$ and $Q_{B,\tau}$ respectively.

The fibration $\pi_A \co A_{\tau}^* \to \mathbb{C}^*$ commutes with each of the involutions $\iota_A$ and $\iota_B$, so we have induced fibrations
\[
\pi_{Q,A} \co Q_{A,\tau} \to \mathbb{C}^*/\langle \iota \rangle \qquad \qquad \pi_{Q,B} \co Q_{B,\tau} \to \mathbb{C}^*/\langle \iota \rangle.
\]
By varying $\tau \in \Conf^{4n,0} \cap \Fix(\iota)$, we obtain monodromy actions 
\[
\Br_{1,2n} \to \Symp(Q_{A,\tau})/\Ham(Q_{A,\tau}) \qquad \qquad \Br_{1,2n} \to \Symp(Q_{B,\tau})/\Ham(Q_{B,\tau}).
\]
Let $\tau_{Q,o}$ denote the configuration $\{-2n,\dots,-1,1,\dots,2n\}$. We consider the subgroup $\Br_{1,n} \subseteq \Br_{1,2n}$ corresponding to treating the image of $\tau_{Q,o}$ under the quotient map as a basepoint and varying $\tau$ among configurations fixed by $\iota$ which contain $\{-2n,\dots, -n-1\} \cup \{n+1, \dots, 2n\}$ and for which the remaining points $z$ in the configuration have $Re(z) \in [-n-0.5,n+0.5]$. In particular, the only elements of $\tau_{Q,o}$ which may vary are $\{-n,\dots,-1,1,\dots,n\}$.

With respect to the fibrations $\pi_{Q,A}$ and $\pi_{Q,B}$, we can form the horizontal Hilbert schemes analogously to \eqref{eq:horHilb}, written as
\[
\Hilb^{n,hor}(Q_{A,\tau}) \qquad \qquad \Hilb^{n,hor}(Q_{B,\tau}).
\]
Where it is clear that $\tau = \tau_{Q,o}$, we omit the subscripts from the horizontal Hilbert scheme of the quotient for notational simplicity.

Let $\ul{c}_{Q,up}=\{c_1, \dots, c_n\}$ denote the set of $n$ annular matching paths on the upper half plane which match the point $k+1$ with the point $2n-k$ for $k=0,\dots,n-1$, as in Figure \ref{fig:upperQ}. Let
\[ \Sigma_{\ul{c}_{Q,up}} = \Sigma_{c_1} \times \dots \times \Sigma_{c_n} \]
be the isotropic submanifold of $\Hilb^{n,hor}(A_{\tau_{Q,o}})$ which is constructed by taking the product of the spheres $\Sigma_{c_i}$ associated to each of the $n$ paths $c_i$, where $\Sigma_{c_i}$ is the matching sphere of \eqref{eqn:sphere}. The image of $\Sigma_{\ul{c}_{Q,up}}$ under each of the quotient maps to $\Hilb^{n,hor}(Q_{A,\tau})$ and $\Hilb^{n,hor}(Q_{B,\tau})$ is then Lagrangian. Denote these resulting Lagrangians by $\Sigma_{\ul{c}_{Q,A,up}}$ and $\Sigma_{\ul{c}_{Q,B,up}}$ respectively.

\begin{figure}[ht]
\includegraphics{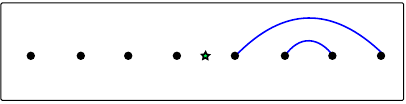}
\caption{The set of $n$ annular matching paths $\ul{c}_{Q,up}$ in the upper-half plane, drawn in blue, which avoid the origin, marked with a green star. These paths may be used to construct a Lagrangian in the horizontal Hilbert scheme of either quotient of $A_{\tau}$.}
\label{fig:upperQ}
\end{figure}

Using the induced $\Br_{1,n}$ actions on $\Hilb^{n,hor}(Q_{A,\tau})$ and $\Hilb^{n,hor}(Q_{B,\tau})$ respectively, given an annular link $L$ represented by an $n$-strand annular braid $\beta$, we may define
\[
\AKh^*_{Q,A/B}(L):=HF^{*+n+w}(\Hilb^{n,hor}(Q_{A/B,\tau}),   \Sigma_{\ul{c}_{Q,A/B,up}},\beta\Sigma_{\ul{c}_{Q,A/B,up}}).
\]
Note that the theory above is standard Lagrangian Floer cohomology over a field of characteristic two, as we have already removed the preimage of zero from $A_{\tau}$ before constructing the various symplectic manifolds involved. Here, as in Lemma \ref{l:AKHvariant} and surrounding discussion, we use the K\"ahler form $\omega_{A^*}=-dd^c(|u|^2+|v|^2+\log |z|)$ on $A_{\tau}^*$ and the resulting quotient symplectic form on $Q_{A/B,\tau}$.

The main result of this section is the following.

\begin{proposition}\label{p:annularSame}
For any annular link $L$, we have
$\AKh_{symp}(L)\simeq\AKh_{Q,A}(L)\simeq\AKh_{Q,B}(L)$.
\end{proposition}

The proof of Proposition \ref{p:annularSame} will occupy the rest of the section.

\subsubsection{$\AKh_{symp}(L)\simeq \AKh_{Q,A}(L)$} We begin by checking that for any annular link $L$, the first of our new annular invariants agrees with our annular symplectic Khovanov homology. Given $\tau=\{\pm \tau_1, \dots , \pm \tau_{2n}\}$ a configuation in $\Conf^{4n,0}(\mathbb C) \cap \Fix(\iota)$, set $\tau^2=\{\tau_1^2,\dots \tau_{2n}^2\} \in \Conf^{2n}(\mathbb C)$.

\begin{lemma}
There is a biholomorphism from  $(\mathbb{C}^2_{u,v} \times \mathbb{C}^*_z) /\langle \iota_A \rangle$ to $(\mathbb{C}^2_{u,v} \times \mathbb{C}^*_z)$
which restricts to a
biholomorphism from $Q_{A,\tau}$ to $A^*_{\tau^2}$.
\end{lemma}

\begin{proof}
Consider the holomorphic map 
\[\phi \co \mathbb{C}^2_{u,v} \times \mathbb{C}^*_z \to \mathbb{C}^2_{u,v} \times \mathbb{C}^*_z, \qquad (u,v,z) \mapsto (u,v,z^2).\]
We have $ \phi \circ \iota_A=\phi$, so it follows that $\phi$ descends to a holomorphic map 
from $(\mathbb{C}^2_{u,v} \times \mathbb{C}^*_z) /\langle \iota_A \rangle$
to $(\mathbb{C}^2_{u,v} \times \mathbb{C}^*_z)$, which is clearly a biholomorphism. Since $\phi|_{A^*_{\tau}}$ is a holomorphic map onto $A^*_{\tau^2}$, it descends to a biholomorphism from $Q_{A,\tau}$ to $A^*_{\tau^2}$.
\end{proof}

Recall the definition of Stein homotopy/deformation equivalence from Definition \ref{d:stein}.

\begin{corollary}\label{c:A=A}
$Q_{A,\tau}$ is Stein deformation equivalent to $A^*_{\tau^2}$. As a result, for any annular link $L$ we have $\AKh_{symp}(L)\simeq\AKh_{Q,A}(L)$.
\end{corollary}

\begin{proof}
Let $\omega_{A^*}= -dd^c(|u|^2+|v|^2+\log |z|)=-dd^c(\rho)$ be the standard symplectic form on $\mathbb{C}^2_{u,v} \times \mathbb{C}^*_z$.
For $t \in [0,1]$, let $\rho_t:=(1-t)\rho+t\phi^*\rho$.
Since any convex linear combination of exhausting plurisubharmonic functions is an  exhausting plurisubharmonic function,
the family $\rho_t$ gives a family of $\iota_A$-invariant Stein structures on $\mathbb{C}^2_{u,v} \times \mathbb{C}^*_z$.
Let
 \[\omega_t=(1-t)\omega_{A^*}+t \phi^* \omega_{A^*}=-dd^c ((1-t)\rho+t\phi^*\rho)\] 
be the associated family of K\"ahler forms.
It is routine to check that the critical points of $\rho_t$ do not go to infinity so $(\mathbb{C}^2_{u,v} \times \mathbb{C}^*_z, \omega_0)$ is $\iota_A$-equivariantly Stein homotopy equivalent to $(\mathbb{C}^2_{u,v} \times \mathbb{C}^*_z, \omega_1)$.
The space $Q_{A,\tau}$ is a complex hypersurface of the $\iota_A$-quotient of $\mathbb{C}^2_{u,v} \times \mathbb{C}^*_z$.
Restricting the Stein homotopy to $Q_{A,\tau}$  shows that $(Q_{A,\tau}, \omega_0|_{Q_{A,\tau}})$
is Stein homotopy equivalent to $(Q_{A,\tau}, \omega_1|_{Q_{A,\tau}})$, which is isomorphic to $(A^*_{\tau^2}, \omega_{A^*}|_{A^*_{\tau^2}})$ via $\phi$.
Therefore, $(Q_{A,\tau}, \omega_0|_{Q_{A,\tau}})$, which is equipped with the Stein structure descended from 
$(\mathbb{C}^2_{u,v} \times \mathbb{C}^*_z, \omega_{A^*})$, is Stein deformation equivalent  to $(A^*_{\tau^2}, \omega_{A^*}|_{A^*_{\tau^2}})$.

 Finally, we want to compare the Lagrangian matching spheres in $(Q_{A,\tau},\omega_t|_{Q_{A,\tau}})$ for $t \in [0,1]$
with respect to the Lefschetz fibration to $\mathbb{C}^*_z/(z=-z)$.
For each $t_0, t_1 \in [0,1]$, the Stein homotopy from $t=t_0$ to $t=t_1$ induces a symplectomorphism $\psi_{t_0,t_1}$ from $\omega_{t_0}$ to $\omega_{t_1}$.
For any matching path $c$, denote the Lagrangian matching sphere with respect to $\omega_t$ by $\cL_{c,t}$.
When $t_0$ is close to $t_1$, $\psi_{t_0,t_1}(\cL_{c,t_0})$ is $C^1$ close to $\cL_{c,t_1}$.
In particular, $\psi_{t_0,t_1}(\cL_{c,t_0})$ is a graph with respect to a Weinstein neighborhood of $\cL_{c,t_1}$.
Therefore we see that $\psi_{t_0,t_1}(\cL_{c,t_0})$ is Lagrangian isotopic to $\cL_{c,t_1}$ when $t_0$ is close to $t_1$, and hence in fact Hamiltonian isotopic since spheres are simply connected\footnote{In fact, the nearby Lagrangian conjecture for $T^*S^2$ is known \cite{Hind} so to conclude that $\psi_{t_0,t_1}(\cL_{c,t_0})$ is Hamiltonian isotopic to $\cL_{c,t_1}$, we only need that $\psi_{t_0,t_1}(\cL_{c,t_0})$ is $C^0$ close to  $\cL_{c,t_1}$ when $t_0$ is close to $t_1$.}.
By the compactness of $[0,1]$, we know that $\psi_{0,1}(\cL_{c,0})$ is Hamiltonian isotopic to $\cL_{c,1}$.
Since the matching spheres of $(Q_{A,\tau},\omega_1|_{Q_{A,\tau}})$ can be identified with the matching spheres of $(A^*_{\tau^2}, \omega_{A^*}|_{A^*_{\tau^2}})$ via $\phi$, this proves that $\AKh_{symp}(L)\simeq \AKh_{Q,A}(L)$.

\end{proof}

\subsubsection{$\AKh_{Q,A}(\kappa)\simeq \AKh_{Q,B}(\kappa)$} We now check that the two new annular link invariants of this subsection agree.

\begin{lemma}\label{l:biholoAB}
There is a biholomorphism from $(\mathbb{C}^2_{u,v} \times \mathbb{C}^*_z) /\langle \iota_A \rangle$
to $(\mathbb{C}^2_{u,v} \times \mathbb{C}^*_z) /\langle \iota_B \rangle$ which restricts to a
biholomorphism from $Q_{A,\tau}$ to $Q_{B,\tau}$.
\end{lemma}

\begin{proof}
It is easier to define the biholomorphism after a change of coordinates.
Let $x=u-iv$ and $y=u+iv$ so that $u^2+v^2=xy$.
With respect to these new coordinates, we have 
\[
\iota_A(x,y,z)=(x,y,-z), \qquad \qquad \iota_B(x,y,z)=(-x,-y,-z).\]
Moreover, with respect to these coordinates we have
\[A^*_{\tau} = \{xy=(z-\tau_1) \dots (z-\tau_{2n})(z+\tau_1) \dots (z+\tau_{2n})\}.\]
Consider the biholomorphic map 
\[\phi \co \mathbb{C}^2_{x,y} \times \mathbb{C}^*_z \to \mathbb{C}^2_{x,y} \times \mathbb{C}^*_z, \qquad \qquad (x,y,z) \mapsto (xz,yz^{-1},z).\]
Now $ \phi \circ \iota_A=\iota_B \circ \phi$, hence $\phi$ descends to a biholomorphism 
from $(\mathbb{C}^2_{x,y} \times \mathbb{C}^*_z) /\langle \iota_A \rangle$
to $(\mathbb{C}^2_{x,y} \times \mathbb{C}^*_z) /\langle \iota_B \rangle$. As $\phi|_{A^*_{\tau}}$ is a biholomorphism onto $A^*_{\tau}$, it also descends to a biholomorphism from $Q_{A,\tau}$ to $Q_{B,\tau}$.
\end{proof}

\begin{corollary}\label{c:A=B}
$Q_{A,\tau}$ is Stein deformation equivalent to $Q_{B,\tau}$. As a result, for any annular link $L$, we have $\AKh_{Q,A}(L)\simeq\AKh_{Q,B}(L)$.
\end{corollary}

\begin{proof}
We may repeat the proof of Corollary \ref{c:A=A} with the map $\phi$ of that proof replaced by the map $\phi$ of Lemma \ref{l:biholoAB}.
\end{proof}

\begin{proof}[Proof of Proposition~\ref{p:annularSame}] This follows from combining the isomorphisms of Corollaries~\ref{c:A=A} and \ref{c:A=B}.
\end{proof}

\subsection{Manolescu's reformulation, bridge diagrams, and actions} \label{subsec:bridge}

In this subsection we review important features of Manolescu's reformulation \cite{Manolescu:nilpotent} of symplectic Khovanov homology, already introduced in Section~\ref{sec:geometry}, and its relationship to bridge diagrams and actions, and discuss their generalization to the annular case. Note that \cite{Manolescu:nilpotent} deals only with flattened braid diagrams, but it was checked by Waldron \cite[Section 4.2]{Waldron} that the construction works for general bridge diagrams. His arguments apply straightforwardly to the annular case.

Recall that a bridge diagram is a decomposition of a link into the union of two trivial tangles. For our purposes, the notation for bridge diagrams will be as follows. Let $\tau = \{\tau_1, \dots, \tau_{2n}\}$ be such that each $\tau_i$ lies on the real line and we have $\tau_i > \tau_{i+1}$ for $1 \leq i \leq 2n$. Our usual basepoint $\tau_o$ is an example of this. We will usually further require that there is some $j$ such that $\tau_{2j} > 0 > \tau_{2j+1}$, which is to say there are an even number of negative $\tau_i$ and an even number of positive $\tau_i$. Consider the crossingless matching $\ul{b}$ consisting of line segments $b_i$ on the real line joining $\tau_{2i-1}$ to $\tau_{2i}$ for $i=1,\dots,n$, which is annular if additional requirement on the configuration is imposed; these line segments are the \emph{bridges}. We form a link diagram by adding a second crossingless annular matching $\ul{a}$ consisting of $n$ matching paths between the elements $\tau_i$ in $\tau$, insisting that the interiors of the matching paths intersect transversely, the $a_i$ undercross the $b_i$, and no $a_i$ intersects the origin. The result is an (annular) link diagram for a link $L$. For examples, see Figures~\ref{fig:periodic_bridge}, \ref{fig:periodic_bridge_quotient}, \ref{fig:stronginv}, \ref{fig:periodicresolutions}, \ref{fig:zero}, and \ref{fig:one}.

Recall that we may compute symplectic Khovanov homology in $\Hilb^{n,hor}(A_{\tau})$ using the restriction of the K\"ahler form $\omega_{\cY}$, with Lagrangians which are products of spheres. In particular, if given a matching path $c$, $\Sigma_{c}$ is the Lagrangian sphere in $A_{\tau}$ defined in \eqref{eqn:sphere}, then we have Lagrangians
\begin{align*} 
\Sigma_{\beta} &= \Sigma_{b_1} \times \dots \times \Sigma_{b_n} \\
\Sigma_{\alpha} &= \Sigma_{a_1} \times \dots \times \Sigma_{a_n}.
\end{align*}
With this in mind one then has $\Kh_{symp}(L) \simeq HF(\Hilb^{n,hor}(A_{\tau}), \Sigma_{\alpha}, \Sigma_{\beta})$. In the annular case, we treat $L$ as an annular link such that the origin is the intersection of the annular axis with $\mathbb C$.  We may as previously let $A_{\tau}^* = A_{\tau} \setminus \{z=0\}$, such that $D_o$ is exactly the complement of $\Hilb^{n,\tau}(\Sigma_{\alpha}, \Sigma_{\beta})$. The intersection counting theory $HF(\Hilb^{n,hor}(A_{\tau}), \Sigma_{\alpha}, \Sigma_{\beta}; \mathbb F_2[U])$ agrees with that constructed previously, such that its annular truncation $HF_{ann}(\Hilb^{n,hor}(A_{\tau}), \Sigma_{\alpha}, \Sigma_{\beta})$ is again annular symplectic Khovanov homology. If we delete the divisor and compute with the symplectic form of Lemma~\ref{l:AKHvariant}, we equivalently obtain \[\AKh_{symp}(L) \simeq HF(\Hilb^{n,hor}(A_{\tau}^*), \Sigma_{\alpha}, \Sigma_{\beta}).\]

The manifolds involved in symplectic Khovanov homology carry an $O(2)$ action, which is especially easy to see from this perspective. The action of $O(2)$ on $(u,v) \subset \mathbb C^2$ induces a symplectic action on $A_{\tau}$, which fixes the Lagrangian sphere $\Sigma_{c}$ for any matching path $c$. The induced action on the Hilbert scheme preserves $\Hilb^{n,hor}(A_{\tau})$ and indeed $\Hilb^{n,hor}(A_{\tau}^*)$ \cite[Lemma 7.3]{HLS:flexible}, and fixes the Lagrangians. A candidate $O(2)$-equivariant symplectic Khovanov homology has been constructed by 
the first author, Lipshitz, Sarkar \cite{HLS:simplicial}.

Finally, note that $\Sigma_{\beta}$ and $\Sigma_{\alpha}$ intersect in a circle above any intersection between a path $b_i$ and a path $a_j$. Typically one chooses a Hamiltonian perturbation of $\Sigma_{\beta}$ such that this intersection becomes finite and consists of the generators for a Morse complex for the circle. In most cases, it suffices to use a complex with one generator in degree zero and one generator in degree one. If $\tau=\tau_o$, this perturbation may be carried out so as to result in a Lagrangian $\Sigma_{\beta}'$ which is preserved by the action induced by $\iota_A(u,v,z) = (u,v,-z)$ on the Hilbert scheme, as in \cite[Lemma 7.5]{HLS:flexible}. If we instead are interested in the involution $\iota_B(u,v,z) = (-u,-v,-z)$, we must instead choose a Morse complex with two generators in each of degrees zero and one, which again allows us to carry out the perturbation such that $\Sigma_{\beta}'$ remains preserved by the induced action on the Hilbert scheme. In either case, this explicit choice of perturbation makes it possible to list the generators of a chain complex computing $\Kh_{symp}$ or $\AKh_{symp}$ starting from a bridge diagram of the link, and to use the tautological correspondence of Remark~\ref{r:tauto} to restrict the list of possible differentials on the complex.

\section{Symplectic annular Khovanov homology and link Floer homology in the branched cover} \label{sec:knotfloer}

In this section we prove Theorem~\ref{thm:knotfloer}, establishing a symplectic version the Roberts' annular refinement of the Ozsv{\'a}th-Szab{\'o} spectral sequence from Khovanov homology to Heegaard Floer homology. Toward this end, we recall the relationship of Manolescu's reformulation of symplectic Khovanov homology with Heegaard Floer homology \cite{Manolescu:nilpotent, SS10}.

Let $L$ be an annular link with annular axis $A$. Consider an annular bridge diagram for $L$ with notation as in Section~\ref{subsec:bridge}, so that the endpoints of the bridges lie at points of $\tau = \{\tau_1, \dots \tau_{2n}\} \in \Conf^{2n,0}(\CC)$ such that $\tau_{2i-1}$ and $\tau_{2i}$ are connected by the bridge $b_i$ along the real axis in $\CC$ not crossing the origin, and the undercrossing arcs are labelled $a_i, \dots, a_n$ in some order. We may form the Riemann surface 
\[\Sigma'_{\tau} = \{(u,0,z) \in \CC^3 : u^2+p(z) = 0\} \subseteq A_{\tau}\]
which is the branched cover of $\CC \cup \{\infty\}$ of $\CC$ over the $2n$ points in $\tau$. We may compatify $\Sigma_{\tau}'$ to a surface $\Sigma_{\tau}$ by filling in the two punctures at infinity to obtain a closed surface of genus $n-1$. We mark the the two preimages of $\{\infty\}$ with basepoints labeled $w_1$ and $w_2$ and the two preimages of $\{0\}$ with basepoints labeled $z_1$ and $z_2$, so that $\Sigma_{\tau} \setminus \{w_1, w_2, z_1, z_2\}$ is the double branched cover of $\CC \setminus \{0\}$ over the set of $2n$ points $\tau$. Note that Manolescu calls this surface $\Sigma_0$; to avoid confusion in the annular context, we use $\Sigma_{\tau}$. 

\begin{figure}
\scalebox{.5}{
\begingroup%
  \makeatletter%
  \providecommand\color[2][]{%
    \errmessage{(Inkscape) Color is used for the text in Inkscape, but the package 'color.sty' is not loaded}%
    \renewcommand\color[2][]{}%
  }%
  \providecommand\transparent[1]{%
    \errmessage{(Inkscape) Transparency is used (non-zero) for the text in Inkscape, but the package 'transparent.sty' is not loaded}%
    \renewcommand\transparent[1]{}%
  }%
  \providecommand\rotatebox[2]{#2}%
  \newcommand*\fsize{\dimexpr\f@size pt\relax}%
  \newcommand*\lineheight[1]{\fontsize{\fsize}{#1\fsize}\selectfont}%
  \ifx\svgwidth\undefined%
    \setlength{\unitlength}{410.16723633bp}%
    \ifx\svgscale\undefined%
      \relax%
    \else%
      \setlength{\unitlength}{\unitlength * \real{\svgscale}}%
    \fi%
  \else%
    \setlength{\unitlength}{\svgwidth}%
  \fi%
  \global\let\svgwidth\undefined%
  \global\let\svgscale\undefined%
  \makeatother%
  \begin{picture}(1,0.5440313)%
    \lineheight{1}%
    \setlength\tabcolsep{0pt}%
    \put(0,0){\includegraphics[width=\unitlength,page=1]{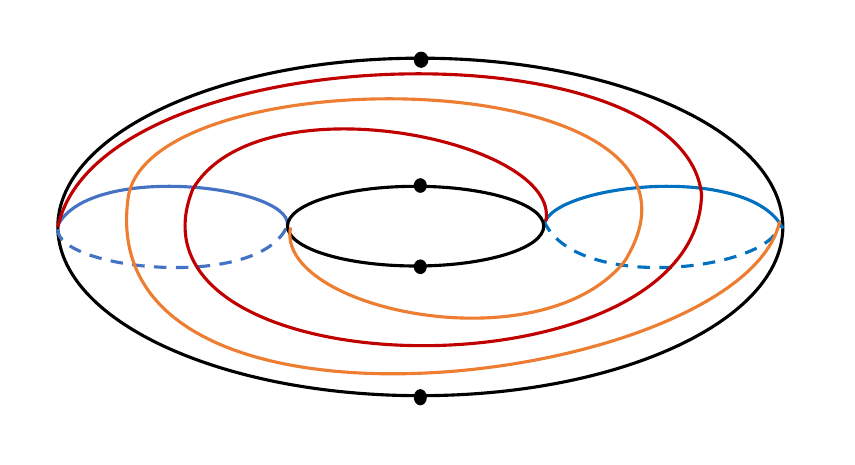}}%
    \put(0.67344937,0.27801881){\color[rgb]{0,0,0}\makebox(0,0)[lt]{\lineheight{1.25}\smash{\begin{tabular}[t]{l}$\widehat{\beta}_1$\end{tabular}}}}%
    \put(0.82385304,0.35995512){\color[rgb]{0,0,0}\makebox(0,0)[lt]{\lineheight{1.25}\smash{\begin{tabular}[t]{l}$\widehat{\alpha}_1$\end{tabular}}}}%
    \put(0.46243405,0.29504111){\color[rgb]{0,0,0}\makebox(0,0)[lt]{\lineheight{1.25}\smash{\begin{tabular}[t]{l}$z_1$\end{tabular}}}}%
    \put(0.45794557,0.24771359){\color[rgb]{0,0,0}\makebox(0,0)[lt]{\lineheight{1.25}\smash{\begin{tabular}[t]{l}$z_2$\end{tabular}}}}%
    \put(0.45794552,0.50250193){\color[rgb]{0,0,0}\makebox(0,0)[lt]{\lineheight{1.25}\smash{\begin{tabular}[t]{l}$w_1$\end{tabular}}}}%
    \put(0.45457837,0.04343388){\color[rgb]{0,0,0}\makebox(0,0)[lt]{\lineheight{1.25}\smash{\begin{tabular}[t]{l}$w_2$\end{tabular}}}}%
  \end{picture}%
\endgroup%
}
\caption{The double branched cover of the annular bridge diagram shown in Figure \ref{fig:periodic_bridge} along the endpoints of the bridges. Only the half of each of the red and orange curves $\widetilde{a}_1$ and $\widetilde{a}_2$ on the top side of the Heegaard surface facing the reader is shown; the remaining portions of the curve may be obtained by symmetry. The branched covering action is the hyperelliptic involution.}
\label{fig:heegaard_diagram_lift}
\end{figure}
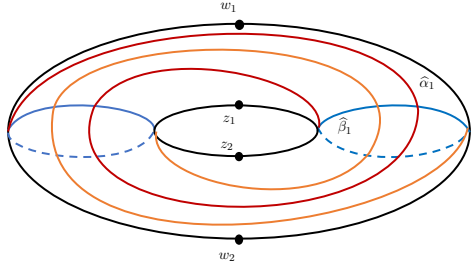

The preimage of each bridge $b_i$ is a simple closed curve $\widehat{\beta}_i$ on $\Sigma_{\tau}$, given by 
\[\widehat{\beta}_i = \{ (u,0,z): z \in b_i, u \in \pm \sqrt{p(z)} \}.\]
Let $\widehat{\betas}$ denote this set of $n$ non-intersecting curves. Likewise, the preimage of each $a_i$ is a simple closed curve $\widehat{\alpha}_i$ and we may consider the set $\widehat{\alphas}$. An example of an annular bridge diagram appears in Figure~\ref{fig:periodic_bridge}; the corresponding surface $\Sigma_{\tau}$ together with the decorations we have introduced is shown in Figure~\ref{fig:heegaard_diagram_lift}.

Manolescu \cite[Proposition 7.4]{Manolescu:nilpotent} shows that $\mathcal H' = (\Sigma_0, \widehat{\alphas}, \widehat{\betas}, \{w_1, w_2\})$ is a multipointed Heegaard diagram for $\Sigma(L)$. Since $\Sigma(mL) \simeq -\Sigma(L)$, it follows that that using the cohomologial conventions of this paper the Lagrangian Floer cochain complex $CF(\Sym^n(\Sigma_0 \backslash \{w_1,w_2\}), \mathbb{T}_\alphas, \mathbb{T}_\betas)$ computes the Heegaard Floer cohomology \[\HFhatDual(\Sigma(L) \# (S^1\times S^2)) \simeq \HFhatDual(\Sigma(L))\otimes H^1(S^1) \simeq \HFhat(\Sigma(mL))\otimes V.\]  

We claim that the set of data $\mathcal H = (\Sigma_{\tau}, \widehat{\alphas}, \widehat{\betas}, \{w_1, w_2, z_1, z_2\})$ is a multi-pointed Heegaard diagram for $(\Sigma(L), \widetilde{A})$, where as in the introduction $\widetilde{A}$ denotes the lift of the annular axis $A$ under the map $\Sigma(L) \rightarrow S^3$. We must first check that each component of $\Sigma_{\tau} \setminus \widehat{\alphas}$ contains exactly one $w_i$ and one $z_i$, and similarly for $\Sigma_{\tau} \setminus \widehat{\betas}$. This follows because both sets of curves $\widehat{\alphas}$ and $\widehat{\betas}$ cut the surface into two components interchanged by the branched covering involution. Since the points $w_1$ and $w_2$ are interchanged by the branched covering involution, they must lie in different components in $\Sigma_{\tau} \setminus \widehat{\alphas}$ and also in $\Sigma_{\tau} \setminus \widehat{\betas}$; the same is true of the points $z_1$ and $z_2$.

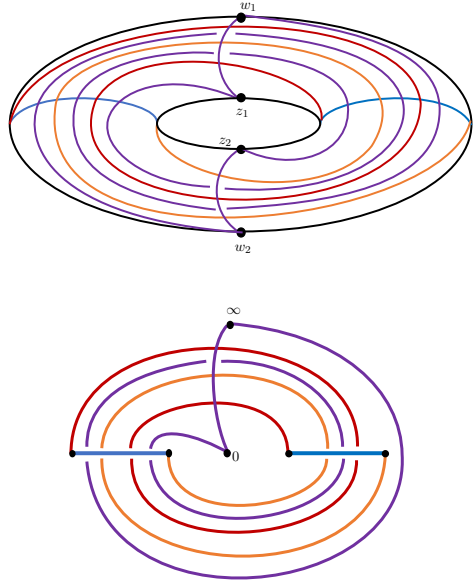
\begin{figure}
\scalebox{.5}{
\begingroup%
  \makeatletter%
  \providecommand\color[2][]{%
    \errmessage{(Inkscape) Color is used for the text in Inkscape, but the package 'color.sty' is not loaded}%
    \renewcommand\color[2][]{}%
  }%
  \providecommand\transparent[1]{%
    \errmessage{(Inkscape) Transparency is used (non-zero) for the text in Inkscape, but the package 'transparent.sty' is not loaded}%
    \renewcommand\transparent[1]{}%
  }%
  \providecommand\rotatebox[2]{#2}%
  \newcommand*\fsize{\dimexpr\f@size pt\relax}%
  \newcommand*\lineheight[1]{\fontsize{\fsize}{#1\fsize}\selectfont}%
  \ifx\svgwidth\undefined%
    \setlength{\unitlength}{394.11373901bp}%
    \ifx\svgscale\undefined%
      \relax%
    \else%
      \setlength{\unitlength}{\unitlength * \real{\svgscale}}%
    \fi%
  \else%
    \setlength{\unitlength}{\svgwidth}%
  \fi%
  \global\let\svgwidth\undefined%
  \global\let\svgscale\undefined%
  \makeatother%
  \begin{picture}(1,1.17718937)%
    \lineheight{1}%
    \setlength\tabcolsep{0pt}%
    \put(0,0){\includegraphics[width=\unitlength,page=1]{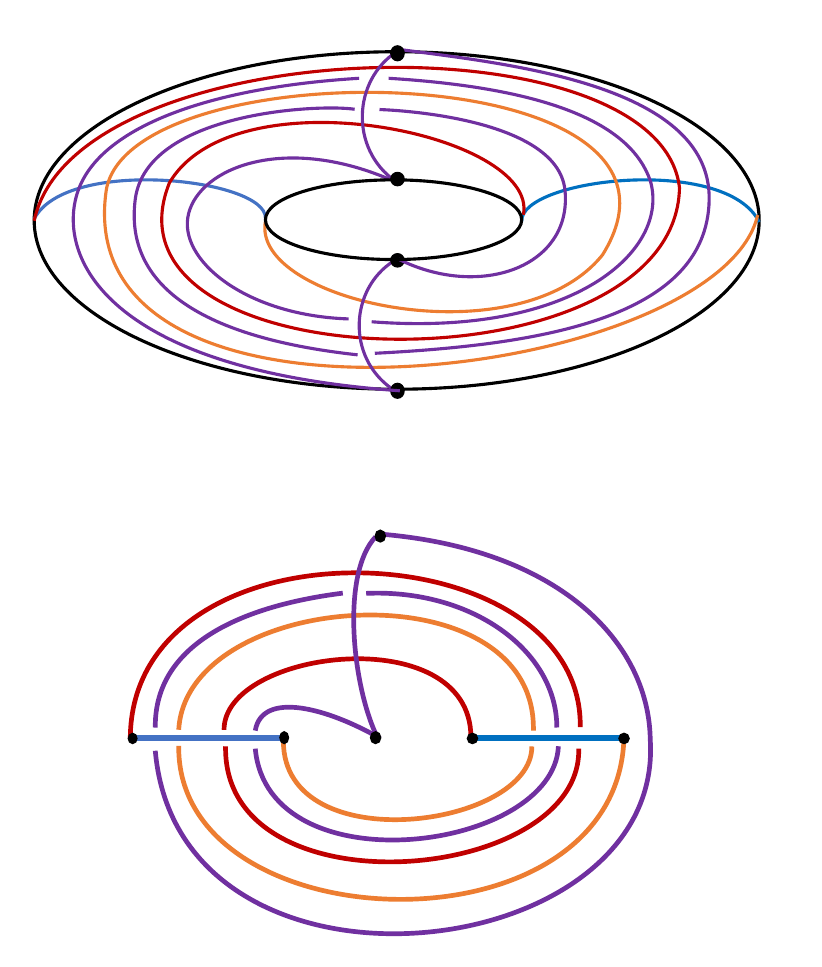}}%
    \put(0.45612086,0.54576426){\color[rgb]{0,0,0}\makebox(0,0)[lt]{\lineheight{1.25}\smash{\begin{tabular}[t]{l}$\infty$\end{tabular}}}}%
    \put(0.46725405,0.26370763){\color[rgb]{0,0,0}\makebox(0,0)[lt]{\lineheight{1.25}\smash{\begin{tabular}[t]{l}$0$\end{tabular}}}}%
    \put(0.48243984,1.12929582){\color[rgb]{0,0,0}\makebox(0,0)[lt]{\lineheight{1.25}\smash{\begin{tabular}[t]{l}$w_1$\end{tabular}}}}%
    \put(0.47192663,0.66671431){\color[rgb]{0,0,0}\makebox(0,0)[lt]{\lineheight{1.25}\smash{\begin{tabular}[t]{l}$w_2$\end{tabular}}}}%
    \put(0.47426287,0.92604029){\color[rgb]{0,0,0}\makebox(0,0)[lt]{\lineheight{1.25}\smash{\begin{tabular}[t]{l}$z_1$\end{tabular}}}}%
    \put(0.44155511,0.87113795){\color[rgb]{0,0,0}\makebox(0,0)[lt]{\lineheight{1.25}\smash{\begin{tabular}[t]{l}$z_2$\end{tabular}}}}%
  \end{picture}%
\endgroup%
}
\caption{The knot defined by the Heegaard diagram shown in Figure \ref{fig:heegaard_diagram_lift}, with its projection to the bridge diagram as the axis of the annular link shown in purple. In this diagram, we have drawn the knot upstairs with regard for visibility, such that it is not preserved by the branch covering action, but it may easily be redrawn as exactly the lift of the purple axis downstairs.}
\label{fig:axis_lift}
\end{figure}

This is sufficient to establish that $\mathcal H$ is indeed a Heegaard diagram for some link $J$ in $\Sigma(L)$. We must now show that the link in question is $\widetilde{A}$, the lift of the annular axis. Recall that we reconstruct $J$ by first connecting the $w$ basepoints to the $z$ basepoints in $\Sigma_{\tau} \setminus \widehat{\betas}$, and then connecting the $z$ basepoints to the $w$ basepoints in $\Sigma_{\tau} \setminus \widehat{\alphas}$, undercrossing any arcs from the first step which we encounter. In particular, in each of these steps we may choose to draw the two arcs in question such that they are interchanged by the branched covering involution on $\Sigma_{\tau}$. This produces either a knot or a two-component link; in either case, projecting to $\CC$ produces a knot with a diagram on $\CC$ consisting of two arcs, one running from $\infty$ to $0$ and the second from $0$ to $\infty$. The first of these arcs does not intersect the blue bridges $b_i$, and may be chosen to be the straight line segment shown in Figure \ref{fig:axis_lift}; the other does not intersect the red arcs $a_i$, and may be chosen to follow any path from $0$ to $\infty$ avoiding those arcs, undercrossing the vertical line segment as it passes it. This is the annular axis $A$. Its lift $\widetilde{A}$ in $\Sigma(L)$ is a knot if $\ell k(A, L)$ is odd and is a two-component link if $\ell k(A, L)$ is even. We therefore see that if $\ell k(A, L)$ is odd we have 
\begin{align*}
H_*(CF(\Sym^n(\Sigma_{\tau} \backslash \{w_1, w_2, z_1, z_2\}), \mathbb T_{\widehat{\alphas}}, \mathbb T_{\widehat{\betas}})) & \simeq \HFKhatDual(\Sigma(L)\# (S^1\times S^2),\widetilde{A}) \\
& \simeq \HFKhatDual(\Sigma(L), \widetilde{A}) \otimes H^*(S^1) \\
& \simeq \HFKhat(\Sigma(mL), m\widetilde{A}) \otimes V. \end{align*}
Here we again use the fact that $-\Sigma(L)=\Sigma(mL)$. Furthermore, note that $m\widetilde{A}$ is the same as the lift of the annular axis $mA$ of the annular link $mL$ to its branched double cover $\Sigma(mL)$. If instead $\ell k (A, L)$ is even, we have
\begin{align*}H_*(CF(\Sym^n(\Sigma_{\tau} \backslash \{w_1, w_2, z_1, z_2\}), \mathbb T_{\widehat{\alphas}}, \mathbb T_{\widehat{\betas}}) &\simeq \HFLhatDual(\Sigma(L), \widetilde{A}) \\
&\simeq \HFLhat(\Sigma(mL),m\widetilde{A}). 
\end{align*}
We now recall from \cite[Proposition 7.1]{Manolescu:nilpotent} that the \emph{anti-diagonal} divisor in $\Sym^n(\Sigma_{\tau})$ is the set
\[ \nabla =  \left \{ \{(u_1,0,z_1), \dots, (u_n, 0, z_n)\} : u_i^2+p_(z)=0, \ \exists \  i, j \text{ such that } z_i = z_j \text{ and } u_i = -u_j \right\} .\]
There is a filtration $\mathcal F$ on the chain complex $CF(\Sym^n(\Sigma_{\tau} \backslash \{w_1,w_2\}), \mathbb T_{\widehat{\alphas}}, \mathbb{T}_{\widehat{\betas}})$ called the \emph{anti-diagonal filtration}, coming from counting intersections of strips in the differential with $\nabla$. We denote the levels of this filtration by $\mathcal F_i$. There exists a combinatorial way to fix the absolute filtration level $i$; for discussion of this, see \cite[Introduction and Section 4]{Tweedy:anti-diagonal}. Deleting the anti-diagonal divisor from the symmetric product and suitably perturbing the symplectic form to be convex at infinity returns the homology of the associated graded of the filtration in the usual way; that is, we have
\[
H_*(CF(\Sym^n(\Sigma_{\tau} \backslash \{w_1,w_2\}) \setminus \nabla, \mathbb T_{\widehat{\alphas}}, \mathbb{T}_{\widehat{\betas}})) \simeq \oplus_{i \in \ZZ} H_*(\mathcal F_i/\mathcal F_{i+1}).  
\] 

Manolescu considers the involution induced by the map $(u,v,z) \mapsto (u,-v,z)$ on the horizontal Hilbert scheme. When we need to distinguish this involution from others, we will refer to it as $\iota_D(u,v,z)=(u,-v,z)$. Here the choice of subscript is meant to remind the reader that this is the involution related to the branched \emph{double} cover. He shows that we have fixed sets
\[ ((\Hilb^{n,hor}(A_{\tau}))^{fix}, \Sigma_{\alpha}^{fix}, \Sigma_{\beta}^{fix}) = (\Sym^{n}(\Sigma_{\tau} \setminus \{w_1,w_2\}) \setminus \nabla, \mathbb T_{\widehat{\alphas}}, \mathbb T_{\widehat{\betas}}).\]
Seidel and Smith checked that this example admits a stable normal trivialization \cite[Lemma 31]{SS10}; their proof is reviewed in Section \ref{sec:normal-trivs}. As a consequence \cite[Corollary 34]{SS10}, using their localization theorem for $\mathbb Z/2\mathbb Z$ Lagrangian Floer cohomology, recapped in this paper as Theorem~\ref{thm:localization} and Corollary~\ref{c:localization}, they prove there is a spectral sequence
\begin{align} \label{eq:ssdbc} \Kh_{symp}(L) \otimes \FF_2[\theta, \theta^{-1}] \rightrightarrows \oplus_{i \in \ZZ} H_*(\mathcal F_i / \mathcal F_{i+1}) \otimes \FF_2 [\theta, \theta^{-1}].\end{align}
This is the symplectic analog of the Ozsv{\'a}th-Szab{\'o} spectral sequence from Khovanov homology to Heegaard Floer homology. Combined with the spectral sequence from $H_*(\mathcal F_i / \mathcal F_{i+1})$ to $\HFhatDual(\Sigma(L))\otimes H^*(S^1)\simeq \HFhat(\Sigma(mL))\otimes V$, this gives a dimension inequality
\[ \dim(\Kh_{symp}(L)) \geq 2\dim(\HFhat(\Sigma(L))).\]
The anti-diagonal filtration and the spectral sequence it induces were further studied by Tweedy \cite{Tweedy:anti-diagonal, Tweedy:anti-diagonal-reduced}, who showed among other results that the spectral sequence collapses on the $E_1$-page for two-bridge knots \cite[Theorem 1.0.14]{Tweedy:anti-diagonal-reduced}; that is, for two-bridge knots we have
\[ H_*(\mathcal F_i / \mathcal F_{i+1}) \simeq \widehat{HF}(\Sigma(K)). \]
A conjecture of Seidel and Smith \cite[Section 4e]{SS10}, echoed by Tweedy \cite[Conjecture 1.0.15]{Tweedy:anti-diagonal-reduced}, states that the spectral sequence in fact collapses immediately for every knot. As far as the authors are aware this remains open even for alternating knots.

We are now ready to prove Theorem~\ref{thm:knotfloer}.  

\begin{proof}[Proof of Theorem~\ref{thm:knotfloer}.] Let $\mathcal F_i$ be the anti-diagonal filtration on $CF(\Sym^n(\Sigma_{\tau} - \{w_1,w_2\}), \mathbb T_{\widetilde{a}}, \mathbb T_{\widetilde{b}})$ coming from counting intersections of pseudoholomorphic curves in the differential with $\nabla$, as above. Let 
\[\nabla_o = \nabla \cap \Sym^n(\Sigma_0 \setminus \{w_1, w_2, z_1, z_2\}))\]
be the \emph{annular anti-diagonal divisor}. The filtration $\mathcal F_i$ induces a filtration on the Lagrangian Floer cochain complex $CF(\Sym^n(\Sigma_{\tau} \setminus \{w_1,w_2, z_1, z_2\}), \mathbb T_{\widehat{\alphas}}, \mathbb T_{\widehat{\betas}})$. We denote this induced filtration by $\mathcal G_i$ and refer to it as the \emph{annular anti-diagonal filtration}. It is explicitly computed by counting intersections of strips in the differential with $\nabla_o$.

Now, the fixed set of the action induced by $\iota_D(u,v,z) = (u,-v,z)$ on the horizontal Hilbert scheme $\Hilb^{n,hor}(A_{\tau}^*) = \Hilb^{n,hor}(A_{\tau}) \setminus D_o$ is precisely the complement of $\nabla_o$ in $\Sym^n(\Sigma_0 \setminus \{w_1, w_2, z_1, z_2\})$, and the fixed sets of the Lagrangians are unchanged from the non-annular case. This set of manifolds and Lagrangians inherits a stable normal trivialization from the non-annular case by restricting the vector bundle isomorphisms to the complement of $D_o \cap \nabla_o$. There is therefore by Corollary~\ref{c:localization} a spectral sequence whose first page is
\[HF(\Hilb^{n,hor}(A_{\tau}^*), \Sigma_{\alpha}, \Sigma_{\beta}) \simeq \AKh_{symp}(K) \otimes \mathbb F_2[\theta, \theta^{-1}]\]
and whose final page is isomorphic to 
\[ H_*(\mathcal G_i/\mathcal G_{i+1}) \otimes \mathbb F_2[\theta, \theta^{-1}].\]
There is then a spectral sequence from $\oplus_{i \in \ZZ} H_*(\mathcal G_i/\mathcal G_{i+1})$ to $HF(\Sym^{n}(\Sigma_{\tau} \setminus \{w_0,w_1,z_0,z_1\}), \mathbb T_{\widehat{\alphas}}, \mathbb T_{\widehat{\betas}})$. By the discussion above, this is isomorphic to either
\begin{align*}
\HFKhatDual(\Sigma(L), \widetilde{A}) \otimes H^*(S^1) \simeq \HFKhat(\Sigma(mL), m\widetilde{A}) \otimes V \end{align*}
if $\ell k(L, A)$ is odd and 
\begin{align*}
\HFLhatDual(\Sigma(L), \widetilde{A}) \simeq \HFLhat(\Sigma(mL), m\widetilde{A}) \end{align*}
if $\ell k (L,A)$ is even. The claimed rank inequalities follow.
\end{proof} 

We now consider Corollary~\ref{cor:two-bridge}. Let $K$ be a two-bridge knot with its natural 2-periodic symmetry, and choose the annular axis $A$ to be the axis of periodicity. Figure~\ref{fig:periodic_bridge} is an example of this choice. We recall to the reader the following well-known lemma.
\begin{lemma} \label{lemma:interchange} If $K$ is two-bridge and $A$ is its axis of periodicity, then the two lifts $\widetilde{K}$ and $\widetilde{A}$ are isotopic in the double-branched cover $\Sigma(K)$. \end{lemma} 

\begin{proof} If $K$ is two-bridge and $A$ is its axis of periodicity, then the quotient $\overline{K}$ together with its axis $\overline{A}$ is a link of two unknots, and $(S^3, K \cup A)$ is the branched double cover of $(S^3, \overline{K}\cup\overline{A})$ over $\overline{A}$. There exists an isotopy interchanging the components, for example by viewing $\overline{K}$ as a one-bridge knot and arranging $\overline{A}$ such that rotation around the origin in the plane of the page is the desired isotopy; see Figure~\ref{fig:interchange} for an example. It is well-known that the double branched cover $\Sigma(K)$ of $K$ is the unique $\mathbb Z/2\mathbb Z \times \mathbb Z/2\mathbb Z$ cover of $\overline{K}\cup \overline{A}$; that is, the manifold $\Sigma(K)$ may be constructed by taking the branched double cover over each of the components of $\overline{K}\cup \overline{A}$ in either order (see, for example, \cite{MV:cyclic}). In the diagram below, $(S^3, B \cup K')$ denotes the branched double cover of $(S^3, \overline{K}, \overline{A})$ over $\overline{K}$, with $B$ being the unknot which is the lift of $\overline{K}$ and $K'$ denoting the lift of $\overline{A}$. Since $\overline{A}$ and $\overline{K}$ are interchanged by isotopy in the bottom copy of the three-sphere, $B$ is isotopic to a copy of the knot $K$ in $(S^3, B \cup K')$.
\begin{center}
\begin{tikzcd}
& (\Sigma(K), \widetilde{K} \cup \widetilde{A}) \dlar \drar & \\
(\Sigma(\overline{K}) = S^3, B \cup K') \drar & & (\Sigma(\overline{A})=S^3, K \cup A) \dlar \\
& (S^3, \overline{K} \cup \overline{A}) & 
\end{tikzcd}
\end{center}
We see that the lift $\widetilde{A}$ of $A$ in $\Sigma(K)$ is also up to isotopy the lift of $K$ in its branched double cover $\Sigma(K)$. \end{proof}

\begin{figure}
\scalebox{.5}{
\begingroup%
  \makeatletter%
  \providecommand\color[2][]{%
    \errmessage{(Inkscape) Color is used for the text in Inkscape, but the package 'color.sty' is not loaded}%
    \renewcommand\color[2][]{}%
  }%
  \providecommand\transparent[1]{%
    \errmessage{(Inkscape) Transparency is used (non-zero) for the text in Inkscape, but the package 'transparent.sty' is not loaded}%
    \renewcommand\transparent[1]{}%
  }%
  \providecommand\rotatebox[2]{#2}%
  \newcommand*\fsize{\dimexpr\f@size pt\relax}%
  \newcommand*\lineheight[1]{\fontsize{\fsize}{#1\fsize}\selectfont}%
  \ifx\svgwidth\undefined%
    \setlength{\unitlength}{384.83056641bp}%
    \ifx\svgscale\undefined%
      \relax%
    \else%
      \setlength{\unitlength}{\unitlength * \real{\svgscale}}%
    \fi%
  \else%
    \setlength{\unitlength}{\svgwidth}%
  \fi%
  \global\let\svgwidth\undefined%
  \global\let\svgscale\undefined%
  \makeatother%
  \begin{picture}(1,0.62410072)%
    \lineheight{1}%
    \setlength\tabcolsep{0pt}%
    \put(0,0){\includegraphics[width=\unitlength,page=1]{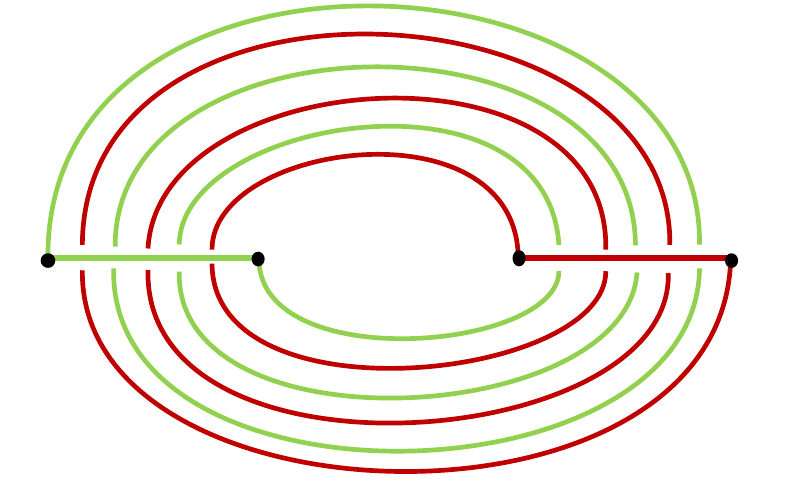}}%
    \put(0.8183453,0.09172663){\color[rgb]{0,0,0}\makebox(0,0)[lt]{\lineheight{1.25}\smash{\begin{tabular}[t]{l}$\overline{K}$\end{tabular}}}}%
    \put(0.80395681,0.5233813){\color[rgb]{0,0,0}\makebox(0,0)[lt]{\lineheight{1.25}\smash{\begin{tabular}[t]{l}$\overline{A}$\end{tabular}}}}%
  \end{picture}%
\endgroup%
}
\caption{The quotient knot $\overline{K}$ of the two-bridge link $K$ of Figure~\ref{fig:periodic_bridge} along its natural two-periodic symmetry, together with the image $\overline{A}$ of the axis $A$ under the quotient map. The knot is drawn in red and the axis in green. Note that the two components are interchanged by rotation.}
\label{fig:interchange}
\end{figure}

We now complete the proof of Corollary~\ref{cor:two-bridge}.

\begin{proof}[Proof of Corollary~\ref{cor:two-bridge}] This is the analog of \cite[Theorem 1.0.14]{Tweedy:anti-diagonal-reduced}; indeed, Tweedy proves that for a two-bridge knot the differential on $CF(\Sym^n(\Sigma_{\tau} \backslash \{w_1,w_2\}), \mathbb T_{\widehat{\alphas}}, \mathbb T_{\widehat{\betas}})$ preserves the filtration level of the anti-diagonal filtration, which therefore remains true after removing the basepoints $z_0$ and $z_1$ and their corresponding divisors in the symmetric product. However, in our setting we may additionally present a direct elementary proof. Consider a periodic two-bridge diagram for $K$, such as that for the trefoil in Figure~\ref{fig:periodic_bridge}, and form a multi-pointed Heegaard diagram for $\Sigma(K)$ by taking the branched cover over the four endpoints of the bridges as described above, {\`a} Figure~\ref{fig:heegaard_diagram_lift}. We observe that this Heegaard diagram is necessarily \emph{nice} in the sense of \cite{SW:nice, MOS:combinatorial}; a nice Heegaard diagram is one in which every region in the closure of the surface with the $\alpha$ and $\beta$ curves removed which does not contain a basepoint is either a square or a bigon. In the Heegaard diagram specified by taking the branched double cover of a two-bridge knot, every region which does not contain one of the four basepoints is in fact a rectangle. Sarkar and Wang showed that given a nice Heegaard diagram, all disks counted by the Heegaard Floer differentials have shadows which are embedded bigons or embedded rectangles on the Heegaard surface \cite{SW:nice}. In our situation, this implies that every domain $D$ which is the shadow of a curve $u$ which counts for the differential is an embedded rectangle. In order for such a curve $u$ to have positive intersection with $\nabla$, its shadow $D$ must contain two points $(u,0,z)$ and $(-u,0,z)$ on the Heegaard diagram. This is not true of any embedded rectangle in the diagram, since the placement of the four basepoints ensures that all embedded rectangles occur either on the front half of the Heegaard diagram, facing the reader, or on the back half. We conclude that the anti-diagonal filtration is trivial for two-bridge knots. It follows that the spectral sequence from $H_*(\mathcal G_i/\mathcal G_{i+1})$ to $\HFKhatDual(\Sigma(K), \widetilde{K})\otimes H^*(S^1)$ collapses on the $E_1$ page.
\end{proof}

\section{Spectral sequences for two-periodic links}\label{s:2periodic}

In this section we prove Theorem~\ref{thm:2periodic}, establishing an analog of the Stoffregen-Zhang spectral sequence from the symplectic Khovanov homology of a doubly-periodic link to the annular symplectic Khovanov homology of the quotient link \cite{SZ:localization}. As a warm-up, in Section~\ref{subsec:easy} we start by proving a simpler spectral sequence from the annular symplectic Khovanov homology of a doubly-periodic link to the annular symplectic Khovanov homology of the quotient link, which is the analog of earlier work of Zhang~\cite[Theorem 1]{Zhang:annular}. We then proceed to the main result of the section in Section~\ref{subsec:szthm}.

Recall from Section \ref{ss:quotient} that $\tau_{Q,o}$ denotes the configuration $\{-2n, \dots, -1,1, \dots, 2n\}$ in the set $\Conf^{4n,0} \cap \Fix(\iota)$, where $\iota$ is the map $z \mapsto -z$ on $\CC$. Furthermore, recall that given a two-periodic link $L$, there is a canonical way of making $L$ into an annular link by choosing the axis of periodicity as the annular axis. We may choose an annular bridge diagram for $L$ which is compatible with the symmetry $\iota$ and such that the endpoints of the bridges occur at the elements of $\tau_{Q,o}$, as in Figure~\ref{fig:periodic_bridge}. Given such a bridge diagram, there is a quotient annular link diagram as in Figure~\ref{fig:periodic_bridge_quotient}, such that the endpoints of the bridges occur on the points of the standard configuration $\tau_o = \{1,\dots, 2n\}$.

\subsection{From the annular theory of a two-periodic knot to the annular theory of the quotient} \label{subsec:easy}

\begin{figure}
\scalebox{.5}{
\begingroup%
  \makeatletter%
  \providecommand\color[2][]{%
    \errmessage{(Inkscape) Color is used for the text in Inkscape, but the package 'color.sty' is not loaded}%
    \renewcommand\color[2][]{}%
  }%
  \providecommand\transparent[1]{%
    \errmessage{(Inkscape) Transparency is used (non-zero) for the text in Inkscape, but the package 'transparent.sty' is not loaded}%
    \renewcommand\transparent[1]{}%
  }%
  \providecommand\rotatebox[2]{#2}%
  \newcommand*\fsize{\dimexpr\f@size pt\relax}%
  \newcommand*\lineheight[1]{\fontsize{\fsize}{#1\fsize}\selectfont}%
  \ifx\svgwidth\undefined%
    \setlength{\unitlength}{380.46823883bp}%
    \ifx\svgscale\undefined%
      \relax%
    \else%
      \setlength{\unitlength}{\unitlength * \real{\svgscale}}%
    \fi%
  \else%
    \setlength{\unitlength}{\svgwidth}%
  \fi%
  \global\let\svgwidth\undefined%
  \global\let\svgscale\undefined%
  \makeatother%
  \begin{picture}(1,0.63502109)%
    \lineheight{1}%
    \setlength\tabcolsep{0pt}%
    \put(0,0){\includegraphics[width=\unitlength,page=1]{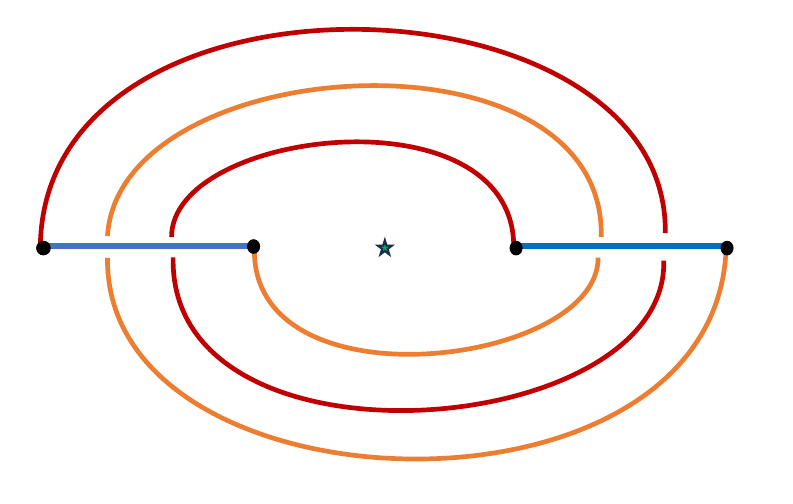}}%
    \put(0.89028623,0.35052878){\color[rgb]{0,0,0}\makebox(0,0)[lt]{\lineheight{1.25}\smash{\begin{tabular}[t]{l}$\tau_1$\end{tabular}}}}%
    \put(0.59577964,0.29185451){\color[rgb]{0,0,0}\makebox(0,0)[lt]{\lineheight{1.25}\smash{\begin{tabular}[t]{l}$\tau_2$\end{tabular}}}}%
    \put(0.01805452,0.2783116){\color[rgb]{0,0,0}\makebox(0,0)[lt]{\lineheight{1.25}\smash{\begin{tabular}[t]{l}$\tau_{2n}$\end{tabular}}}}%
    \put(0.78196172,0.33938764){\color[rgb]{0,0,0}\makebox(0,0)[lt]{\lineheight{1.25}\smash{\begin{tabular}[t]{l}$b_1$\end{tabular}}}}%
    \put(0.41524084,0.4262723){\color[rgb]{0,0,0}\makebox(0,0)[lt]{\lineheight{1.25}\smash{\begin{tabular}[t]{l}$a_1$\end{tabular}}}}%
  \end{picture}%
\endgroup%
}
\caption{A 2-periodic bridge diagram for the trefoil knot. The intersection of the annular axis with the diagram is marked with a star.}
\label{fig:periodic_bridge}
\end{figure}
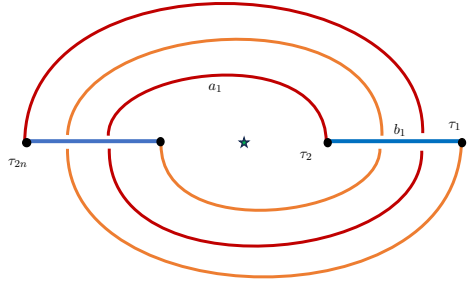

\begin{figure}
\scalebox{.5}{
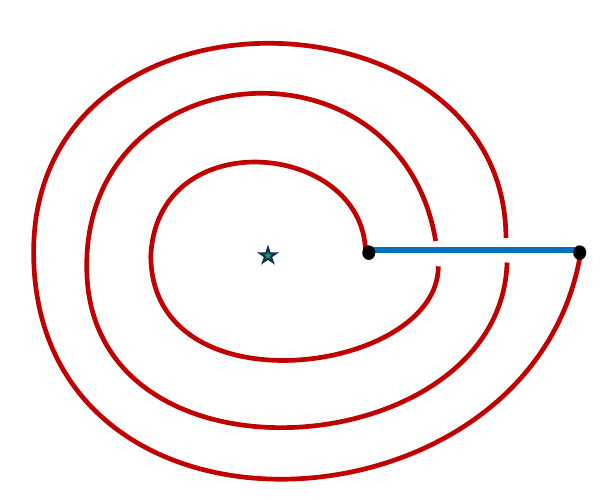}
\caption{The quotient of the periodic bridge diagram shown in Figure \ref{fig:periodic_bridge}. We continue to mark the annular axis with a star.}
\label{fig:periodic_bridge_quotient}
\end{figure}

Consider the involution $\iota_A \co A^*_{\tau_{Q,o}} \to A^*_{\tau_{Q,o}}$ given by $(u,v,z) \mapsto (u,v,-z)$. As in Section~\ref{ss:quotient}, this map induces an involution on $\Hilb^{2n, hor}(A^*_{\tau_{Q,o}})$, which we here call $\iota_{A,H}$. Recall further that $Q_{A,\tau_{Q,o}}$ denotes the quotient of $A^*_{\tau_{Q,o}}$ by $\iota_A$.

\begin{lemma}\label{l:fixed1}
The fixed point set of $\iota_{A,H}$ is naturally isomorphic to $\Hilb^{n, hor}(Q_{A, \tau_{Q,o}})$.
\end{lemma}

\begin{proof}
Let $I$ be a subscheme of points on $Q_{A,\tau_{Q,o}}$ which represents a point in $\Hilb^{n, hor}(Q_{A,\tau_{Q,o}})$.
Let $\tilde{I}$ be the pull-back of $I$ under the quotient map $A^*_{\tau_{Q,o}} \to Q_{A,\tau_{Q,o}}$.
 The pull-back $\tilde{I}$ is a subscheme of $A^*_{\tau_{Q,o}}$ with length $2n$.
Since $\iota_A$ is a lift of $\iota$ to $A^*_{\tau}$, and in particular has no fixed points, the subscheme $\tilde{I}$ necessarily lies in $\Hilb^{2n, hor}(A^*_{\tau_{Q,o}}) \subseteq \Hilb^{2n}(A^*_{\tau_{Q,o}})$.
It is also clear that $\tilde{I}$ is fixed by $\iota_{A,H}$. Therefore, the pullback map defines an embedding from $\Hilb^{n, hor}(Q_{A,\tau_{Q,o}})$ to the fixed set of $\iota_{A,H}$.

On the other hand, if $J$ is a subscheme lying in the fixed set of $\iota_{A,H}$, then its image $\pi_{HC}(J)$ in $\Sym^{2n}(A^*_{\tau_{Q,o}})$ must lie in the fixed set of the involution $\iota_{A,S}$ on $\Sym^{2n}(A^*_{\tau_{Q,o}})$ induced by $\iota_A$.
Since $\iota_A$ has no fixed points on $A_{\tau}^*$, it must be the case that $\pi_{HC}(J)=\{x,\iota_{A,S}(x)\}$ for some $x \in \Sym^{n}(A^*_{\tau_{Q,o}})$. Therefore, $J$ is necessarily of the form $J_x \cap \iota_A^* J_x $ for some length $n$ subscheme $J_x \subset A^*_{\tau_{Q,o}}$ such that $\pi_{HC}(J_x)=x$.
But the image of the embedding from $\Hilb^{n, hor}(Q_{A,\tau_{Q,o}})$ precisely covers all subschemes of the form $J_x \cap \iota_A^* J_x $. We conclude that this embedding is an isomorphism.
\end{proof}

We may now use this identification to prove the promised first spectral sequence.

\begin{proposition}\label{p:2periodicEasy}
 There is a spectral sequence whose $E_1$ page is isomorphic to
\[ \AKh_{symp}(L) \otimes \mathbb F_2[\theta, \theta^{-1}]\]
and whose $E_{\infty}$ page is isomorphic to \[\AKh_{symp}(\overline{L})\otimes \mathbb F_2[\theta, \theta^{-1}].\]
\noindent It follows that 
\[ \dim \AKh_{symp}(L) \geq \dim \AKh_{symp}(\overline{L}).\] \end{proposition}

\begin{proof} Given an annular bridge diagram for $L$ a two periodic link, number the bridges in the diagram for $L$ from right to left, as in the notation of Section~\ref{subsec:bridge}, so that $\iota(b_i) = b_{2n+1-i}$, and choose similar notation for the undercrossing arcs $a_i$ so that $\iota(a_i) = a_{2n+1-i}$. Let $\overline{b}_i$ and $\overline{a}_i$ denote the quotient arcs in the bridge diagram.

Now, we recall that $\AKh_{symp}(L)$ may be computed as the Lagrangian Floer cohomology of the triple $(\Hilb^{2n, hor}(A_{\tau_{Q,o}}), \Sigma_{\alpha}, \Sigma_{\beta})$, where 
\begin{align*}\Sigma_{\alpha} &= \Sigma_{a_1} \times \dots \times \Sigma_{a_{2n}} \\
				\Sigma_{\beta} &= \Sigma_{b_1} \times \dots \times \Sigma_{b_{2n}}
\end{align*}
are the products of the matching spheres lying above the arcs $a_i$ and $b_i$ as in \eqref{eqn:sphere} and Section~\ref{subsec:bridge}. The fixed sets of these two Lagrangians under the involution $\iota_{A,H}$ are exactly the embedded copies of the products
 \begin{align*}\Sigma_{\bar{\alpha}} &= \Sigma_{\bar{a}_1} \times \dots \times \Sigma_{\bar{a}_{n}} \\
				\Sigma_{\bar{\beta}} &= \Sigma_{\bar{b}_1} \times \dots \times \Sigma_{\bar{b}_{2n}}
\end{align*}
under the isomorphism of Lemma \ref{l:fixed1}. Therefore, the fixed sets of this triple under $\iota_{A,H}$ are
\[ ((\Hilb^{2n, hor}(A_{\tau_{Q,o}}))^{fix}, \Sigma_{\alpha}^{fix}, \Sigma_{\beta}^{fix}) \simeq (\Hilb^{n,hor}(Q_{A,\tau_{Q,o}}), \Sigma_{\bar{\alpha}}, \Sigma_{\bar{\beta}}).\]
The righthand side computes $\AKh_{Q,A}(\overline{L})$, which we recall from Proposition \ref{p:annularSame} is isomorphic to $\AKh_{symp}(\overline{K})$. Moreover, it follows from the arguments of Section~\ref{sec:normal-trivs} that this setup admits a stable normal trivialization. We may now apply Corollary \ref{c:localization} to obtain the desired spectral sequence.
\end{proof}

\begin{remark} \label{rem:either} The use of $\iota_A$ in this section is for expositional simplicity only. It follows from Lemma \ref{l:fixed2} and Corollary~\ref{c:A=B} that on $\Hilb^{2n, hor}(A^*_{\tau_{Q,o}})$, the involutions $\iota_A$ and $\iota_B$ have isomorphic fixed sets, and each induce a spectral sequence from the annular symplectic Khovanov homology of a two-periodic link to the annular symplectic Khovanov homology of the quotient. \end{remark}

\subsection{From the ordinary theory of a two-periodic knot to the annular theory of the quotient} \label{subsec:szthm}

We now consider the involution $\iota_B \co A_{\tau_{Q,o}} \to A_{\tau_{Q,o}}$ given by $(u,v,z) \mapsto (-u,-v,-z)$, which induces an involution $\iota_{B,H}$ on $\Hilb^{2n, hor}(A_{\tau_{Q,o}})$. We maintain the same choices of notation for periodic annular bridge diagrams as in Section~\ref{subsec:easy}. Recall from Section~\ref{ss:quotient} that $Q_{B,\tau_{Q,o}}$ denotes the quotient of $A_{\tau}^*$ under $\iota_B$.

\begin{lemma}\label{l:fixed2}
The fixed point set of $\iota_{B,H}$ is naturally isomorphic to $\Hilb^{n, hor}(Q_{B,\tau_{Q,o}})$.
\end{lemma}

\begin{proof} Much as in the proof of Lemma \ref{l:fixed1}, one can construct an embedding from $\Hilb^{n, hor}(Q_{B,\tau_{Q,o}})$ to $\Hilb^{2n, hor}(A_{\tau_{Q,o}}^*) \subseteq \Hilb^{2n, hor}(A_{\tau_{Q,o}})$.
Moreover, the image of this embedding is precisely the intersection between the fixed locus of $\iota_{B,H}$ and $\Hilb^{2n, hor}(A_{\tau_{Q,o}}^*)$.
It therefore suffices to check that the fixed locus of $\iota_{B,H}$ lies entirely in $\Hilb^{2n, hor}(A_{\tau_{Q,o}}^*)$.

Let $J \in  \Hilb^{2n, hor}(A_{\tau_{Q,o}})$ be an element in the fixed locus of $\iota_{B,H}$. If $y \in A_{\tau_{Q,o}}$ lies in the support of $J$, then so does $\iota_B(y)$.
Notice that $\iota_B$ has no fixed points in $A_{\tau_{Q,o}}$, so $y$ and $\iota_B(y)$ are two distinct points on the surface. Suppose for the sake of contradiction that $\pi_A(y)=0 \in \CC$. Then we also have $\pi_A(\iota_B(y))=0$.
Since $y$ and $\iota_B(y)$ are distinct points lying in $\pi_A^{-1}(0)$, the ideal $(\pi_{A})_* J$ has length strictly less than $J$ and therefore cannot be in $\Hilb^{2n, hor}(A_{\tau_{Q,o}})$. It follows that in fact $\pi_A(y)$ cannot be $0$.
In other words, the support of $J$ misses $\pi_A^{-1}(0)$ and hence $J \in \Hilb^{2n, hor}(A_{\tau_{Q,o}}^*)$.
\end{proof}

We now have all the tools to give a short proof of Theorem~\ref{thm:2periodic}.

\begin{proof}[Proof of Theorem \ref{thm:2periodic}] Given an annular bridge diagram for $L$ a two periodic link, as in the previous section we number the bridges in the diagram for $L$ from right to left, so that $\iota(b_i) = b_{2n+1-i}$, and choose similar notation for the undercrossing arcs $a_i$ so that $\iota(a_i) = a_{2n+1-i}$. Let $\overline{b}_i$ and $\overline{a}_i$ denote the quotient arcs in the bridge diagram. As previously, $\AKh_{symp}(L)$ may be computed as the Lagrangian Floer cohomology of the triple $(\Hilb^{2n, hor}(A_{\tau_{Q,o}}), \Sigma_{\alpha}, \Sigma_{\beta})$, and the fixed sets of this triple under $\iota_{B,H}$ are
\[ (\Hilb^{2n, hor}((A_{\tau_{Q,o}}))^{fix}, \Sigma_{\alpha}^{fix}, \Sigma_{\beta}^{fix}) \simeq (\Hilb^{n,hor}(Q_{B,\tau_{Q,o}}), \Sigma_{\bar{\alpha}}, \Sigma_{\bar{\beta}})\]
where the identification is via viewing the manifolds on the righthand side as their embedded copies under the isomorphism of Lemma~\ref{l:fixed2}. The righthand side computes $\AKh_{Q,B}(\overline{L})$, which we recall from Proposition \ref{p:annularSame} is isomorphic to $\AKh_{symp}(\overline{L})$. Moreover, it follows from the arguments of Section~\ref{sec:normal-trivs} that this setup admits a stable normal trivialization. We may now apply Corollary \ref{c:localization} to obtain the desired spectral sequence. \end{proof}

\subsection{The example of the Hopf link} \label{subsec:Hopf}

In this section we elaborate on an example originally due to Seidel by examining the positive Hopf link with its standard two-crossing diagram, considered as a two-periodic link, with a diagram as in Figure~\ref{fig:Hopf}. The combinatorial Khovanov homology is
\[ \Kh(H_+; \ZZ) \simeq \ZZ_{(0,0)} \oplus \ZZ_{(0,2)} \oplus \ZZ_{(2,4)} \oplus \ZZ_{(2,6)} \]
where the gradings $(j,q)$ in the subscript are the usual homological gradings followed by the quantum gradings. The Khovanov homology over any field is therefore also four-dimensional with the same gradings. The involution on the chain complex $\CKh(L)$ for the combinatorial Khovanov homology of a two-periodic link $L$ induced by the link periodicity, which we shall call $\iota_{Kh}$, is constructed on the chain level by interchanging resolutions of the periodic link diagram by rotation, and is easily seen to preserve the homological and quantum gradings on combinatorial Khovanov homology. Therefore, it must be the identity. As the differential also preserves the quantum grading, the equivariant homology of $(\CKh(H_+), \iota_{Kh})$ therefore has rank four over $H_*(\mathbb{RP}^{\infty}; \FF_2) = \FF_2[\theta]$; by Stoffregen and Zhang's work \eqref{eq:sz}, it localizes to the annular Khovanov homology of the one-crossing unknot shown in Figure~\ref{fig:Hopf}, which is indeed of dimension four.

\begin{figure}[ht]
\scalebox{.5}{
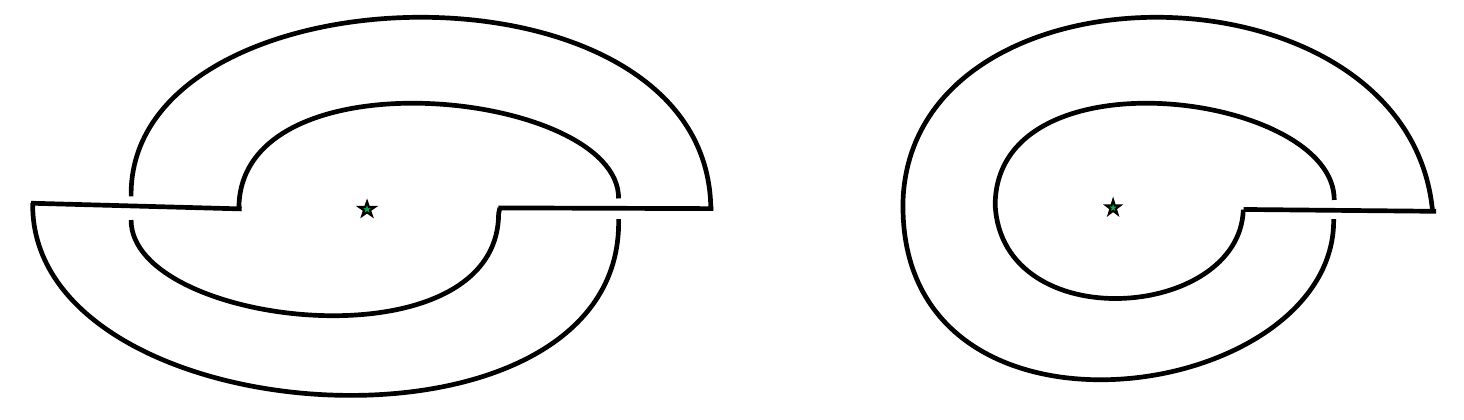}
\caption{The Hopf link considered as a two-periodic link, and its quotient (annular) unknot.}
\label{fig:Hopf}
\end{figure}

The symplectic Khovanov homology of the Hopf link is
\[ \Kh_{symp}(H_+; \FF_2) \simeq \FF_{(0)} \oplus \FF^{2}_{(-2)} \oplus \FF_{(-4)}.\]
Here the subscripts denote the Khovanov gradings $i=j-q$. Let $\iota_A^*$ be the map on $\CKh_{symp}(H_+)$ induced by $\iota_A$ using the constructions of Section~\ref{sec:localization}. This map preserves the Khovanov gradings, and by Seidel and Smith's work \eqref{eq:ss} the equivariant cohomology of $(\CKh_{symp}(H_+), \iota_A^*)$ localizes to the symplectic Khovanov homology of the quotient unknot, hence must have rank two over $\FF[\theta]$. It follows that either the map $\iota_A^*$ exchanges the two generators in Khovanov grading $-2$, or there is a higher-degree cancellation in the localization spectral sequence; consideration of the gradings implies that in fact it is the second possibility that holds. Seidel therefore observed that this implies that the involution induced by $\iota_A$ on $\CKh_{symp}(H_+)$ admits no lift to a bigraded theory, including over characteristic zero. Therefore there is no $\ZZ/2\ZZ$-equivariant quasi-isomorphism between $(\CKh(L), \iota_{Kh})$ and $(\CKh_{symp}(L), \iota_A^*)$, even over characteristic zero, for two-periodic links.

The results of this section, however, show that $\iota_B^*$ must induce the identity on $Kh_{symp}(H_+)$, and that the equivariant cohomology of $(\CKh(H_+), \iota_B^*)$ must have rank four over $\FF_2[\theta]$, so that it localizes to the annular symplectic Khovanov homology of the annular unknot shown in Figure~\ref{fig:Hopf}, which is easily checked to be rank four. One may conjecture there that is a $\ZZ/2\ZZ$-equivariant quasi-isomorphism between $(\CKh(L), \iota_{Kh})$ and $(\CKh_{symp}(L), \iota_B^*)$.

\begin{remark} The reader is reminded that $\iota_A^*$ and $\iota_B^*$ may not be naively computable before stabilization and an isotopy of the Lagrangians; see Remark~\ref{rem:caveat} for further discussion.  \end{remark} 

We are grateful to Ivan Smith for reminding us of this example and helpful discussion.

\section{A spectral sequence for strongly invertible knots}\label{s:strongly}

In this section, we prove Theorem \ref{thm:stronglyinvertible}, establishing an analog of the Lipshitz-Sarkar spectral sequence from the Khovanov homology of a strongly invertible knot to the cone of an axis-moving map relating the annular Khovanov homology on appropriate annular versions of the quotient knot \cite{LS:strongly_invertible}.

\subsection{The symplectic axis-moving map} \label{subsec:axis} Let $L$ be an annular link which is represented by an annular bridge diagram such that the endpoints of the bridges lie at the points of a configuration $\tau \in \Conf^{2n,0}(\CC)$ as in Section~\ref{subsec:bridge}. Let $\ul{b} = \{b_1, \dots, b_n\}$ be the set of bridges and $\ul{a} = \{a_1,\dots, a_n\}$ be the set of undercrossing arcs. Let $\Sigma_{\beta}$ and $\Sigma_{\alpha}$ be the Lagrangians associated respectively to the bridges and the undercrossing arcs as in the same section so that 
$\AKh_{symp}(L)=HF_{ann}(\Hilb^{n,hor}(A_{\tau}), \Sigma_{\alpha}, \Sigma_{\beta})$. Suppose there exists an isotopy of one of the undercrossing arcs $a_i$ across the origin which introduces no new crossings as in Figure~\ref{fig:alpha_slide}; equivalently, suppose there exists a handleslide of one of the arcs $a_i$ over a circle around the origin such that the interior of the handleslide region does not intersect any other $a_j$. Let $a_i'$ denote the result of this isotopy and $\Sigma_{\alpha'}$ be the Lagrangian associated to the set of arcs produced by replacing $a_i$ with $a_i'$.
Then $HF_{ann}(\Hilb^{n,hor}(A_{\tau}), \Sigma_{\alpha'}, \Sigma_{\beta})$ computes the annular Khovanov homology of the underlying link of $L$ with respect to a new choice of annular axis; one here regards the axis as having moved across the arc $a_i$. We refer to this annular link as $L'$. In general, $HF_{ann}(\Hilb^{n,hor}(A_{\tau}), \Sigma_{\alpha}, \Sigma_{\beta})$ and $HF_{ann}(\Hilb^{n,hor}(A_{\tau}), \Sigma_{\alpha'}, \Sigma_{\beta})$ are different. 

\begin{figure}[ht]
\fontsize{15pt}{12pt}
\scalebox{.5}{
\begingroup%
  \makeatletter%
  \providecommand\color[2][]{%
    \errmessage{(Inkscape) Color is used for the text in Inkscape, but the package 'color.sty' is not loaded}%
    \renewcommand\color[2][]{}%
  }%
  \providecommand\transparent[1]{%
    \errmessage{(Inkscape) Transparency is used (non-zero) for the text in Inkscape, but the package 'transparent.sty' is not loaded}%
    \renewcommand\transparent[1]{}%
  }%
  \providecommand\rotatebox[2]{#2}%
  \newcommand*\fsize{\dimexpr\f@size pt\relax}%
  \newcommand*\lineheight[1]{\fontsize{\fsize}{#1\fsize}\selectfont}%
  \ifx\svgwidth\undefined%
    \setlength{\unitlength}{458.88967896bp}%
    \ifx\svgscale\undefined%
      \relax%
    \else%
      \setlength{\unitlength}{\unitlength * \real{\svgscale}}%
    \fi%
  \else%
    \setlength{\unitlength}{\svgwidth}%
  \fi%
  \global\let\svgwidth\undefined%
  \global\let\svgscale\undefined%
  \makeatother%
  \begin{picture}(1,0.36500755)%
    \lineheight{1}%
    \setlength\tabcolsep{0pt}%
    \put(0,0){\includegraphics[width=\unitlength,page=1]{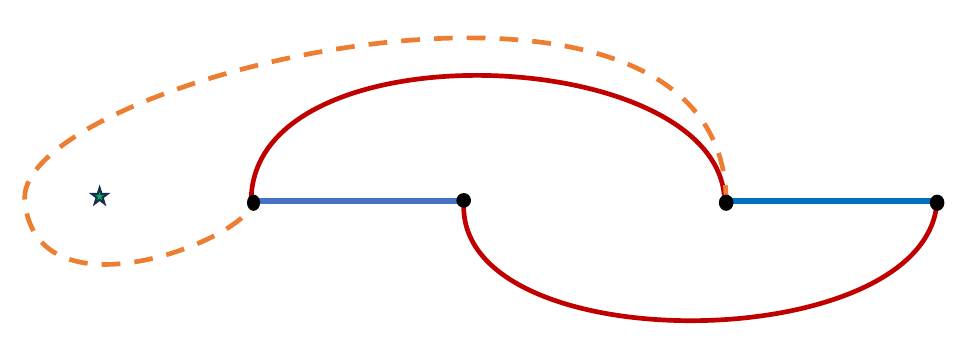}}%
    \put(0.37254903,0.2488688){\color[rgb]{0,0,0}\makebox(0,0)[lt]{\lineheight{1.25}\smash{\begin{tabular}[t]{l}$a_i$\end{tabular}}}}%
    \put(0.09049774,0.25791856){\color[rgb]{0,0,0}\makebox(0,0)[lt]{\lineheight{1.25}\smash{\begin{tabular}[t]{l}$a_i'$\end{tabular}}}}%
  \end{picture}%
\endgroup%
}
\caption{An annular bridge diagram for the unknot together with the result of an isotopy of an undercrossing arc over the origin. The solid red arc $a_i$ is replaced with the dashed orange arc $a_i'$.}
\label{fig:alpha_slide}
\end{figure}

In the case of combinatorial annular Khovanov homology, one may define an axis-moving map between the the annular Khovanov homologies of $L$ and $L'$ in terms of compositions of cobordism maps between them \cite[Section 3]{LS:strongly_invertible}. For symplectic Khovanov homology, the analog of this map has a simple construction in terms of the continuation map associated to the isotopy over the origin, as follows.

Recall that $CF_{ann}(\Hilb^{n,hor}(A_{\tau}), \Sigma_{\alpha},\Sigma_{\beta})=CF(\Hilb^{n,hor}(A_{\tau}), \Sigma_{\alpha},\Sigma_{\beta}; \FF_2[U])/(U=0)$. For the remainder of this section, the symplectic manifold will be fixed, so we drop it from the notation. Notice that with $\Sigma_{\alpha}$ and $\Sigma_{\alpha'}$ as above, up to an isotopy in the complement of $D_o$, we may arrange that $\Sigma_{\alpha}\cap \Sigma_{\alpha'}$ is the set of critical points of the projection map $\pi_A$. More precisely, we can isotope the copies of the matching paths $a_2,\dots,a_n$ in $\alpha'$ to $a_2',\dots,a_n'$ relative to their endpoints so that $a_i \cap a_j' \neq \emptyset$ if and only if $i=j$, and furthermore so that $a_i$ intersects $a_i'$ only at the end points. This isotopy from $a_j$ to $a_j'$ then induces an isotopy from $\Sigma_{\alpha'}$ to another Lagrangian whose intersection with $\Sigma_{\alpha}$ is the 
set of critical points.
After this isotopy, all of the generators of the chain complex have even degree and the differential vanishes. It follows that
\begin{align}
HF(\Sigma_{\alpha'},\Sigma_{\alpha}; \FF[U])&=H^*(S^2)^{\otimes n} \otimes_{\mathbb{F}_2} \mathbb{F}_2[U]  \label{eq:moveHF}\\
HF_{ann}(\Sigma_{\alpha'},\Sigma_{\alpha})&=H^*(S^2; \FF_2)^{\otimes n}. \label{eq:moveHF2}
\end{align}
Let $\tilde{c}_{\alpha,\alpha'} \in HF(\Sigma_{\alpha'},\Sigma_{\alpha}; \FF_2[U])$ be the element corresponding to the identity $1 \in H^0(S^2)^{\otimes n} \otimes_{\mathbb{F}_2} \mathbb{F}_2[U]$ and likewise $c_{\alpha,\alpha'} \in HF_{ann}(\Sigma_{\alpha'},\Sigma_{\alpha})$ be the element corresponding to the identity $1 \in H^0(S^2; \FF_2)^{\otimes n}$ respectively.

There is a Floer theoretic product map 
\begin{align}\label{eq:product}
\mu_2 \co CF(\Sigma_{\alpha},\Sigma_{\beta}; \FF_2[U]) \times CF(\Sigma_{\alpha'},\Sigma_{\alpha}; \FF_2[U]) \to CF(\Sigma_{\alpha'},\Sigma_{\beta}; \FF_2[U]) 
\end{align}
given by counting rigid pseudo-holomorphic triangles $u$ with boundary on $(\Sigma_{\alpha'}, \Sigma_{\alpha}, \Sigma_{\beta})$ weighted by $U^{u \cdot D_o}$ coming from counting intersections with the annular divisor as previously.
Setting $U=0$ induces a map
\begin{align}
\mu_2|_{U=0} \co CF_{ann}(\Sigma_{\alpha},\Sigma_{\beta}) \times CF_{ann}(\Sigma_{\alpha'},\Sigma_{\alpha}) \to CF_{ann}(\Sigma_{\alpha'},\Sigma_{\beta}) 
\end{align}
which descends to a map on cohomology 
\begin{align}\label{eq:productH}
\mu_2|_{U=0} \co HF_{ann}(\Sigma_{\alpha},\Sigma_{\beta}) \times HF_{ann}(\Sigma_{\alpha'},\Sigma_{\alpha}) \to HF_{ann}(\Sigma_{\alpha'},\Sigma_{\beta}).
\end{align}

\begin{definition}\label{d:axis}
The axis-moving map $HF_{ann}(\Sigma_{\alpha},\Sigma_{\beta}) \to HF_{ann}(\Sigma_{\alpha'},\Sigma_{\beta})$ associated to the isotopy from $\alpha$ to $\alpha'$ over the origin is the map $\mu_2(\cdot, c_{\alpha,\alpha'})$.
\end{definition}

We pause to explain the geometric reasoning behind this definition. Consider a smooth isotopy $(\alpha_t)_{t \in [0,1]}$ from $\alpha_0=\alpha$ to $\alpha_1=\alpha'$ such that $\alpha_t$ passes through the origin only at $t=1/2$.
For any small $\epsilon>0$, we may consider the diagram
\begin{align}\label{eq:Comm}
\xymatrixcolsep{5mm}
\xymatrix{ 
CF(\Sigma_{\alpha},\Sigma_{\beta}; \FF_2[U])   \ar[r] \ar[d] & CF(\Sigma_{\alpha_{1/2-\epsilon}},\Sigma_{\beta}; \FF_2[U])  \ar[r] \ar[d] & CF(\Sigma_{\alpha_{1/2+\epsilon}},\Sigma_{\beta}; \FF_2[U]) \ar[r] \ar[d] & CF(\Sigma_{\alpha'},\Sigma_{\beta}; \FF_2[U]) \ar[d]\\
CF_{ann}(\Sigma_{\alpha},\Sigma_{\beta}) \ar[r]            & CF_{ann}(\Sigma_{\alpha_{1/2-\epsilon}},\Sigma_{\beta}) \ar[r]            & CF_{ann}(\Sigma_{\alpha_{1/2+\epsilon}},\Sigma_{\beta}) \ar[r]            & CF_{ann}(\Sigma_{\alpha'},\Sigma_{\beta}).
}
\end{align}
The horizontal maps between the first two columns and last two columns are continuation maps defined by counting rigid pseudo-holomorphic strips with moving Lagrangian boundary conditions and weighting the resulting count by $U^{u \cdot D_o}$, and setting $U=0$ for the bottom horizontal arrows. The fact that the boundary condition defined by the isotopy $\alpha_t$ does not intersect $D_o$ guarantees that $u \cdot D_o$ is well-defined, and as usual positivity of intersections guarantees that it is non-negative.
A standard cobordism argument may be used to show that these continuation maps are homotopic to $\mu_2(\cdot, \tilde{c}_{\alpha,\alpha_{1/2-\epsilon}})$
and $\mu_2(\cdot, \tilde{c}_{\alpha_{1/2+\epsilon},\alpha'})$ respectively, as shown in Figure \ref{fig:cont-tri}.
\begin{figure}[ht]
\begingroup%
  \makeatletter%
  \providecommand\color[2][]{%
    \errmessage{(Inkscape) Color is used for the text in Inkscape, but the package 'color.sty' is not loaded}%
    \renewcommand\color[2][]{}%
  }%
  \providecommand\transparent[1]{%
    \errmessage{(Inkscape) Transparency is used (non-zero) for the text in Inkscape, but the package 'transparent.sty' is not loaded}%
    \renewcommand\transparent[1]{}%
  }%
  \providecommand\rotatebox[2]{#2}%
  \newcommand*\fsize{\dimexpr\f@size pt\relax}%
  \newcommand*\lineheight[1]{\fontsize{\fsize}{#1\fsize}\selectfont}%
  \ifx\svgwidth\undefined%
    \setlength{\unitlength}{311.75831952bp}%
    \ifx\svgscale\undefined%
      \relax%
    \else%
      \setlength{\unitlength}{\unitlength * \real{\svgscale}}%
    \fi%
  \else%
    \setlength{\unitlength}{\svgwidth}%
  \fi%
  \global\let\svgwidth\undefined%
  \global\let\svgscale\undefined%
  \makeatother%
  \begin{picture}(1,0.19684084)%
    \lineheight{1}%
    \setlength\tabcolsep{0pt}%
    \put(0,0){\includegraphics[width=\unitlength,page=1]{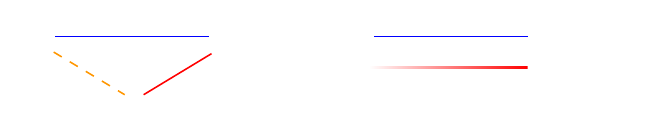}}%
    \put(0.1956257,0.02673147){\color[rgb]{0,0,0}\makebox(0,0)[lt]{\lineheight{1.25}\smash{\begin{tabular}[t]{l}$c$\end{tabular}}}}%
    \put(0.04981762,0.11543139){\color[rgb]{0,0,0}\makebox(0,0)[lt]{\lineheight{1.25}\smash{\begin{tabular}[t]{l}$\leftarrow$\end{tabular}}}}%
    \put(0.32124231,0.11786565){\color[rgb]{0,0,0}\makebox(0,0)[lt]{\lineheight{1.25}\smash{\begin{tabular}[t]{l}$\leftarrow$\end{tabular}}}}%
    \put(0.54504674,0.10910368){\color[rgb]{0,0,0}\makebox(0,0)[lt]{\lineheight{1.25}\smash{\begin{tabular}[t]{l}$\leftarrow$\end{tabular}}}}%
    \put(0.82157511,0.10751612){\color[rgb]{0,0,0}\makebox(0,0)[lt]{\lineheight{1.25}\smash{\begin{tabular}[t]{l}$\leftarrow$\end{tabular}}}}%
    \put(0,0){\includegraphics[width=\unitlength,page=2]{cont-tri.pdf}}%
  \end{picture}%
\endgroup%

\caption{Left: A Floer product $\mu_2(\cdot, \tilde{c}_{\alpha_t,\alpha_s})$. Read clockwise the boundary conditions are $\Sigma_{\alpha_t}$, drawn in red, $\Sigma_{\alpha_s}$, drawn in orange and dashed, and $\Sigma_{\beta}$, drawn in blue. The intersection point $\tilde{c}_{\alpha_t,\alpha_s}$ is labelled $c$.  Right: A strip counted by the continuation map with moving boundary condition from $\Sigma_{\alpha_s}$ to $\Sigma_{\alpha_t}$. The boundary conditions are $\Sigma_{\beta}$ on top, drawn in blue and solid, and the moving Lagrangian below, drawn in red and shaded. Arrows indicate inputs and outputs.}
\label{fig:cont-tri}
\end{figure}

It is not a priori clear how to define the continuation maps between the second and the third columns of the diagram \eqref{eq:Comm} so as to be compatible with the $\mathbb{F}_2[U]$-module structure.
However, motivated by the discussion above, we can define the top and bottom horizontal maps to be 
$\mu_2(\cdot, \tilde{c}_{\alpha_{1/2-\epsilon},\alpha_{1/2+\epsilon}})$ and $\mu_2(\cdot, c_{\alpha_{1/2-\epsilon},\alpha_{1/2+\epsilon}})$ respectively.
It is clear that the resulting diagram commutes and induces a commuting diagram on homology.
\begin{align}\label{eq:CommH} \xymatrixcolsep{5mm}
\xymatrix{ 
HF(\Sigma_{\alpha},\Sigma_{\beta}; \FF_2[U])   \ar[r]^{\kern-10pt{\sim}} \ar[d] & HF(\Sigma_{\alpha_{1/2-\epsilon}},\Sigma_{\beta}; \FF_2[U])  \ar[r] \ar[d] & HF(\Sigma_{\alpha_{1/2+\epsilon}},\Sigma_{\beta}; \FF_2[U]) \ar[r]^{\kern10pt{\sim}}  \ar[d] & HF(\Sigma_{\alpha'},\Sigma_{\beta}; \FF_2[U]) \ar[d]\\
HF_{ann}(\Sigma_{\alpha},\Sigma_{\beta}) \ar[r]^{\kern-10pt{\sim}}             & HF_{ann}(\Sigma_{\alpha_{1/2-\epsilon}},\Sigma_{\beta}) \ar[r]            & HF_{ann}(\Sigma_{\alpha_{1/2+\epsilon}},\Sigma_{\beta}) \ar[r]^{\kern10pt{\sim}}             & HF_{ann}(\eL_{\alpha'},\eL_{\beta}).
}
\end{align}
The compositions of the horizontal maps of \eqref{eq:CommH} along the top and bottom rows are independent of $\epsilon$ and indeed are precisely $\mu^2(\cdot, \tilde{c}_{\alpha,\alpha'})$ and $\mu^2(\cdot, c_{\alpha,\alpha'})$ respectively. This motivates Definition \ref{d:axis}: the symplectic-axis moving map is the composition of the maps along the bottom row of \eqref{eq:CommH} and can be interpreted as an appropriately defined continuation map.

\begin{example}\label{e:continuation}
We consider the example of the one-bridge diagram for the unknot shown in Figure \ref{fig:continuation} with bridge $b_1$ and undercrossing arc $a_1$, along with the result $a_1'$ of isotoping the $a_1$ over the origin. This will provide the basis for a general computation. Note that all three Lagrangians are spheres. Where we are interested in analyzing the generators of the symplectic Khovanov chain complex, we omit crossing information from our diagrams in order to see the intersection points and domains on $\CC$ more clearly. Each of the triple intersection points labeled $c$ and $f$ has a single lift to $A^*_{\tau}$; abusing notation slightly, we also refer to these points as $c$ and $f$. As explained in Section~\ref{subsec:bridge}, we perturb $\Sigma_{\beta}$ compatibly with the projection to $\CC$ so that above each of the intersection points $d$ and $e$ there are two points each in $\Sigma_{\alpha} \cap \Sigma_{\beta}$, called $d_1, d_2, e_1, e_2$ and likewise above $d'$ and $e'$ there are two points each in $\Sigma_{\alpha'} \cap \Sigma_{\beta}$ called $d_1',d_2',e_1',e_2'$.

The intersection points $c$ and $f$ have degrees $0$  and $2$, respectively, as generators in $CF(\Sigma_{\alpha'},\Sigma_{\alpha})$. In particular, we have $\tilde{c}_{\alpha,\alpha'}=c$ and $c_{\alpha,\alpha'}=c$. As generators in $CF(\Sigma_{\alpha},\Sigma_{\beta}; \FF_2[U])$ and $CF(\Sigma_{\alpha'},\Sigma_{\beta}; \FF_2[U])$, the degrees of the intersection points between the spheres described above are
\begin {align*}
|c|=0, \quad |d_1|=|d_1'|=1, \quad |d_2|=|d_2'|=2 \\
|f|=2, \quad |e_1|=|e_1'|=0, \quad |e_2|=|e_2'|=1
\end{align*}
On $CF(\Sigma_{\alpha}, \Sigma_{\beta}; \FF_2[U])$ the differentials are
\begin {align*}
\partial f=0, \quad \partial  c=d_1 , \quad \partial e_1=d_1, \quad \partial  e_2=f+d_2, \quad \partial d_1=\partial d_2 =0
\end{align*}
and on  $CF(\Sigma_{\alpha'}, \Sigma_{\beta}; \FF_2[U])$ the differentials are
\begin {align*}
\partial f=0, \quad \partial c = d_1', \quad \partial e_1'=d_1', \quad \partial  e_2'=Uf+d_2', \quad \partial  d_1' =\partial d_2'=0.
\end{align*}
Therefore, $HF_{ann}(\Sigma_{\alpha},\Sigma_{\beta})$ is generated by $[e_1+c]$ and $[f]=[d_2]$, while $HF_{ann}(\Sigma_{\alpha'},\Sigma_{\beta})$ is generated by $[e_1'+c]$ and $[f]$. Note that $[d_2']=0 \neq [f]$.

We now consider the Floer product \eqref{eq:product}. The constant Floer solution at $c$ gives us $\mu_2(c,c)=c.$

There is a holomorphic triangle which projects onto the bigon bounded by the solid red and dashed orange arcs giving us $\mu_2(f,c)=Uf$.
On the other hand, there are holomorphic triangles giving us $\mu_2(e_i,c)=b_i'$ and $\mu_2(d_i,c)=d_i'$.
Therefore, recalling that we set $U=0$, the chain level axis-moving map sends 
\[ c \mapsto c \qquad d_i \mapsto d_i' \qquad e_i \mapsto e_i' \qquad f \mapsto 0.\]
On cohomology, the axis-moving map sends $[e_1+c]$ to $[e_1'+c]$ and $[f]=[d_2]$ to $0$.
\end{example}

\begin{figure}[ht]
\scalebox{.7}{
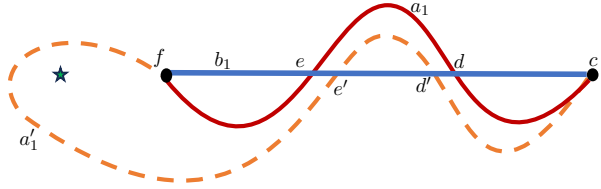}
\caption{A one-bridge diagram for the unknot with horizontal bridge $b_1$ drawn in blue and original undercrossing arc $a_1$in solid red, together with the result $a_1'$ of sliding the undercrossing arc over the origin, drawn orange and dashed. As usual the green star denotes the origin. The intersection points are labeled by $c, d, d', b, b', f$.}
\label{fig:continuation}
\end{figure}

Example \ref{e:continuation} can be generalized to the case of multiple bridges, as we now explain. Consider an annular bridge diagram for a link with bridges $\ul{b} = (b_1, \dots, b_n)$ and undercrossing arcs $\ul{a}=(a_1,\dots,a_n)$.
Let $l=a_i$ be a choice of undercrossing arc which may be slid across the origin to $l'=a_i'$ such that the isotopy region does not intersect any other $a_j$.
Let $l_t$ be the smooth isotopy from $l$ to $l'$. We assume that the image of the isotopy intersects the origin exactly at $t=1/2$.
We may assume that $l_{1/2}$ intersects the bridges transversely and furthermore for any $t,t'$ sufficiently close to $1/2$, we additionally assume that the intersection between $l_t$ and $l_{t'}$ consists precisely of their mutual endpoints, which are also the endpoints of $a_i$.
This gives us a natural bijective correspondence between $l_t \cap b_i$ and $l_{t'} \cap b_i$ for $t,t'$ sufficiently close to $1/2$ and all $i=1,\dots, n$.
Let $\phi_{t,t'}$ be the map which takes an intersection point in the total set of intersection points between $\ell_t$ and the bridges $\cup_{i=1}^n(l_t \cap b_i)$ to its counterpart in $\cup_{i=1}^n (l_{t'} \cap b_i)$.
This bijection in turn gives us a natural bijective correspondence between $\Sigma_{\alpha_{1/2-\epsilon}} \pitchfork \Sigma_{\beta}$ and $\Sigma_{\alpha_{1/2 +\epsilon}} \pitchfork \Sigma_{\beta}$ for any small $\epsilon>0$, where as usual $\Sigma_{\beta}$ is perturbed compatibly with the projection map to the complex plane to achieve transversality.
This bijection induces module isomorphisms
\begin{align}
\Phi_{\epsilon}\co &CF(\Sigma_{\alpha_{1/2-\epsilon}},\Sigma_{\beta}; \FF_2[U])  \to CF(\Sigma_{\alpha_{1/2+\epsilon}},\Sigma_{\beta}; \FF_2[U]) \label{eq:Phi} \\
\Phi_{\epsilon}|_{U=0} \co &CF_{ann}(\Sigma_{\alpha_{1/2-\epsilon}},\Sigma_{\beta})  \to CF_{ann}(\Sigma_{\alpha_{1/2+\epsilon}},\Sigma_{\beta})  \label{eq:PhiU}
\end{align}
In Example \ref{e:continuation}, these module isomorphisms send $c$ to $c$, $d_i$ to $d_i'$, $e_i$ to $e_i'$, and $f$ to $f$.

The assumption on $l_t$ near $t=1/2$ implies that the union of the arcs $l_t$ over $t \in [1/2-\epsilon,1/2+\epsilon]$ foliates a bigon $B$ with corners at the two endpoints.
Denote the corners of $B$ by $c$ and $f$ respectively so that the boundary of $B$ in clockwise order is given by $(f,l_{1/2-\epsilon},c,l_{1/2+\epsilon})$. In the example of Figure~\ref{fig:continuation}, we may take $\epsilon = 1/2$, and $B$ is the bigon bounded by the solid red arc $a_1$ and dashed orange curve $a_1'$.

Motivated by Example \ref{e:continuation}, we divide the intersection points $\cup_{i=1}^n(l_{1/2-\epsilon} \cap b_i)$ into two types, and call them type-I and type-II.
An intersection point $d \in l_{1/2-\epsilon} \cap b_i$ is of type-I if there is a triangle $T \subset B$ with corners $c$, $d$ and $\phi_{1/2-\epsilon,1/2+\epsilon}(d)$, and boundary on the curves $(l_{1/2-\epsilon}, l_{1/2+\epsilon}, b_i)$ in clockwise order such that $T$ does not contain the origin. 
We allow $T$ to be a constant triangle, so $c$ is also of type-I.
We say the remaining intersection points are of type-II.
 
We say that a generator $x$ of $CF(\Sigma_{\alpha_{1/2-\epsilon}},\Sigma_{\beta}; \FF_2[U]) $ is of type-I if $\pi_A(x)$ contains a type-I intersection point of $l_{1/2-\epsilon} \cap b_i$ for some $i$.
Otherwise, it is of type-II. We define a $\mathbb{F}_2[U]$-linear map 
\begin{align}
\Psi_{\epsilon} \co & CF(\Sigma_{\alpha_{1/2-\epsilon}},\Sigma_{\beta}; \FF_2[U])  \to CF(\Sigma_{\alpha_{1/2+\epsilon}},\Sigma_{\beta}; \FF_2[U])  \label{eq:PSI}\\
 x & \mapsto  \begin{cases}\Phi_{\epsilon}(x) \qquad &\text{if $x$ is a type-I generator}  \nonumber \\
 U\Phi_{\epsilon}(x) \qquad &\text{if $x$ is a type-II generator} \nonumber \end{cases}
\end{align}
and let $\Psi_{\epsilon}|_{U=0}\co CF_{ann}(\Sigma_{\alpha_{1/2-\epsilon}},\Sigma_{\beta})  \to CF_{ann}(\Sigma_{\alpha_{1/2+\epsilon}},\Sigma_{\beta})$ be defined by setting $U=0$.
In Example \ref{e:continuation}, these module isomorphisms send $c$ to $c$, $d_i$ to $d_i'$, $e_i$ to $e_i'$, and $f$ to $Uf$, which becomes zero in the annular case.

\begin{lemma}\label{l:continuation}
For $\epsilon>0$ sufficiently small, $\Psi_{\epsilon}$ agrees with the chain map in the first row of \eqref{eq:Comm}. Therefore, the map induced by $\Psi_{\epsilon}|_{U=0}$ on cohomology is the axis-moving map.
\end{lemma}

\begin{proof}
First note that the generator $\tilde{c}_{\alpha_{1/2-\epsilon},\alpha_{1/2+\epsilon}}$ consists of the intersection point $c$ between $l_{1/2-\epsilon}$ and $l_{1/2+\epsilon}$ together with, for each $j \neq i$, the lowest-degree intersection point between $a_i$ and a small Hamiltonian translate of $a_i$, which is a mutual endpoint of the two arcs. When $\epsilon>0$ is sufficiently small, we have that
\[
\mu_2(x,\tilde{c}_{\alpha_{1/2-\epsilon},\alpha_{1/2+\epsilon}})|_{U=1}=x
\] and all the Floer triangles contributing to the product are local.
In fact, every such triangle lies in $\Sym^n(\CC) \cap \Hilb^{n,hor}(A_{\tau}^*)$ and is the map induced by taking the product of a local holomorphic triangle with one corner at $c$ with constant maps projecting to intersections between the undercrossing arcs $a_j$ such that $i\neq j$ and the bridges $b_i$.
Moreover, any local holomorphic triangle near $c$ is either a constant map to $c$, or admits a map of degree $1$ onto its $\pi_A$-image, which is contained in the bigon $B$.
It follows that $\mu_2(x,\tilde{c}_{\alpha_{1/2-\epsilon},\alpha_{1/2+\epsilon}})=x$ if the projection $\pi_A(x)$ contains a type-I intersection point between $l_{1/2-\epsilon}$ and some $b_i$ and 
$\mu_2(x,\tilde{c}_{\alpha_{1/2-\epsilon},\alpha_{1/2+\epsilon}})=Ux$ otherwise.
By setting $U=0$, we have shown that $\Psi_{\epsilon}|_{U=0}$ coincides with the axis-moving map.
\end{proof}

\subsection{The mapping cone between resolutions of a strongly invertible knot} \label{subsec:mapping}

Let $K$ be a strongly invertible knot together with a preferred choice of half-axis. We may choose a strongly invertible bridge diagram which is \emph{intravergent} in the sense of Section~\ref{sec:symmetry}; that is, the symmetry on $K$ is given by rotation in the plane of the page, which preserves the diagram. Such a bridge diagram necessarily has an odd number of bridges, say $2n+1$. Note that this bridge diagram is not annular and, unlike every other bridge diagram in this paper, contains a bridge crossing the origin. We adopt the convention that the bridge crossing the origin fixed by the action is always $b_0$; meanwhile for $1 \leq i \leq n$ the bridge $b_i$ is interchanged by the symmetry with $b_{2n+1-i}$ as they are in the case of periodic knots, and similarly for the arcs $a_i$. An example of such an intravergent bridge diagram for the right-handed trefoil, matching the diagram of Figure~\ref{fig:strongly_invertible_one}, is shown in Figure~\ref{fig:stronginv}. Recall that taking the zero and one resolutions of the central crossing results in two periodic annular links, for which we may choose periodic bridge diagrams by separating $b_0$ into two bridges as in Figure~\ref{fig:periodicresolutions}. Taking the quotient of the resolutions by the action as in  Section~\ref{sec:symmetry} returns annular bridge diagrams for the two annularizations $\overline{K}_0$ and $\overline{K}_1$ associated to the diagram introduced in the same section, which differ by a slide of the undercrossing arc in the image of $a_0$ over the origin. We will sometimes refer to $\overline{K}_0$ and $\overline{K}_1$ as the \emph{quotient annular resolutions}. Figures~\ref{fig:zero} and \ref{fig:one} show these two bridge diagrams in the case of the trefoil, matching the diagrams of Figure~\ref{fig:strongly_invertible_two}.

\begin{figure}[ht]
\fontsize{15pt}{12pt}
\scalebox{.5}{
\begingroup%
  \makeatletter%
  \providecommand\color[2][]{%
    \errmessage{(Inkscape) Color is used for the text in Inkscape, but the package 'color.sty' is not loaded}%
    \renewcommand\color[2][]{}%
  }%
  \providecommand\transparent[1]{%
    \errmessage{(Inkscape) Transparency is used (non-zero) for the text in Inkscape, but the package 'transparent.sty' is not loaded}%
    \renewcommand\transparent[1]{}%
  }%
  \providecommand\rotatebox[2]{#2}%
  \newcommand*\fsize{\dimexpr\f@size pt\relax}%
  \newcommand*\lineheight[1]{\fontsize{\fsize}{#1\fsize}\selectfont}%
  \ifx\svgwidth\undefined%
    \setlength{\unitlength}{431.83946228bp}%
    \ifx\svgscale\undefined%
      \relax%
    \else%
      \setlength{\unitlength}{\unitlength * \real{\svgscale}}%
    \fi%
  \else%
    \setlength{\unitlength}{\svgwidth}%
  \fi%
  \global\let\svgwidth\undefined%
  \global\let\svgscale\undefined%
  \makeatother%
  \begin{picture}(1,0.52602225)%
    \lineheight{1}%
    \setlength\tabcolsep{0pt}%
    \put(0,0){\includegraphics[width=\unitlength,page=1]{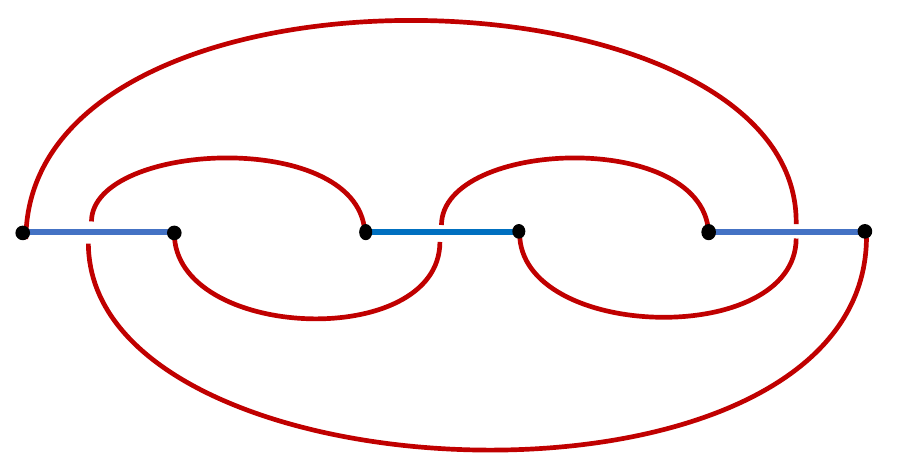}}%
    \put(0.51127238,0.23424215){\color[rgb]{0,0,0}\makebox(0,0)[lt]{\lineheight{1.25}\smash{\begin{tabular}[t]{l}$b_0$\end{tabular}}}}%
    \put(0.4863002,0.33150217){\color[rgb]{0,0,0}\makebox(0,0)[lt]{\lineheight{1.25}\smash{\begin{tabular}[t]{l}$a_0$\end{tabular}}}}%
    \put(0.92002737,0.28287215){\color[rgb]{0,0,0}\makebox(0,0)[lt]{\lineheight{1.25}\smash{\begin{tabular}[t]{l}$b_1$\end{tabular}}}}%
    \put(0.82276736,0.05680828){\color[rgb]{0,0,0}\makebox(0,0)[lt]{\lineheight{1.25}\smash{\begin{tabular}[t]{l}$a_1$\end{tabular}}}}%
  \end{picture}%
\endgroup%
}
\caption{An intravergent bridge diagram for a strongly invertible knot.}
\label{fig:stronginv}
\end{figure}
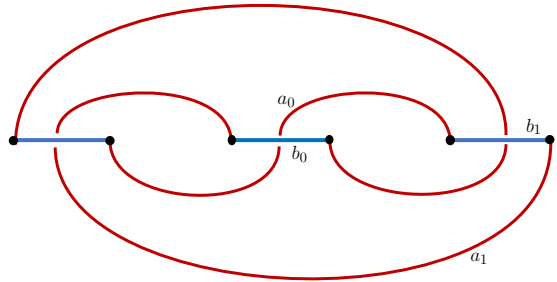

\begin{figure}[ht]

\scalebox{.5}{
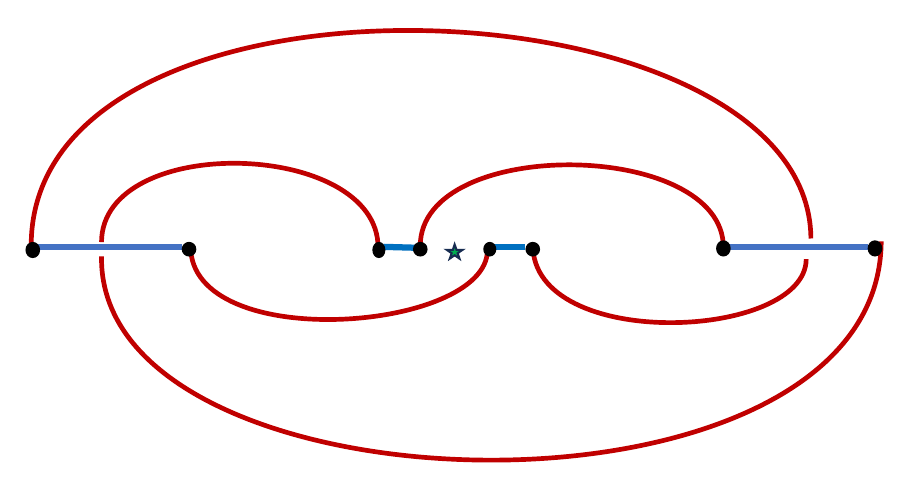
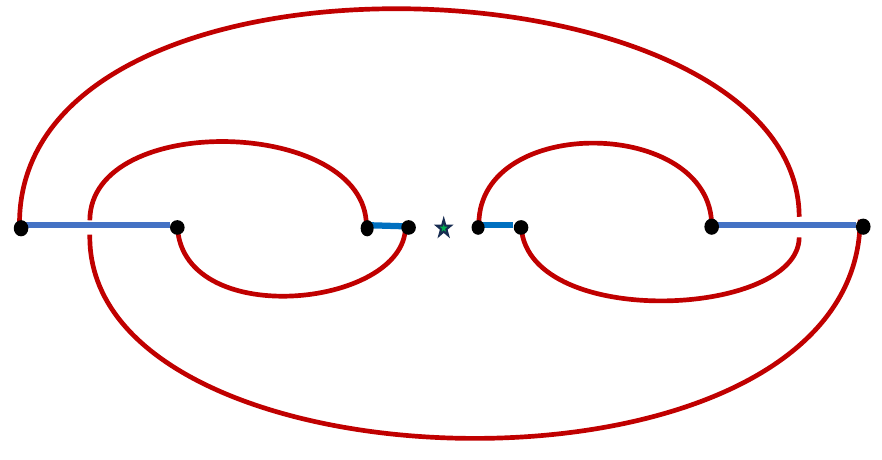}
\caption{Bridge diagrams for the result of resolving the central crossing of the diagram of Figure~\ref{fig:stronginv}. The zero-resolution appears on the left and the one-resolution appears on the right.}
\label{fig:periodicresolutions}
\end{figure}

\begin{figure}[ht]
\fontsize{15pt}{12pt}
\scalebox{.5}{
\begingroup%
  \makeatletter%
  \providecommand\color[2][]{%
    \errmessage{(Inkscape) Color is used for the text in Inkscape, but the package 'color.sty' is not loaded}%
    \renewcommand\color[2][]{}%
  }%
  \providecommand\transparent[1]{%
    \errmessage{(Inkscape) Transparency is used (non-zero) for the text in Inkscape, but the package 'transparent.sty' is not loaded}%
    \renewcommand\transparent[1]{}%
  }%
  \providecommand\rotatebox[2]{#2}%
  \newcommand*\fsize{\dimexpr\f@size pt\relax}%
  \newcommand*\lineheight[1]{\fontsize{\fsize}{#1\fsize}\selectfont}%
  \ifx\svgwidth\undefined%
    \setlength{\unitlength}{666.22077942bp}%
    \ifx\svgscale\undefined%
      \relax%
    \else%
      \setlength{\unitlength}{\unitlength * \real{\svgscale}}%
    \fi%
  \else%
    \setlength{\unitlength}{\svgwidth}%
  \fi%
  \global\let\svgwidth\undefined%
  \global\let\svgscale\undefined%
  \makeatother%
  \begin{picture}(1,0.31927708)%
    \lineheight{1}%
    \setlength\tabcolsep{0pt}%
    \put(0,0){\includegraphics[width=\unitlength,page=1]{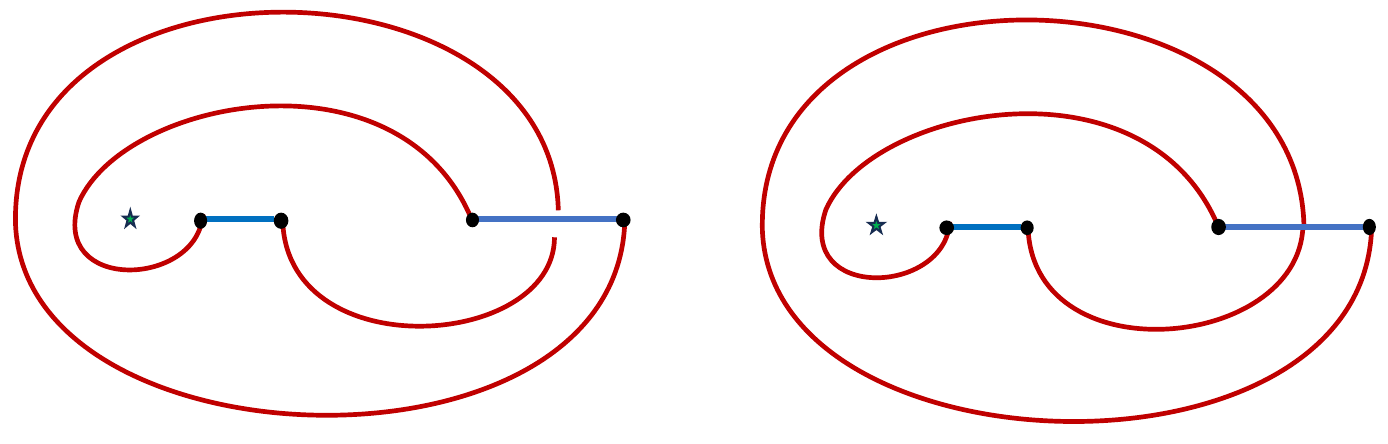}}%
    \put(0.66686748,0.16746985){\color[rgb]{0,0,0}\makebox(0,0)[lt]{\lineheight{1.25}\smash{\begin{tabular}[t]{l}$\omega$\end{tabular}}}}%
    \put(0.74698792,0.15060239){\color[rgb]{0,0,0}\makebox(0,0)[lt]{\lineheight{1.25}\smash{\begin{tabular}[t]{l}$\eta$\end{tabular}}}}%
    \put(0.86144574,0.13554214){\color[rgb]{0,0,0}\makebox(0,0)[lt]{\lineheight{1.25}\smash{\begin{tabular}[t]{l}$\gamma$\end{tabular}}}}%
    \put(0.92168672,0.1638554){\color[rgb]{0,0,0}\makebox(0,0)[lt]{\lineheight{1.25}\smash{\begin{tabular}[t]{l}$\delta$\end{tabular}}}}%
    \put(0.96084333,0.1379518){\color[rgb]{0,0,0}\makebox(0,0)[lt]{\lineheight{1.25}\smash{\begin{tabular}[t]{l}$\epsilon$\end{tabular}}}}%
  \end{picture}%
\endgroup%
}
\caption{The annular link diagram arising from taking the zero-resolution of central crossing of the bridge diagram of Figure~\ref{fig:stronginv} and then taking the quotient of the resulting periodic link. The green star as usual represents the axis of the rotational symmetry. On the right-hand side, intersection points between the bridges and undercrossing arcs are labelled for later use in the examples of Section~\ref{subsec:examples}.}
\label{fig:zero}
\end{figure}

\begin{figure}[ht]
\fontsize{15pt}{12pt}
\scalebox{.5}{
\begingroup%
  \makeatletter%
  \providecommand\color[2][]{%
    \errmessage{(Inkscape) Color is used for the text in Inkscape, but the package 'color.sty' is not loaded}%
    \renewcommand\color[2][]{}%
  }%
  \providecommand\transparent[1]{%
    \errmessage{(Inkscape) Transparency is used (non-zero) for the text in Inkscape, but the package 'transparent.sty' is not loaded}%
    \renewcommand\transparent[1]{}%
  }%
  \providecommand\rotatebox[2]{#2}%
  \newcommand*\fsize{\dimexpr\f@size pt\relax}%
  \newcommand*\lineheight[1]{\fontsize{\fsize}{#1\fsize}\selectfont}%
  \ifx\svgwidth\undefined%
    \setlength{\unitlength}{644.54850769bp}%
    \ifx\svgscale\undefined%
      \relax%
    \else%
      \setlength{\unitlength}{\unitlength * \real{\svgscale}}%
    \fi%
  \else%
    \setlength{\unitlength}{\svgwidth}%
  \fi%
  \global\let\svgwidth\undefined%
  \global\let\svgscale\undefined%
  \makeatother%
  \begin{picture}(1,0.32129516)%
    \lineheight{1}%
    \setlength\tabcolsep{0pt}%
    \put(0,0){\includegraphics[width=\unitlength,page=1]{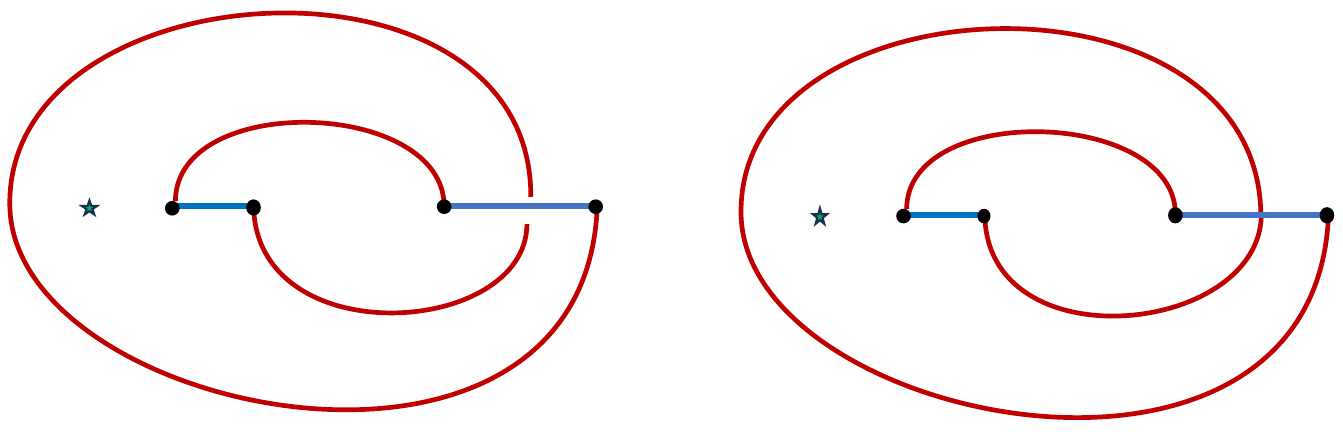}}%
    \put(0.65504361,0.13947698){\color[rgb]{0,0,0}\makebox(0,0)[lt]{\lineheight{1.25}\smash{\begin{tabular}[t]{l}$\omega'$\end{tabular}}}}%
    \put(0.72478201,0.16874223){\color[rgb]{0,0,0}\makebox(0,0)[lt]{\lineheight{1.25}\smash{\begin{tabular}[t]{l}$\eta'$\end{tabular}}}}%
    \put(0.8592777,0.14134497){\color[rgb]{0,0,0}\makebox(0,0)[lt]{\lineheight{1.25}\smash{\begin{tabular}[t]{l}$\gamma'$\end{tabular}}}}%
    \put(0.91843081,0.16936491){\color[rgb]{0,0,0}\makebox(0,0)[lt]{\lineheight{1.25}\smash{\begin{tabular}[t]{l}$\delta'$\end{tabular}}}}%
    \put(0.97198001,0.14134499){\color[rgb]{0,0,0}\makebox(0,0)[lt]{\lineheight{1.25}\smash{\begin{tabular}[t]{l}$\epsilon'$\end{tabular}}}}%
  \end{picture}%
\endgroup%
}
\caption{The annular link diagram arising from taking the one-resolution of the central crossing of the bridge diagram of Figure~\ref{fig:stronginv} and then taking the quotient of the resulting periodic link. The green star as usual represents the axis of the rotational symmetry. On the right-hand side, intersection points between the bridges and undercrossing arcs are labelled for later use in the examples of Section~\ref{subsec:examples}.}
\label{fig:one}
\end{figure}

For $i=0,1$, let $\overline{\Sigma}_{\alpha,i}$ and $\overline{\Sigma}_{\beta,i}$ be the product Lagrangians corresponding to these quotient annular bridge diagrams such that 
$\AKh_{symp}(\overline{K}_i)=HF_{ann}(\Hilb^{n+1,hor}(A^*_{\tau}),\overline{\Sigma}_{\alpha,i},\overline{\Sigma}_{\beta,i})$. Here $\tau$ may be taken to be, for example, the configuration $[1/4, 1/2, 1, 2, \dots, 2n]$. In particular, the annular symplectic Khovanov homologies of $\overline{K}_0$ and $\overline{K}_1$ are computed in the same symplectic manifold; only the placement of $\Sigma_{\alpha,i}$ differs. Moreover, after choosing appropriate Hamiltonian perturbations of the Lagrangians, we may arrange $\overline{\Sigma}_{\alpha,i}$ and $\overline{\Sigma}_{\beta,i}$ such that
\begin{align}\label{eq:intersectionequal}
\overline{\Sigma}_{\alpha,0} \pitchfork \overline{\Sigma}_{\beta,0} = \overline{\Sigma}_{\alpha,1} \pitchfork \overline{\Sigma}_{\beta,1}
\end{align}
so that there is a bijective correspondence between the generators in $CF_{ann}(\overline{\Sigma}_{\alpha,0},\overline{\Sigma}_{\beta,0})$ and those in $CF_{ann}(\overline{\Sigma}_{\alpha,1},\overline{\Sigma}_{\beta,1})$.
Generators in $CF_{ann}(\overline{\Sigma}_{\alpha,i},\overline{\Sigma}_{\beta,i})$ are distinguished by whether they contain the intersection point which occurs at the lefthand endpoint of the leftmost bridge, nearest to the origin. (This point is labelled $\omega$ in Figure~\ref{fig:zero} and $\omega'$ in Figure~\ref{fig:one}.) We let the bridge and the undercrossing arc which intersect at this point be $\overline{b}_0$ and $\overline{a}_0$ respectively, and we call a generator containing this intersection point a \emph{type-f generator}.

\begin{lemma}\label{l:subquotient}
In the chain complex $CF_{ann}(\overline{\Sigma}_{\alpha,0},\overline{\Sigma}_{\beta,0})$, the generators of type-f form a subcomplex. In the chain complex $CF_{ann}(\overline{\Sigma}_{\alpha,1},\overline{\Sigma}_{\beta,1})$, the generators which are not of type-f form a subcomplex.
\end{lemma}

\begin{proof} We first observe that, because the central crossing of an intravergent bridge diagram consists in all cases of a horizontal bridge overcrossing a vertical undercrossing arc, the local picture of the two resolutions and their quotients is always the same. In particular, for either quotient bridge diagram for a resolution of the central crossing of an intravergent bridge diagram, consider the connected components of the complement of the union of the bridges and undercrossing arcs in $\CC$. For the zero resolution, the component whose boundary contains, running clockwise, a subarc of $\overline{b}_0$, the intersection point between $\overline{a}_0$ and $\overline{b}_0$, and a subarc of $\overline{a}_0$ contains the origin, as in Figure~\ref{fig:zero}. For the one resolution, the component
which contains, running counterclockwise, a a subarc of $\overline{b}_0$, the intersection point between $\overline{a}_0$ and $\overline{b}_0$, and a subarc of $\overline{a}_0$ contains the origin, as in Figure~\ref{fig:one}.

Suppose for the sake of contradiction that $u$ is a Floer strip from a type-f generator to a non-type-f generator in $CF_{ann}(\overline{\Sigma}_{\alpha,0},\overline{\Sigma}_{\beta,0})$. Let $\pi_{\Sigma} \co \Sigma \to S$ and $v \co\Sigma \to A_{\tau}$ be the pair of maps obtained by applying the tautological correspondence of Remark~\ref{r:tauto} to $u$.
By the boundary conditions of $v$ and the open mapping theorem, the image of $\pi_A \circ v$ must intersect the origin, implying that $u$ intersects the divisor $D_o$. As we are working in the annular theory, this is a contradiction. Therefore, there is no such strip $u$ and type-f generators form a subcomplex. The argument for the non-existence of differentials from non-type-f generators to type-f generators in $CF_{ann}(\overline{\Sigma}_{\alpha,1},\overline{\Sigma}_{\beta,1})$ is similar, and uses the fact that a strip incoming to a type-f generator must necessarily intersect $D_o$ in this case.
\end{proof}

Let $(C_0, \partial_{C_0})$ be the subcomplex in $CF_{ann}(\overline{\Sigma}_{\alpha,0},\overline{\Sigma}_{\beta,0})$ generated by the type-f generators. We denote the quotient complex $CF_{ann}(\overline{\Sigma}_{\alpha,0},\overline{\Sigma}_{\beta,0})/C_0$  by $(N_0, \partial_{N_0})$, and treat it as being generated by the the non-type-f generators. Then $CF_{ann}(\overline{\Sigma}_{\alpha,0},\overline{\Sigma}_{\beta,0}) \simeq C_0 \oplus N_0$ with differential given by
\[\left( \begin{matrix} \partial_{C_0} & \partial_{N_0}^{C_0} \\ 0 & \partial_{N_0}\end{matrix}\right)\]
where $\partial_{N_0}^{C_0}$ consists of those terms of the differential which connect a generator of in $N_0$ to one in $C_0$. We represent this schematically as the total complex of
\[ N_0 \xrightarrow{\partial_{N_0}^{C_0}} C_0.\]
Similarly, we denote the subcomplex in $CF_{ann}(\overline{\Sigma}_{\alpha,1},\overline{\Sigma}_{\beta,1})$ generated by non-type-f generators by $N_1$
and the  corresponding quotient complex by $C_1$, generated by the type-f generators, so that $CF_{ann}(\overline{\Sigma}_{\alpha,1},\overline{\Sigma}_{\beta,1}) \simeq C_1 \oplus N_1$ with differential
\[\left( \begin{matrix} \partial_{C_1} & 0 \\ \partial_{N_1}^{C_1}  & \partial_{N_1}\end{matrix}\right)\]
represented schematically as the total complex of
\[ N_1 \xleftarrow{\partial^{N_1}_{C_1}} C_1.\]
Note that the generators of $C_0$ and $C_1$ are in bijection, and the generators of $N_0$ and $N_1$ are in bijection. Indeed, let $\overline{a}_{0,0}$ be the undercrossing arc that meets the center intersection point in the bridge diagram we have constructed for $\overline{K}_0$, and $\overline{a}_{0,1}$ be its counterpart on the diagram for $\overline{K}_1$. 


\begin{lemma}\label{l:axis-move} When the arcs $\overline{a}_{0,0}$ and $\overline{a}_{0,1}$ differ by a sufficiently small isotopy, there are canonical isomorphisms $C_0 \simeq C_1$, $N_0 \simeq N_1$ as chain complexes.
Moreover,  the mapping cone of the axis-moving map $Cone (CF_{ann}(\overline{\Sigma}_{\alpha,0},\overline{\Sigma}_{\beta,0}) \to CF_{ann}(\overline{\Sigma}_{\alpha,1},\overline{\Sigma}_{\beta,1}))$ has total complex of the form
\[
\xymatrix{ 
C_0 & C_1[-1]  \ar[d]_{\partial_{C_1}^{N_1} }\\
N_0 \ar[r]_{\id_N} \ar[u]_{\partial_{N_0}^{C_0}} & N_1[-1].
}
\]
where $[-1]$ denotes a upward grading shift by one and $\id_N$ is the map induced by the canonical identification of $N_0$ and $N_1$. In other words the axis-moving map is the identity on generators in $N_0$ and zero on generators in $C_0$, and the total differential on the resulting mapping cone $C_0\oplus N_0\oplus C_1[-1] \oplus N_1[-1]$ is
\[ \left( \begin{matrix} \partial_{C_0} & \partial_{N_0}^{C_0} & 0 & 0 \\ 0 & \partial_{N_0}& 0 & 0 \\ 0 & 0&  \partial_{C_1} & 0 \\ 0 & \id_N & \partial_{N_1}^{C_1}  & \partial_{N_1} \end{matrix}\right).\]

\noindent As a consequence, the complex $S$ represented schematically by
\[
\xymatrix{ 
C_0 & \\
N_0 \ar[r]_{\id_N} \ar[u]_{\partial_{N_0}^{C_0}} & N_1[-1].
}
\]
is a subcomplex of the mapping cone complex, with quotient complex $C_1[-1]$.
\end{lemma}

\begin{proof}
When the arcs the arcs $\overline{a}_{0,0}$ and $\overline{a}_{0,1}$ differ by a sufficiently small isotopy, we are in exactly the situation of Lemma \ref{l:continuation} for $\epsilon$ sufficiently small. In this case, in the terminology of Lemma \ref{l:continuation}, $N_0$ is precisely the subspace generated by type-I intersection points, and $C_0$ is the subspace generated by type-II intersection points.

By Lemma \ref{l:continuation}, we know that $\Psi_\epsilon$ (defined by \eqref{eq:PSI}) agrees with the chain map in the first row of \eqref{eq:Comm}.
By Lemma \ref{l:subquotient}, we know that $C_0,C_1,N_0,N_1$ are all chain complexes.
Both together imply that $\Phi_{\epsilon}|_{C_0}:C_0 \to C_1$ and  $\Phi_{\epsilon}|_{N_0}:N_0 \to N_1$ are chain maps (see the definition of $\Phi_{\epsilon}$ in \eqref{eq:Phi}).
Since $\Phi_{\epsilon}|_{C_0}$ and $\Phi_{\epsilon}|_{N_0}$ are bijective, this implies that the maps are isomorphisms of chain complexes.
We denote $\Phi_{\epsilon}|_{N_0}$ by $\id_N$.
Since $\Phi_{\epsilon}|_{U=0}$ is $\id_N$ on $N_0$ and zero on $C_0$, the remaining statements are algebraic rephrasings of this fact together with Lemma \ref{l:subquotient}.

\end{proof}

We next observe that the inclusion map from $C_0$ to $S$ is a chain homotopy equivalence, and thus that $H(S) \simeq H(C_0)$. Ergo, the short exact sequence of cochain complexes 
\begin{align}\label{eq:ses}
0 \to S \to Cone (CF_{ann}(\overline{\Sigma}_{\alpha,0},\overline{\Sigma}_{\beta,0}) \to CF_{ann}(\overline{\Sigma}_{\alpha,1},\overline{\Sigma}_{\beta,1})) \to C_1[-1] \to 0
\end{align}
induces a long exact sequence on cohomology
\begin{align} \label{eq:lescone}
\dots \to H(C_0) \to H(Cone (CF_{ann}(\overline{\Sigma}_{\alpha,0},\overline{\Sigma}_{\beta,0}) \to CF_{ann}(\overline{\Sigma}_{\alpha,1},\overline{\Sigma}_{\beta,1}))) \to H(C_1)[-1] \to \dots
\end{align}
We claim that this long exact sequence splits, as follows.

\begin{proposition} \label{prop:splits}
The connecting map in the long exact sequence on cohomology induced by the short exact sequence of cochain complexes \eqref{eq:ses} is zero. As a consequence, there is an isomorphism
\begin{align}
H( Cone (CF_{ann}(\overline{\Sigma}_{\alpha,0},\overline{\Sigma}_{\beta,0}) \to CF_{ann}(\overline{\Sigma}_{\alpha,1},\overline{\Sigma}_{\beta,1}))) \simeq H(C_0) \oplus H(C_1)[-1]. \label{eq:splits}
\end{align}
\end{proposition}

\begin{proof}
Let $F$ be the composition of the maps \[C_1 \xrightarrow{\partial_{C_1}^{N_1}} N_1 \xrightarrow{\id_N} N_0 \xrightarrow{\partial_{N_0}^{C_0}} C_0.\]
Here again $\id_N$ is the map induced by the canonical identification of $N_0$ and $N_1$. Let $x$ be a cycle in $C_1$. Then we may compute the connecting map in the long exact sequence \eqref{eq:lescone} as follows. In the short exact sequence \eqref{eq:ses}, the element $x$ in $C_1[-1]$ pulls back to the element $x$ in the copy of $C_1$ inside the mapping cone in the second term. Applying the differential inside the mapping cone produces $\partial_{C_1}^{N_1}(x)$. The pullback of this element to the subcomplex $S$ is simply $\partial_{C_1}^{N_1}(x) \subset N_1$. To apply the chain homotopy equivalence between $S$ and $C_0$ to $\partial_{C_1}^{N_1}(x)$, we use $\id_N$ to regard this term as an element of $N_0$ and apply $\partial_{N_0}^{C_0}$. Therefore for $x$ a cycle of $C_1$, we have that the image of $[x]$ under the connecting map is $[F(x)] \in H(S) \simeq H(C_0)$.
Our goal is to show that $[F(x)]=0$. Indeed, when the arcs $\overline{a}_{0,0}$ and $\overline{a}_{0,1}$ differ by a sufficiently small isotopy, we can show the stronger statement that $F=0$.

Let $x_1$ be a cycle in $C_1$ and $y_0$ be a generator of $C_0$.
Let $y_1:=\Phi^{-1}(y_0) \in C_1$, where $\Phi$ is the map defined in \eqref{eq:Phi}.
We claim that the coefficient of $y_0$ in $F(x_1)$ is precisely the coefficient of $y_1$ in $\partial ^2 (x_1)$, where $\partial$ is the differential of $CF(\overline{\Sigma}_{\alpha,1},\overline{\Sigma}_{\beta,1}))|_{U=1}$, and that this coefficient is therefore zero. For since $x_1$ is a cycle in $C_1$, the differential $\partial(x_1)$ is $\partial_{C_1}^{N_1}(x_1) \in N_1$, which we recall is canonically isomorphic to $N_0$. The claim about the coefficients follows. \end{proof}

\begin{remark}\label{r:new}
By Lemma \ref{l:axis-move}, we know that $C_0 \simeq C_1$ as chain complexes, so the left-hand side of \eqref{eq:splits} is a direct sum of two isomorphic homology groups offset by a grading shift. As far as we are aware this result is not known for combinatorial Khovanov homology.
\end{remark}

We end this section with a discussion of the shift in the winding grading with respect to the axis-moving map in Lemma \ref{l:axis-move}.

\begin{lemma}
Let $x_0, x_1 \in CF_{ann}(\overline{\Sigma}_{\alpha,0},\overline{\Sigma}_{\beta,0})$. If $x_0,x_1 \in C_0$ or $x_0,x_1 \in N_0$, then $w(x_0,x_1)=w(\Phi(x_0),\Phi(x_1))$.
If $x_0 \in C_0$ and $x_1 \in N_0$, then $w(x_0,x_1)=w(\Phi(x_0),\Phi(x_1))-2$.
\end{lemma}

\begin{proof}
As in the text preceding Definition \ref{d:winding}, consider the map $\pi_{\Delta}:\cY_{n+1,\overline{\tau}} \to \Sym^{n+1}(\CC)$.
Let $\ul{\bar{a}}_0$ be the collection of undercrossing arcs for the zero-resolution, and $\ul{\bar{a}}_1$ be the collection of undercrossing arcs for the one-resolution.
Let $\ul{\bar{b}}$ be the collection of bridges $b_i$ for both resolutions. We may arrange the undercrossing arcs such that the intersection points between $\ul{\bar{a}}_0 \cap \ul{\bar{b}}$ agree with those in $\ul{\bar{a}}_1 \cap \ul{\bar{b}}$, unlike our usual custom of letting them be offset by a small isotopy. Given this choice, we notice that  for any generator $x \in CF_{ann}(\overline{\Sigma}_{\alpha,0},\overline{\Sigma}_{\beta,0})$, we have
\[\pi_{\Delta}(\Phi(x))=\pi_{\Delta}(x) \in \Sym(\bar{\ul{a}}_0)\cap \Sym(\bar{\ul{b}}) = \Sym(\bar{\ul{a}}_1) \cap \Sym(\bar{\ul{b}}).\]

For $y_0,y_1$ two intersection points in $\Sym(\bar{\ul{a}}_0)\cap \Sym(\bar{\ul{b}}) = \Sym(\bar{\ul{a}}_1) \cap \Sym(\bar{\ul{b}})$, let $p_{a,0}(y_0,y_1)\co [0,1] \to \Sym(\bar{\ul{a}}_0)$ be the unique path up to homotopy relative to the endpoints from $y_0$ to $y_1$, and likewise for $p_{a,1}(y_0,y_1)\co[0,1] \to \Sym(\bar{\ul{a}}_1)$ and $p_{b}(y_0,y_1)\co[0,1] \to \Sym(\bar{\ul{b}})$. 
Let  $x_0,x_1 \in CF_{ann}(\overline{\Sigma}_{\alpha,0},\overline{\Sigma}_{\beta,0})$ be  two generators and $y_i:=\pi_{\Delta}(x_i)$.
It follows easily from definition that the relative winding grading
 $w(x_0,x_1)$ is twice the winding number of the loop 
\[p_{a,0}(y_0,y_1)\#p_{b}(y_1,y_0)\]
about the divisor $d_o$.
Similarly,  $w(\Phi(x_0),\Phi(x_1))$ is twice the winding number of the loop \[p_{a,1}(y_0,y_1)\#p_{b}(y_1,y_0)\] about the divisor $d_o$.
As a result, $w(\Phi(x_0),\Phi(x_1))-w(x_0,x_1)$  is twice the winding number of the loop \[p_{a,1}(y_0,y_1)\#p_{a,0}(y_1,y_0)\] about the divisor $d_o$.

We can write $p_{a,1}(y_0,y_1)(t)$ as the symmetric product of a tuple of paths $(v_0(t),\dots,v_n(t))$ where $v_i(t)$ lies in the $i^{th}$ arc of $\ul{\bar{a}}_1$.
Since $\ul{\bar{a}}_0$ and $\ul{\bar{a}}_1$ only differ by their zeroth undercrossing arc, we can choose $p_{a,0}(y_1,y_0)(t)$ to be of the form $(v_0'(t),v_1(1-t)\dots,v_n(1-t))$.
It is therefore clear that $w(\Phi(x_0),\Phi(x_1))$ is twice the winding number of $v_0 \# v_0'$ about the origin in $\CC$.
The result now follows from a simple calculation.
\end{proof}
We observe that, because the symplectic-axis-moving map is zero on generators of $C_0$, this implies that the axis-moving map itself has a constant grading shift.

\subsection{Localization for strongly invertible knots} \label{subsec:strongly-loc}

In this section we use the tools of Sections~\ref{subsec:axis} and \ref{subsec:mapping} to complete the proof of Theorem~\ref{thm:stronglyinvertible} modulo the discussion of stable normal trivializations in Section~\ref{sec:normal-trivs}. As in the previous section, let $K$ be a strongly invertible knot, and choose an intravergent bridge diagram for $K$ with $2n+1$ bridges as in the example of Figure~\ref{fig:stronginv}. Let $\tau = [-2n,\dots,-1, -1/2,1/2,1,\dots,2n]$ be the associated configuration. Let $(\Hilb^{2n+1, hor}(A_{\tau}), \Sigma_{\alpha}, \Sigma_{\beta})$ be the symplectic manifold and Lagrangians associated to the bridge diagram.  We will be interested in the involution $\iota_{A,H}$ on the Hilbert scheme induced by $\iota_A(u,v,z) = (u,v,-z)$. Note that after possibly averaging the symplectic form this is a symplectic involution which preserves the Lagrangians setwise. Indeed, we may perturb $\Sigma_{\beta}$ compatibly with the projection map to $\CC$  so that $\Sigma_{\alpha} \pitchfork \Sigma_{\beta}$ and so that above any intersection point between the interiors of the bridges and the undercrossing arcs there are two intersection points between the Lagrangian matching spheres in $A_{\tau}$ differing in degree by one, as explained in Section~\ref{subsec:bridge}. Indeed, we may insist that this perturbation be equivariant with respect to the map on the Hilbert scheme $\iota_{A,H}$ induced by $\iota_A$, so that $\Sigma_{\beta}$ is still preserved by $\iota_{A,H}$ after perturbation.

There is a special intersection point between the bridges and undercrossing arcs in an intravergent bridge diagram at the origin, which we denote $\lambda$. Its two lifts in $\pi_A^{-1}(0) \subseteq A_{\tau}$ will be denoted $\lambda_0$ and $\lambda_1$. Note that every generator in $\Sigma_{\alpha} \cap \Sigma_{\beta} \subseteq \Conf^{2n+1}(A_{\tau}) \cap \Hilb^{n, hor}(A_{\tau})$ which is fixed by $\iota_{A,H}$ must contain exactly one of $\lambda_0$ and $\lambda_1$ exactly once, and otherwise consist of pairs of intersection points between the Lagrangian matching spheres in $A_{\tau}$ which are interchanged by the map induced by $\iota_A$. Note further that there are two pseudo-holomorphic (with respect to a generic choice of almost complex structure) strips in $\pi_A^{-1}(0)$ connecting $\lambda_1$ and $\lambda_0$. We call these strips $v_1$ and $v_2$. 

We furthermore form bridge diagrams for the two annular quotients of $K$ as described in Section~\ref{subsec:mapping} and the example of Figures~\ref{fig:zero} and \ref{fig:one}. Let the configuration used in both diagrams be $\overline{\tau}=[1/4,1/2,1,\dots, 2n]$ as previously, so that there is a double branched cover $A_{\tau} \setminus \pi_A^{-1}(B_{1/3}(0)) \rightarrow A_{\overline{\tau}} \setminus \pi_A^{-1}(B_{1/3}(0)).$ Here $B_{1/3}(0)$ denotes a neighborhood of the origin with radius $1/3$. Using the notation of Section~\ref{subsec:mapping}, observe that after a compatible Hamiltonian perturbation there is a bijection between the union of the type-f intersection points for $\overline{K}_0$ in $\overline{\Sigma}_{\alpha,0} \pitchfork \overline{\Sigma}_{\beta,0}$ together with those for $\overline{K}_1$ in $\overline{\Sigma}_{\alpha,1} \pitchfork \overline{\Sigma}_{\beta,1}$, and the fixed intersection points in $\Sigma_{\alpha} \pitchfork \Sigma_{\beta}$ upstairs as follows. Given a type-f intersection for $\overline{K}_0$ consisting of a union of intersection points between the Lagrangian matching spheres in $A_{\overline{\tau}}$, remove the intersection point $\omega$ between $\overline{a}_0$ and $\overline{b}_0$, lift the remaining points along the branched covering map to their symmetric lifts in $A_{\tau}$, and take the union of these symmetric lifts and adjoin $\lambda_0$. For type-f generators for $\overline{K}_1$ we similarly lift, but adjoin $\lambda_1$ instead. In particular this gives an identification of graded modules
\begin{align}CF(\Hilb^{2n+1,hor}(A_{\tau})^{fix}, \Sigma_{\alpha}^{fix}, \Sigma_{\beta}^{fix}) \simeq C_0 \oplus C_1[-1]\label{eq:bijection}\end{align}
where $C_0$ and $C_1$ are the subcomplex and quotient complex of type-f generators in their respective annular cochain groups as in Section~\ref{subsec:mapping}. 

We now begin with the following observation about Floer strips whose images lie in the fixed set of $\iota_{A,H}$ and therefore contribute to the Lagrangian Floer differential on $(\Hilb^{2n+1, hor}(A_{\tau})^{fix}, \Sigma_{\alpha}^{fix}, \Sigma_{\beta}^{fix})$.

\begin{lemma} \label{lemma:fixed} Let $K$ be a strongly invertible knot together with a choice of intravergent bridge diagram for $K$, and $(\Hilb^{2n+1, hor}(A_{\tau}), \Sigma_{\alpha}, \Sigma_{\beta})$ be the associated symplectic manifold and Lagrangians, as above. There is no non-trivial Floer strip passing through the fibre over zero which contributes to the differential of the Lagrangian Floer complex of the fixed point set of the involution $\iota_{A,H}$; more precisely, applying the tautological correspondence to any Floer strip in the differential of the fixed set returns either the union of one of the strips $v_1$ and $v_2$ with a union of constant strips in $A_{\tau}$, or the union of a constant strip at one of $\lambda_0$ or $\lambda_1$ with a strip with image in $A_{\tau}^*$ which is the lift of a strip in $A_{\overline{\tau}}^*$ along the branched covering map away from a neighborhood of the origin.
\end{lemma}

\begin{proof} For simplicity, let $S = \RR \times [0,1]$.
Let $u \co S \to \Fix(\iota_{A,H})$ be a Floer strip contributing to the differential. We assume we have already perturbed the Lagrangians to be transverse, and that therefore this strip may be taken to be pseudo-holomorphic with respect to an appropriate family of almost complex structures.

Let $\pi_{\Sigma} \co \Sigma \to S$ and $v \co \Sigma \to A_{\tau}$ be the pair of maps obtained by applying the tautological correspondence of Remark~\ref{r:tauto} to $u$.
For all but finitely many $z \in S$, there is a precisely one $\tilde{z} \in \Sigma$ such that $\pi_{\Sigma}(\tilde{z})=z$ and $v(\tilde{z}) \in \pi_A^{-1}(0)$.
This is because for $u(z)$ to lie in $\Fix(\iota_{A,H})$, we must have $\iota_A(v(\pi_{\Sigma}^{-1}(z)))=v(\pi_{\Sigma}^{-1}(z))$. This means that $v(\pi_{\Sigma}^{-1}(z)) \cap \pi_A^{-1}(0) \neq \emptyset$.
However, since the horizontal divisor is deleted, $v(\pi_{\Sigma}^{-1}(z)) \cap \pi_A^{-1}(0)$ can only be a singleton for any $z \in S$.
For some $z \in S$, there might be more than one lift in $\Sigma$ whose $v$-image lies in $\pi_A^{-1}(0)$. 
However, this occurs only if $z \in u^{-1}(D_{HC} \cap \Hilb^{2n+1,hor}(A_{\tau}))$ which is a finite set.
We denote the map $z \mapsto \tilde{z}$ by $\psi$, which is well-defined away from a finite set in $S$.

Let $S^o$ be the image of $\psi$. Note that the monodromy action on $\Sigma$ as a branched cover of $S$ preserves $S^o$.
Therefore, the closure $\overline{S^o}$ of $S^o$ is a component of $\Sigma$, and is biholomorphic to $S$ under $\pi_{\Sigma}$. We can therefore extend the domain of $\psi$ to be $S$, and $\psi$ is precisely the inverse of $\pi_{\Sigma}|_{\overline{S^o}}$.
By definition, $v|_{\overline{S^o}}$ is a holomorphic map to $\pi_A^{-1}(0)$ with boundary on the vanishing circles.
Therefore, it is either one of the two Floer strips $v_1$ and $v_2$ in $\pi_A^{-1}(0)$, or is a constant map.

We now consider the intersection $v(\Sigma \setminus \overline{S^o}) \cap \pi_A^{-1}(0)$, with the goal of proving that it is empty. For suppose it is nonempty, such that we have some $w \in \Sigma \setminus \overline{S^o}$ with $v(w) \in \pi_A^{-1}(0)$, then since $u$ does not intersect the horizontal divisor, we must have $v(w)=v(\psi(\pi_{\Sigma}(w)))$. There are two cases, distinguished by whether $w$ is a critical point of $\pi_{\Sigma}$ or not. Note that if it is not, there is another $w' \in \Sigma \setminus \overline{S^o}$ such that $v(w)=v(w')=v(\psi(\pi_{\Sigma}(w)))$ because $\iota_A(v(\pi_{\Sigma}^{-1}(z)))=v(\pi_{\Sigma}^{-1}(z))$. We will argue that for a sufficiently generic choice of domain dependent complex structures, neither of these two cases can actually occur.

We keep $\pi_{\Sigma}$ fixed.
For a generic choice of domain dependent family of complex structures which keep the complex structure near $\pi_A^{-1}(0)$ fixed and the corresponding pseudo-holomorphic map $v$, we want to argue that there is no $w \in \crit(\pi_{\Sigma})$ such that $v(w)=v(\psi(\pi_{\Sigma}(w)))$, and furthermore there is no pair of points $w \neq w'$ which are not critical points of $\pi_{\Sigma}$ such that  $v(w)=v(w')=v(\psi(\pi_{\Sigma}(w)))$. Toward this end, observe that since the complex structure near $\pi_A^{-1}(0)$ is fixed, the map $v|_{\overline{S^o}}$ remains the same under generic perturbation of complex structure. If the image of $v|_{\overline{S^o}}$ is a point, then we can choose a generic complex structure such that $v|_{\Sigma \setminus \overline{S^o}}$ avoids that point.
Even if the image of $v|_{\overline{S^o}}$ is not a point, a simple dimension count of the matching condition implies that a generic choice of complex structure is sufficient to avoid both the aforementioned cases.

As a result, we have $v(\Sigma \setminus \overline{S^o}) \cap \pi_A^{-1}(0) = \emptyset$. This implies that either $v|_{\overline{S^o}}$ is one of the two Floer strips $v_1$ and $v_2$ contained in $\pi_A^{-1}(0)$ and $v(\Sigma \setminus \overline{S^o})$ is a union of Lagrangian intersection points, or that $v|_{\overline{S^o}}$ is a constant map at $\lambda_0$ or $\lambda_1$ in $\pi_A^{-1}(0)$ and $v|_{\Sigma \setminus \overline{S^o}}$ is a strip in $A_{\tau}^*$. Indeed, the composition of $v|_{\Sigma \setminus \overline{S^o}}$ with the branched covering map which exists between $A_{\tau}*$ and $A_{\overline{\tau}^*}$ away from a neighborhood of the origin together with $\pi_{\Sigma}|_{\Sigma \setminus \overline{S^o}}$ tautologically corresponds to a Floer strip in $\Hilb^{n+1,hor}(A^*_{\overline{\tau}})$. \end{proof}

This allows us to establish that the differential of the Lagrangian Floer cochain complex in the fixed set of $\iota_{A,H}$ coincides with the cone of the axis-moving map connecting $\overline{K}_0$ and $\overline{K}_1$.

\begin{theorem} \label{thm:samediff}
The Floer cohomology of the fixed point sets $(\Hilb^{2n+1, hor}(A_{\tau})^{fix}, \Sigma_{\alpha}^{fix}, \Sigma_{\beta}^{fix})$ of $\iota_{A,H}$ isomorphic to the homology of the mapping cone of the axis-moving map from the $\AKh_{symp}(\overline{K}_0)$ to $\AKh_{symp}(\overline{K}_1)$.
\end{theorem}

\begin{proof} We recall from Proposition~\ref{prop:splits} that the homology of the cone of the axis-moving map is precisely
$H(C_0)\oplus H(C_1)[-1]$, where 
 $C_0$ is the subcomplex of type-f generators in $CF_{ann}(\Hilb^{n+1,hor}(A_{\overline{\tau}}), \overline{\Sigma}_{\alpha,0}, \overline{\Sigma}_{\beta,0})$ and $C_1$ is the quotient complex of type-f generators in $CF_{ann}(\Hilb^{n+1,hor}(A_{\overline{\tau}}), \overline{\Sigma}_{\alpha,1}, \overline{\Sigma}_{\beta,1})$. We have already seen in \eqref{eq:bijection} that there is an isomorphism of modules 
\begin{align*}CF(\Hilb^{2n+1,hor}(A_{\tau})^{fix}, \Sigma_{\alpha}^{fix}, \Sigma_{\beta}^{fix}) \simeq C_0 \oplus C_1[-1]\end{align*}
It suffices to check that the strips counted in the Floer differential on the left-hand side are exactly those which appear in the differentials $\partial_{C_0}$ or $\partial_{C_1}$ on the right-hand side.

By Lemma~\ref{lemma:fixed}, there are two types of Floer solutions in the fixed set. The first type has a component in the tautological correspondence which maps to one of the two Floer strips $v_1$ and $v_2$ in $\pi_A^{-1}(0)$. These types come in pair and cancel each other. The second type has a component in the tautological correspondence which maps to a constant map at $\lambda_0$ or $\lambda_1$ in $\pi_A^{-1}(0) \subseteq A_{\tau}$. If we forget this component, the remainder of the components in the tautological correspondence determine a Floer solution in $\Hilb^{n, hor}(A_{\overline{\tau}}^*)$, which counts for the differential in $C_0$ or $C_1$ according whether the constant map was at $\lambda_0$ or $\lambda_1$.

In the other direction, suppose $u$ is a Floer solution in $\Hilb^{n,hor}(A_{\overline{\tau}}^*)$ representing a differential from a type-f generator to a type-f generator. Then $u$ which is associated under the tautological correspondence to a constant map at the intersection point $\omega$ or $\omega'$ closest to the origin and a pseudo-holomorphic map $v': \Sigma' \rightarrow A_{\overline{\tau}}^*$ with image missing a small neighborhood of the origin. We may lift $v'$ to $A_{\tau}^*$ and adjoin either a constant map at $\lambda_0$ or a constant map at $\lambda_1$ according to whether we started in $C_0$ or $C_1$ to obtain a differential in the fixed set.

As a result, the Floer cochain complex in the fixed set is isomorphic to $C_0 \oplus C_1[-1]$. The claim follows.\end{proof}

We can now briefly conclude the proof of Theorem~\ref{thm:stronglyinvertible}.

\begin{proof}[Proof of Theorem~\ref{thm:stronglyinvertible}] Let $K$ be strongly invertible, and choose an intravergent bridge diagram for $K$ as in the arguments of this section so that 
\[ Kh_{symp}(K) \simeq HF(\Hilb^{2n+1, hor}(A_{\tau}), \Sigma_{\alpha}, \Sigma_{\beta}).\]
By the arguments of Section~\ref{sec:normal-trivs}, specifically Theorem~\ref{thm: stable normal trivialization}, the triple $(\Hilb^{2n+1, hor}(A_{\tau}), \Sigma_{\alpha}, \Sigma_{\beta})$ together with the involution $\iota_{A,H}$ admits a stable normal trivialization. By Theorem \ref{thm:samediff}, 
\[HF (\Hilb^{2n+1, hor}(A_{\tau})^{fix}, \Sigma_{\alpha}^{fix}, \Sigma_{\beta}^{fix})\]
is isomorphic to the homology of the cone of the axis-moving map from $AKh_{symp}(\overline{K}_0)$ to $AKh_{symp}(\overline{K}_1)$. The claimed localization spectral sequence now follows from Theorem~\ref{thm:localization} and Corollary~\ref{c:localization}. \end{proof}

\subsection{Examples} \label{subsec:examples} In this subsection, we give some example computations of the symplectic axis-moving map and the isomorphism of Theorem~\ref{thm:samediff}. For all bridge diagrams throughout, we will assume that we have perturbed the spheres $\Sigma_{b_i}$ so that an intersection point between the interiors of a bridge $b_i$ and an undercrossing arc $a_j$ lifts to two intersection points in $A_{\tau}$ between the spheres $\Sigma_{a_j}$ and $\Sigma_{b_i}$, differing in degree by one; if the intersection point on the plane has a label $k$ then these intersection points will be $k_0$ and $k_1$. When there is a relevant action we will assume we have carried out the perturbation equivariantly.

\subsubsection{The trefoil with quotient the unknot} We begin with the example of the right-handed trefoil with intravergent bridge diagram as in Figure~\ref{fig:stronginv}, so that the quotient knot is an unknot and the two quotient annular resolutions have bridge diagrams as in Figure~\ref{fig:zero} and Figure~\ref{fig:one}. This is sometimes called $3_1^+$ in the literature to distinguish it from the other possible choice of half-axis; c.f., e.g., \cite[Appendix A]{BI:equivariant-genera}. The intersection points on the strongly intravergent diagram which are relevant to computing the chain complex in the fixed set are exhibited in Figure~\ref{fig:stronginvgenerators}; for the diagrams of Figures~\ref{fig:zero} and \ref{fig:one} the intersection points are marked in the original figures. 
\begin{figure}[ht]
\fontsize{16pt}{14pt}
\scalebox{.5}{
\begingroup%
  \makeatletter%
  \providecommand\color[2][]{%
    \errmessage{(Inkscape) Color is used for the text in Inkscape, but the package 'color.sty' is not loaded}%
    \renewcommand\color[2][]{}%
  }%
  \providecommand\transparent[1]{%
    \errmessage{(Inkscape) Transparency is used (non-zero) for the text in Inkscape, but the package 'transparent.sty' is not loaded}%
    \renewcommand\transparent[1]{}%
  }%
  \providecommand\rotatebox[2]{#2}%
  \newcommand*\fsize{\dimexpr\f@size pt\relax}%
  \newcommand*\lineheight[1]{\fontsize{\fsize}{#1\fsize}\selectfont}%
  \ifx\svgwidth\undefined%
    \setlength{\unitlength}{427.02337646bp}%
    \ifx\svgscale\undefined%
      \relax%
    \else%
      \setlength{\unitlength}{\unitlength * \real{\svgscale}}%
    \fi%
  \else%
    \setlength{\unitlength}{\svgwidth}%
  \fi%
  \global\let\svgwidth\undefined%
  \global\let\svgscale\undefined%
  \makeatother%
  \begin{picture}(1,0.53007525)%
    \lineheight{1}%
    \setlength\tabcolsep{0pt}%
    \put(0,0){\includegraphics[width=\unitlength,page=1]{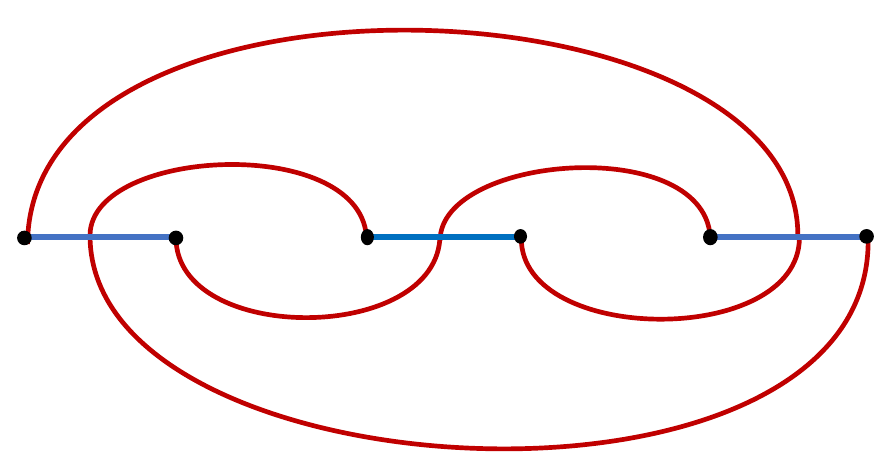}}%
    \put(0.04125028,0.27619203){\color[rgb]{0,0,0}\makebox(0,0)[lt]{\lineheight{1.25}\smash{\begin{tabular}[t]{l}$e^2$\end{tabular}}}}%
    \put(0.06094012,0.21934872){\color[rgb]{0,0,0}\makebox(0,0)[lt]{\lineheight{1.25}\smash{\begin{tabular}[t]{l}$d^2$\end{tabular}}}}%
    \put(0,0){\includegraphics[width=\unitlength,page=2]{stronginvgenerators.pdf}}%
    \put(0.46240605,0.27443614){\color[rgb]{0,0,0}\makebox(0,0)[lt]{\lineheight{1.25}\smash{\begin{tabular}[t]{l}$\lambda$\end{tabular}}}}%
    \put(0.95488723,0.28007522){\color[rgb]{0,0,0}\makebox(0,0)[lt]{\lineheight{1.25}\smash{\begin{tabular}[t]{l}$e^1$\end{tabular}}}}%
    \put(0.85150381,0.27725568){\color[rgb]{0,0,0}\makebox(0,0)[lt]{\lineheight{1.25}\smash{\begin{tabular}[t]{l}$d^1$\end{tabular}}}}%
  \end{picture}%
\endgroup%
}
\caption{The intravergent bridge diagram of Figure~\ref{fig:stronginv} with intersection points relevant for the computation of the chain complex in the fixed set of the involution $\iota_{A,H}$ labelled.}
\label{fig:stronginvgenerators}
\end{figure}

For the bridge diagram in Figure~\ref{fig:stronginvgenerators}, the intersection points in the fixed set of the involution $\iota_{A,H}$ are the six points
\begin{align}
\lambda_0 e_1 e_2, \qquad \lambda_0 d_0^1 d_0^2 , \qquad \lambda_0 d_1^1 d_1^2 \label{eq:fixedcomplex}\\
\lambda_1 e_1 e_2, \qquad \lambda_1 d_0^1 d_0^2 , \qquad \lambda_1 d_1^1 d_1^2.  \nonumber
\end{align}
We remind the reader that we refer to the absolute homological grading on $\Kh_{symp}$ as the Khovanov grading, to distinguish it from the homological grading on Khovanov homology; recall that in Conjecture \ref{con:isom}) the Khovanov grading $i$ corresponds to $j-q$, where $(j,q)$ are the homological and quantum gradings on combinatorial Khovanov homology. In the present example, the Khovanov gradings may be computed via the algorithm for bridge diagrams given by Cheng in \cite[Section 2.2]{Cheng23}, and are
\begin{align*}
|\lambda_0 e^1 e^2| =-6 , \qquad |\lambda_0 d_0^1 d_0^1| =-4 , \qquad |\lambda_0 d_1^1 d_1^2|=-2 \\
|\lambda_1 e^1 e^2| =-5 , \qquad |\lambda_1 d_0^2 d_0^2| =-3 , \qquad |\lambda_1 d_1^1 d_1^2|=-1 .  
\end{align*}
Following the logic of Lemma~\ref{lemma:fixed} we see that in this case the differential in the fixed set consists exclusively of strips connecting intersection points of the form $\lambda_0 a b$ and $\lambda_1 a b$, which come in pairs, so that the total differential is zero and the homology agrees with \eqref{eq:fixedcomplex}.
We now consider the mapping cone of the symplectic axis-moving map, and show that the resulting chain complex is chain homotopy equivalent to the complex above, exhibiting the isomorphism of Theorem~\ref{thm:samediff}. In the notation of Figures~\ref{fig:zero}, the intersection points in $\overline{\Sigma}_{\alpha,0}\cap \overline{\Sigma}_{\beta,0}$ are 
\[ \omega \delta_0, \qquad \omega \delta_1, \qquad \omega \epsilon, \qquad \eta \gamma \]
with Khovanov gradings
\[ |\omega \delta_0|=1, \qquad |\omega \delta_1|=2, \qquad |\omega \epsilon|=-1, \qquad |\eta \gamma|=3. \]
The absolute winding gradings can be easily checked to be
\[ w(\omega \delta_0)=0, \qquad w(\omega \delta_1)=0, \qquad w(\omega \epsilon)=-2, \qquad w(\eta \gamma)=2 \]
using the symmetry properties of the winding grading explained in Section~\ref{subsec:absolutewinding}. There are no nontrivial differentials. The subcomplex $C_0$ is generated by the type-f generators $\{\omega \delta_0, \omega \delta^1, \omega \epsilon\}$. Meanwhile, in the notation of Figures~\ref{fig:one}, the intersection points in $\overline{\Sigma}_{\alpha,1}\cap \overline{\Sigma}_{\beta,1}$ are
\[ \omega' \delta_0', \qquad \omega' \delta_1, \qquad \omega' \epsilon', \qquad \eta' \gamma' \]
with Khovanov gradings
\[ |\omega' \delta_0'|=1, \qquad |\omega' \delta_1'|=2, \qquad |\omega' \epsilon'|=-1, \qquad |\eta' \gamma'|=3 \]
and absolute winding gradings
\[ w(\omega' \delta_0')=1, \qquad w(\omega' \delta_1)=1, \qquad w(\omega' \epsilon')=-1, \qquad w(\eta' \gamma')=1. \]
There is a single nontrivial differential, $\partial(\omega' \delta_1') = \eta' \epsilon'$. The quotient complex $C_1$ is generated by the type-f generators $\{ \omega' \delta_0', \omega' \delta_1, \omega' \epsilon'\}$. In this example the symplectic axis-moving map $CF_{ann}(\overline{\Sigma}_{\alpha,0}, \overline{\Sigma}_{\beta,0}) \rightarrow CF_{ann}(\overline{\Sigma}_{\alpha,1}, \overline{\Sigma}_{\beta,1})$ sends the type-f generators in $C_0$ to zero and sends $\eta \gamma$ to $\eta'\gamma'$. The mapping cone of the symplectic axis-moving map is therefore

\begin{center}
\begin{tikzcd}[row sep=0pt]
\omega \delta_0 & \omega' \delta_0' & [-7ex]\\
\omega \delta_1 & \omega' \delta_1' &[-7ex]\phantom{ }\ar[dd, bend left]\\
\omega \epsilon &  \omega' \epsilon' &[-7ex]\\
\eta \gamma \ar[r] & \eta' \gamma' &[-7ex] \phantom{}
\end{tikzcd}
\end{center}
where the generators on the right-hand side are understood to have been shifted upward in homological grading by one. We observe that this is chain homotopy equivalent to the complex~\eqref{eq:fixedcomplex} and has homology $H_*(C_0)\oplus H_*(C_1)[-1]$ as promised.

\subsubsection{The trefoil with quotient the trefoil} We now consider the example of the right-handed trefoil with the second possible choice of half-axis, so that the quotient knot is a left-handed trefoil. This example is sometimes called $3_1^-$ in the literature. An intravergent bridge diagram compatible with this choice of half-axis is exhibited in Figure~\ref{fig:five_bridge}. Importantly, note that the arguments of Sections~\ref{subsec:mapping} and \ref{subsec:strongly-loc} do not assume that the interior of the bridge $b_0$ crossing the origin only intersects the interior of a single undercrossing arc, and continue to apply to the diagram of Figure~\ref{fig:five_bridge}. 
\begin{figure}[ht]
\fontsize{16pt}{14pt}
\scalebox{.5}{
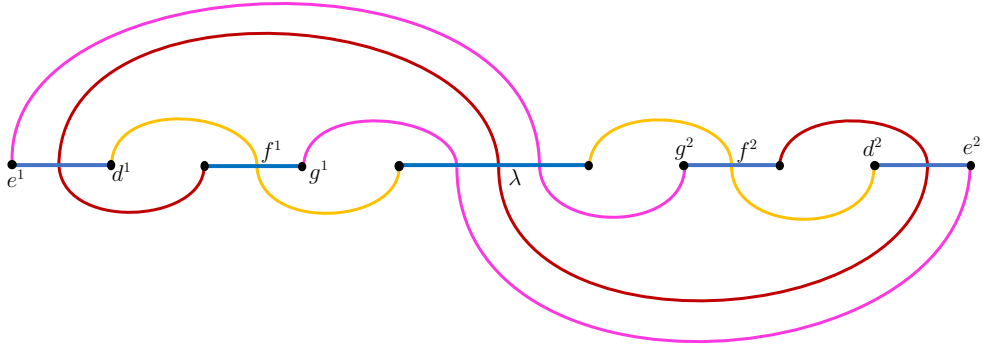}
\caption{An intravergent diagram for the right-handed trefoil with half-axis chosen such that the quotient is again a trefoil. Intersection points relevant to the fixed set are labelled.}
\label{fig:five_bridge}
\end{figure}
Once again there are six generators in the fixed set upstairs, to wit
\begin{align*}
\lambda_0 d^1 d^2 g^1 g^2, \qquad \lambda_0 e^1e^2 f^1_0 f^2_0, \qquad  \lambda_0 e^1e^2 f^1_1 f^2_1 \\
\lambda_1 d^1 d^2 g^1 g^2, \qquad \lambda_1 e^1e^2 f^1_0 f^2_0, \qquad  \lambda_1 e^1e^2 f^1_1 f^2_1
\end{align*}
with Khovanov gradings
\begin{align*}
|\lambda_0 d^1 d^2 g^1 g^2|=0, \qquad |\lambda_0 e^1e^2 f^1_0 f^2_0|=-2, \qquad  |\lambda_0 e^1e^2 f^1_1 f^2_1|=0 \\
|\lambda_1 d^1 d^2 g^1 g^2|=1, \qquad |\lambda_1 e^1e^2 f^1_0 f^2_0|=-1, \qquad  |\lambda_1 e^1e^2 f^1_1 f^2_1|=1
\end{align*}
and the total differential in the fixed set is zero. 

The two quotient resolution bridge diagrams for the annularizations of the quotient knot $\overline{K}_0$ and $\overline{K}_1$ appear in Figure~\ref{fig:five_bridge_resolutions}. The bridges and undercrossing arcs which are shared between the two diagrams are only drawn once, so that the two diagrams only differ in that the bridge diagram for $\overline{K}_0$ contains the solid red contains the solid red undercrossing arc with endpoint at $\omega$ and the bridge diagram for $\overline{K}_1$ contains the orange dashed undercrossing arc with endpoint at $\omega$. Intersection points are drawn unprimed as they would be in the diagram for $\overline{K}_0$, and should be assumed to have a primed counterpart in the diagram for $\overline{K}_1$.
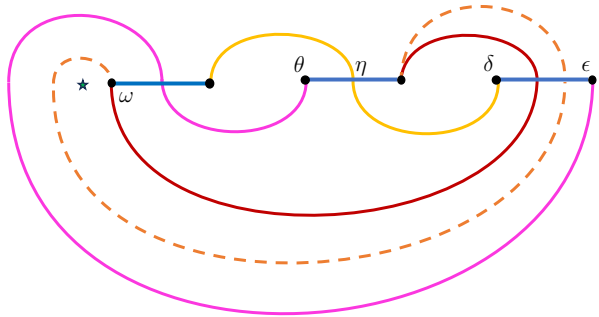
\begin{figure}[ht]
\fontsize{16pt}{14pt}
\scalebox{.5}{
\begingroup%
  \makeatletter%
  \providecommand\color[2][]{%
    \errmessage{(Inkscape) Color is used for the text in Inkscape, but the package 'color.sty' is not loaded}%
    \renewcommand\color[2][]{}%
  }%
  \providecommand\transparent[1]{%
    \errmessage{(Inkscape) Transparency is used (non-zero) for the text in Inkscape, but the package 'transparent.sty' is not loaded}%
    \renewcommand\transparent[1]{}%
  }%
  \providecommand\rotatebox[2]{#2}%
  \newcommand*\fsize{\dimexpr\f@size pt\relax}%
  \newcommand*\lineheight[1]{\fontsize{\fsize}{#1\fsize}\selectfont}%
  \ifx\svgwidth\undefined%
    \setlength{\unitlength}{468.76254272bp}%
    \ifx\svgscale\undefined%
      \relax%
    \else%
      \setlength{\unitlength}{\unitlength * \real{\svgscale}}%
    \fi%
  \else%
    \setlength{\unitlength}{\svgwidth}%
  \fi%
  \global\let\svgwidth\undefined%
  \global\let\svgscale\undefined%
  \makeatother%
  \begin{picture}(1,0.54623284)%
    \lineheight{1}%
    \setlength\tabcolsep{0pt}%
    \put(0,0){\includegraphics[width=\unitlength,page=1]{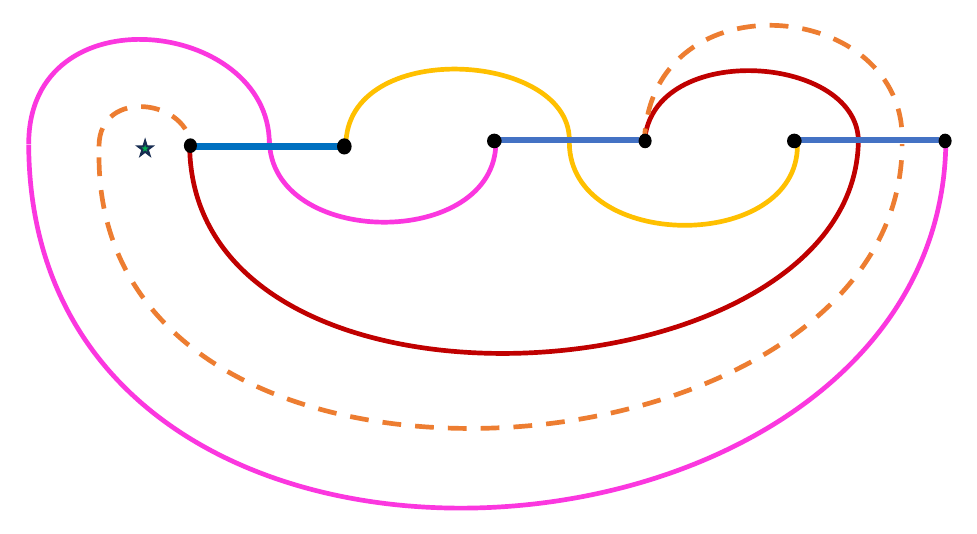}}%
    \put(0.20719177,0.36729448){\color[rgb]{0,0,0}\makebox(0,0)[lt]{\lineheight{1.25}\smash{\begin{tabular}[t]{l}$\omega$\end{tabular}}}}%
    \put(0.95034244,0.41695201){\color[rgb]{0,0,0}\makebox(0,0)[lt]{\lineheight{1.25}\smash{\begin{tabular}[t]{l}$\epsilon$\end{tabular}}}}%
    \put(0.79280818,0.41695203){\color[rgb]{0,0,0}\makebox(0,0)[lt]{\lineheight{1.25}\smash{\begin{tabular}[t]{l}$\delta$\end{tabular}}}}%
    \put(0.58647259,0.41609587){\color[rgb]{0,0,0}\makebox(0,0)[lt]{\lineheight{1.25}\smash{\begin{tabular}[t]{l}$\eta$\end{tabular}}}}%
    \put(0.4880137,0.41523968){\color[rgb]{0,0,0}\makebox(0,0)[lt]{\lineheight{1.25}\smash{\begin{tabular}[t]{l}$\theta$\end{tabular}}}}%
  \end{picture}%
\endgroup%
}
\caption{The quotient annular resolutions of the intravergent diagram of Figure~\ref{fig:five_bridge} The two bridge diagrams for $\overline{K}_0$ and $\overline{K}_1$ agree except that $\overline{K}_0$ contains the solid red undercrossing arc with endpoint at $\omega$ and $\overline{K}_1$ contains the orange dashed undercrossing arc with endpoint at $\omega$.}
\label{fig:five_bridge_resolutions}
\end{figure}
Each of $\overline{\Sigma}_{\alpha,i} \pitchfork \overline{\Sigma}_{\beta,i}$ for $i=0,1$ contains a total of sixteen intersection points. The subcomplex $C_0$ of $CF_{ann}(\overline{\Sigma}_{\alpha,0}, \overline{\Sigma}_{\beta,0})$ is generated by the type-f generators
\[ \omega \delta \theta, \qquad \omega \epsilon \eta_0, \qquad \omega \epsilon \eta_1\]
with Khovanov gradings
\[ |\omega \delta \theta|=5, \qquad |\omega \epsilon \eta_0|=3, \qquad |\omega \epsilon \eta_1|=4\]
and has trivial differential. Likewise the quotient complex $C_1$ of $CF_{ann}(\overline{\Sigma}_{\alpha,1}, \overline{\Sigma}_{\beta_1})$ is generated by the type-f generators 
\[ |\omega' \delta' \theta'|=5, \qquad |\omega' \epsilon' \eta'_0|=3, \qquad |\omega' \epsilon' \eta'_1|=4\]
and has trivial differential. The symplectic axis-moving map is zero on the three generators in $C_0$ and sends each of the remaining thirteen generators in $CF_{ann}(\overline{\Sigma}_{\alpha,0}, \overline{\Sigma}_{\beta,0})$ to its primed counterpart in $CF_{ann}(\overline{\Sigma}_{\alpha,1}, \overline{\Sigma}_{\beta,1})$. Using Theorem~\ref{thm:samediff} we see that the cone of the symplectic axis-moving map has homology
\[ H_*(C_0) \oplus H_*(C_1)[-1]\]
which is isomorphic to the homology of the fixed set upstairs.

\begin{remark} The reader may note that the choice of half-axis for the strong inversion on the right-handed trefoil did not affect the dimension of the cone of the symplectic-axis moving map. As far as the authors are aware, it unknown whether the dimension of the cone of the axis-moving map in combinatorial Khovanov homology, which is to say of the final page of the Lipshitz-Sarkar spectral sequence for strongly invertible knots of \cite{LS:strongly_invertible}, ever depends on the choice of half-axis and corresponding quotient knot for the strongly invertible knot. The same question is similarly open for our analog of their spectral sequence.
\end{remark}

\subsubsection{Crossing changes and the one-to-zero axis-moving map} As our final example, we explain a natural interpretation of the mapping cone of the symplectic axis-moving map from the symplectic annular Khovanov homology of the quotient one-resolution of an intravergent bridge diagram to the symplectic annular Khovanov homology of the quotient zero-resolution.

Let $K$ be an intravergent bridge diagram. Thinking of $K$ purely as a knot diagram, we may consider switching the overcrossing and undercrossing strands of the crossing at zero. In order to express the result as a bridge diagram, we must add two bridges; the result of doing so to the diagram of Figure~\ref{fig:stronginv} is shown in Figure~\ref{fig:crossing_change_bridge}. In this case intersection points are named such that $k$ lies in the one-resolution diagram and $k'$ lies in the zero resolution diagram, to agree with the naming of Figures~\ref{fig:zero} and \ref{fig:one}.

\begin{figure}[ht]
\fontsize{16pt}{14pt}
\scalebox{.5}{
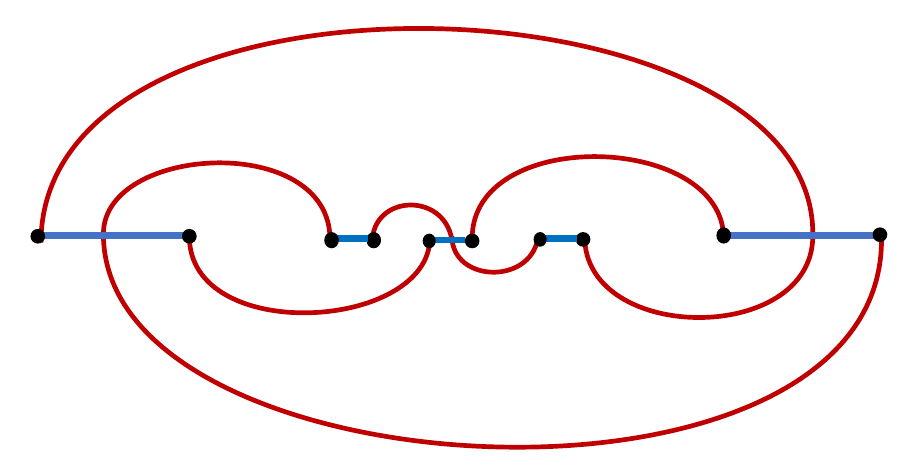}
\caption{A bridge diagram for the result of switching the central crossing in the intravergent diagram for the right-handed trefoil of Figure~\ref{fig:strongly_invertible_one}, obtained by modifying the bridge diagram of Figure~\ref{fig:stronginv} to accommodate the crossing change. The result is an unknot.}
\label{fig:crossing_change_bridge}
\end{figure}
We now consider the zero and one quotient annular resolutions of the bridge diagram in Figure~\ref{fig:crossing_change_bridge}, which are shown in Figure~\ref{fig:crossing_change_resolutions}. We observe that they are stabilizations of (respectively) the one and zero resolutions of the bridge diagram of Figure~\ref{fig:stronginv}, which appear in Figures~\ref{fig:zero} and \ref{fig:one}.
\begin{figure}[ht]
\fontsize{16pt}{14pt}
\scalebox{.5}{
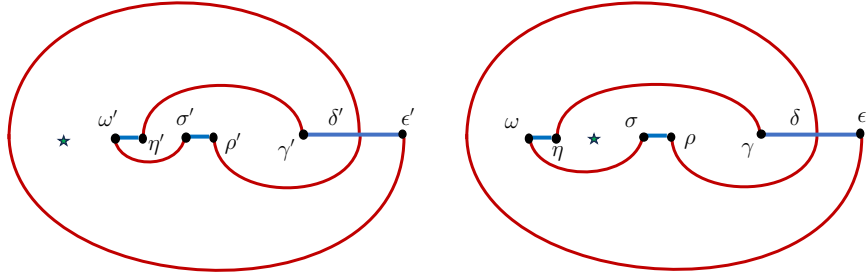}
\caption{The zero and one quotient annular resolutions of the diagram of Figure~\ref{fig:crossing_change_resolutions}. Notice that these are stabilizations of respectively the one and zero quotient annular resolutions of the diagram of Figure~\ref{fig:stronginv}.}
\label{fig:crossing_change_resolutions}
\end{figure}
We observe that the cone of the symplectic axis-moving map between the annular symplectic Khovanov homology is canonically identified with the cone of the symplectic axis-moving map from the one-resolution to the zero-resolution of the original diagram. This mapping cone is drawn below; its homology has rank two.
\begin{center}
\begin{tikzcd}[row sep=0pt]
\omega \delta_0 \rho & \omega' \delta_0' \rho' \ar[l] & [-7ex]\phantom{ }\ar[ddd, bend left]\\
\omega \delta_1 \rho & \omega' \delta_1' \rho' \ar[l] &[-7ex]\\
\omega \epsilon \rho &  \omega' \epsilon' \rho' \ar[l] &[-7ex]\\
\eta \gamma \sigma  & \eta' \gamma' \sigma' &[-7ex] \phantom{}
\end{tikzcd}
\end{center}

\section{Stable normal trivializations} \label{sec:normal-trivs}

In this section we confirm that all of the situations in this paper to which we apply Seidel-Smith localization in fact satisfy the hypotheses of Theorem~\ref{thm:localization} and Corollary~\ref{c:localization}. In all cases the symplectic hypotheses are trivial; our symplectic manifolds are exact and convex at infinity, our Lagrangians are products of spheres and therefore exact, and after possibly averaging the symplectic forms appropriately, the involutions we consider are symplectic. It remains to check the algebraic-topological hypotheses of Seidel-Smith localization, namely that for each instance of a triple $(M, \eL_0, \eL_1)$ with a symplectic involution $\iota$ preserving the Lagrangians setwise, the fixed data $(M^{fix}, \eL^{fix}, \eL_1^{fix})$ admits a stable normal trivialization.

Our proof is essentially the same as Seidel and Smith's proof for the situations of their spectral sequences \eqref{eq:ssdbc} from symplectic Khovanov homology to the Heegaard Floer homology of the branched double cover of the mirror of the link, and \eqref{eq:ss} from the symplectic Khovanov homology of a two-periodic link to that of its quotient \cite[Lemma 31]{SS10}. We first collect some basic definitions.

\begin{definition}[Stable trivialization]
A complex (or real) vector bundle $V$ over a manifold $M$ is said to be stably trivial if there exists an isomorphism $\phi \colon V \oplus \ul{\mathbb{C}}^k \xrightarrow{\sim} \ul{\mathbb{C}}^n$ (respectively $\phi \colon V \oplus \ul{\mathbb{R}}^k \xrightarrow{\sim} \ul{\mathbb{R}}^n$) for some values of $k$ and $n$. In the complex case this is equivalent to the existence of a short exact sequence
 \[
        \begin{tikzcd}
    0 \arrow[r] & V  \arrow[r] & \ul{\mathbb{C}}^n  \arrow[r] & \ul{\mathbb{C}}^k  \arrow[r] & 0 
        \end{tikzcd}
    \]
and similarly in the real case mutatis mutandis. In particular, the choice of a stable trivialization $\phi$ is equivalent to a choice of splitting of the exact sequence above.
\end{definition}

\begin{definition}[Relative stable trivialization with respect to Lagrangian subbundles]
    Let $M$ be a symplectic manifold and $\eL$ be a Lagrangian submanifold in $M$. We say that the pair $(TM,T\eL)$ is relatively stably trivial if 
    \begin{itemize}
        \item $TM$ has a complex stable trivialization $\phi \colon TM \oplus \ul{\mathbb{C}}^k \xrightarrow[]{\sim} \ul{\mathbb{C}}^n$
        \item $T\eL$ is stably trivial as a real vector bundle
        \item If $\Upsilon(M) \rightarrow M \times [0,1]$ is the pullback of $TM$ to $M \times [0,1]$, there is a Lagrangian subbundle $\Lambda \subset (\Upsilon(TM)\oplus \ul{\CC}^k)|_{\eL\times [0,1]}$ such that $\Lambda|_{\{0\} \times \eL} = T\eL \oplus \ul{\mathbb{R}}^k$ and $\phi(\Lambda|_{\{1\} \times \eL}) = \ul{\mathbb{R}}^n$.
    \end{itemize}
\end{definition}

\begin{lemma}\label{lemma: SES stable trivialization}
The pair $(TM,T\eL)$ is relatively stably trivial if there exists a complex stable trivialization of the tangent bundle $TM$ which restricts to a real stable trivialization of the tangent bundle $T\eL$. This is equivalent to the existence of a complex stable trivialization of $TM$ given by

\[
\begin{tikzcd}
0 \arrow[r] & TM  \arrow[r, "F"] & \ul{\mathbb{C}}^n  \arrow[r, "G"] & \ul{\mathbb{C}}^{k}  \arrow[r] & 0 
\end{tikzcd}
\]
and a commutative diagram of complex vector bundles over the Lagrangian $\eL$ as below
\[
        \begin{tikzcd}
    0  \arrow[r] & T\eL \otimes_\mathbb{R}\ul{\mathbb{C}} \arrow [d, "\sim", sloped] \arrow[r, "f "] & \mathbb{R}^n \otimes_\mathbb{R}\ul{\mathbb{C}} \arrow [d, "\sim", sloped] \arrow[r, "g"] & \mathbb{R}^k \otimes_\mathbb{R}\ul{\mathbb{C}}   \arrow [d, "\sim", sloped]\arrow[r] & 0 \\
     0 \arrow[r] & TM  \arrow[r, "F"] & \ul{\mathbb{C}}^n  \arrow[r, "G"] & \ul{\mathbb{C}}^{k}  \arrow[r] & 0 
        \end{tikzcd}
\]
where the first vertical isomorphism is $v\otimes 1 \mapsto v, \; v \otimes i \mapsto Jv$ and the other vertical isomorphisms are the canonical maps given by $v\otimes 1 \mapsto v, \; v \otimes i \mapsto iv$.
    
\end{lemma}
\begin{proof}
    The first assertion, nearly trivial, is proved by the first author in \cite[Lemma 7.1]{Hen12:dcov}. We prove the second assertion by unraveling definitions. Suppose first that there is a complex stable trivialization $ \Phi \colon TM \oplus \ul{\CC}^k \xrightarrow[]{\sim} \ul{\CC}^n$ which restricts to a real stable trivialization of $T\eL$ given by $\phi \colon T\eL \oplus \ul{\RR}^k \xrightarrow[]{\sim} \ul{\RR}^n$. Then, we can choose the maps in the second row to be $F(v) = \Phi(v,0)$ and $G(v)= pr_{\ul{\CC}^k} \circ \Phi^{-1}(v)$, where $pr_{\ul{\CC}^k}$ denotes projection onto the $\ul{\CC}^k$ term. Similarly, we choose the maps in the first row to be $f(v \otimes z) = \phi(v,0) \otimes z$ and $g(v \otimes z)= \left(pr_{\CC^k} \circ \phi^{-1}(v)\right) \otimes z$. The rows are exact as $\Phi$ and $\phi$ are vector bundle isomorphisms. The squares in the diagram commute by construction since $\Phi$ is a complex vector bundle morphism and therefore $\Phi(Jv) = i \phi(v)$ and $\Phi^{-1}(iv) = i \phi^{-1}(v)$.
    
Now suppose that we have a complex stable trivialization of $TM$ given by 
\[
\begin{tikzcd}
0 \arrow[r] & TM  \arrow[r, "F"] & \ul{\mathbb{C}}^n  \arrow[r, "G"] & \ul{\mathbb{C}}^{k}  \arrow[r] & 0 
\end{tikzcd}
\]
and a commutative diagram of vector bundles over $\eL$ as in the statement of the lemma. Then, we can choose a splitting $S \co \ul{\CC}^k \to \ul{\CC}^n$ of the exact sequence above such that $G \circ S = id_{\ul{\CC}^k}$. We lift this splitting $S$ using the vertical isomorphisms to get a splitting $s \colon \ul{\RR}^k \otimes \ul{\CC} \to \ul{\RR}^n \otimes \ul{\CC}$. Then, we define $\Phi \colon TM \oplus \ul{\CC}^k \to \ul{\CC}^n$ by $\Phi (v, w) = F(v)+S(w)$. The map $\Phi$ is then an isomorphism of complex vector bundles which gives the complex stable trivialization of $TM$. Similarly, we can define $\phi \colon T\eL \oplus \ul{\RR}^k \to \ul{\RR}^n$ by $\phi(v,w) = f(v) + s(w)$, and $\phi$ gives a real stable trivialization of $T\eL$. By definition of the vertical maps, we see that $\phi$ is the restriction of $\Phi$ to the Lagrangian $\eL$, so we have shown that $(TM, T\eL)$ is relatively stably trivial.
\end{proof}

We now study the tangent bundles of manifolds that arise as the regular fiber of a smooth map.

\begin{lemma} \label{lemma:fibre}
    If $\pi \colon \mathbb{C}^n \to \mathbb{C}^k$ is a smooth map and $a \in \mathbb{C}^k$ is a regular value of $\pi$, then the tangent bundle of the fibre $M_a = \pi^{-1}(a)$ is stably trivial.
\end{lemma}

\begin{proof}
    By the regular value theorem, we know that $M_a$ is a manifold. Let $j \colon M_a \to \mathbb{C}^n$ be the inclusion of the manifold as the fibre $\pi^{-1}(a)$. Then we have a short exact sequence of vector bundles on $M_a$
 \begin{equation}
       \begin{tikzcd} \label{eq:trivialization}
    0 \arrow[r] & TM_a  \arrow[r, "dj"] & T\mathbb{C}^n|_{M_a}  \arrow[r, "p_a \times d\pi"] & M_a \times T_a \mathbb{C}^k  \arrow[r] & 0 
        \end{tikzcd}\end{equation}
where $p_a$ is the projection map $p_a \co T\mathbb{C}^n|_{M_a} \rightarrow M_a$. Thus we see that $TM_a$ is stably trivial.
\end{proof}

We derive as a corollary that the tangent bundle of $\cY_{n,\tau} = \Hilb^{n,hor}(A_{\tau})$ used in the computation of symplectic and annular symplectic Khovanov homology has stably trivial tangent bundle.

\begin{corollary}
The manifold $\cY_{n,\tau} = \Hilb^{n,hor}(A_{\tau})$ has stably trivial tangent bundle. The submanifold $\Hilb^{n,hor}(A_{\tau}^*)$ therefore does as well.
\end{corollary}
\begin{proof}
    The symplectic manifold $\cY_{n,\tau}$ is a regular fibre of the adjoint quotient map $\chi \colon S_n \to \Sym^{2n}(\mathbb{C})$ over the point $\tau \in \Conf^{2n}_0(\mathbb{C})\subseteq \Sym^{2n}(\CC)$. By choosing local coordinates at $\tau$ and recalling that $S_n \simeq \mathbb C^{4n}$, we can assume that we are in the set up of the Lemma~\ref{lemma:fibre}, from which the conclusions follow.
\end{proof}

\begin{lemma} \label{lemma:monodromy} Let $\pi \co \CC^n \rightarrow \CC^k$ be a symplectic fibration, and let $\pi^{-1}(a) = M_a$ for a regular value $a$. Let $t \mapsto a_t$ for $t \in [0,1]$ be a loop of regular values of $\pi$ in $\CC^k$, so that $a_0=a=a_1$. Consider a path of monodromy symplectomorphisms $\phi_t \co M_a \rightarrow M_{a_t}$ such that $\phi = \phi_1 \co M_a \rightarrow M_{a_1}=M_a$. Suppose that $\eL$ is a Lagrangian in $M$ such that the trivialization of $TM$ from the short exact sequence \eqref{eq:trivialization} restricts to a real stable trivialization of $TL$. Then the triple $(TM, T\eL, T(\phi(\eL)))$ has a stable normal trivialization in the sense of Definition~\ref{def:stable-normal}. \end{lemma}

\begin{proof} Consider the short exact sequence \eqref{eq:trivialization}, replicated below with the additional notation of letting $j_a \co M_a \hookrightarrow \CC^n$ be the inclusion map.
\begin{equation*}
       \begin{tikzcd}
    0 \arrow[r] & TM_a  \arrow[r, "d(j_a)"] & T\mathbb{C}^n|_{M_a}  \arrow[r, "p_a \times d\pi"] & M_a \times T_a \mathbb{C}^k  \arrow[r] & 0 
        \end{tikzcd}\end{equation*}
We observe that there is a unique choice of splitting $s_a$ which has $s_a(p_a \times d\pi) = \id$ and sends $T_a(\CC^k)$ to the orthogonal complement of $TM_a$ in $T\mathbb{C}^n|_{M_a}$ at each point. Any other choice is homotopic to this $s_a$ via applying the Gram-Schmidt algorithm. This implies that using this preferred choice of splitting for each $t$, there is a continuously varying family of stable trivializations  $\psi_t \co TM_{a_t} \times \CC^k \rightarrow M_{a_t} \times \CC^n$ given by $d(j_{a_t})\times s_{a_t}$ with the property that $\psi_0=\psi_1=\psi$. Now let $p_{\CC}$ denote the projection from $M_{a_t}\times \CC^n$ to $\CC^n$. For $(v,x) \in T(\phi(\eL))\oplus \ul{\RR}^k \subseteq TM_a \oplus \ul{\CC}^k$, consider the homotopy 
\[H_t(v,x) = p_{\CC}(\psi_{1-t}((\phi_{1-t})_*(v,x)))\]
Then $H_t$ is a homotopy of Lagrangian vector bundles between $\psi(T(\phi \eL) \oplus \ul{\RR}^k)$ and $\ul{\RR}^N$, which may be rotated to a homotopy between $\psi(T(\phi \eL) \oplus i\ul{\RR}^k)$ and $i\ul{\RR}^N$. As $\psi$ also restricts to a real stable trivialization on $T\eL \oplus \ul{\RR}^k$ we are done.
\end{proof}

We next show the tangent polarization data $(T\cY_{n,\tau}, T\eL_{\ul{c}}, T(\beta \eL_{\ul{c}}))$ on the triple $(\cY_{n,\tau}, \eL_{\ul{c}}, \beta \eL_{\ul{c}})$ is also stably trivial, where as in Section~\ref{s:SympAKh}, $\ul{c}$ is a crossingless matching, $\eL_{\ul{c}}$ is its associated Lagrangian, and the Lagrangian $\beta \eL_{\ul{c}}$ is obtained by applying the symplectic monodromy map associated to the braid $\beta$. We begin by recalling the following from \cite[Lemma 29]{SS10} and \cite[Lemma 32]{SS06}.

\begin{lemma}\label{lemma: trivial neighborhood}
    Let  $\mu_1 < \ldots < \mu_n$ be a collection of real numbers such that $\mu_i \neq 0$ and $\{\mu_1, \dots, \mu_n\}$ is invariant under $z \mapsto -z$. For $\epsilon>0$ sufficiently small, consider the polydisc neighborhood $D_\epsilon \subseteq \Conf^{2n}_0(\mathbb{C}) \subseteq \Sym^{2n}(\CC)$ consisting of configurations
\[
[\mu_1 + \sqrt{z_1}, \mu_1 - \sqrt{z_1}, \ldots, \mu_n + \sqrt{z_n}, \mu_n - \sqrt{z_n}]
\]
where $|z_i|<\epsilon$. For $\delta < \epsilon$, let $\ul{c}_{inv}$ be the crossingless matching whose image is the set of intervals $[\mu_i - \sqrt{\delta}, \mu_i +\sqrt{\delta}]$ on the real line. Then $\chi^{-1}(D_\epsilon)$ contains a properly embedded contractible totally real submanifold $\Delta \subset \chi^{-1}(D_\epsilon)$ such that $ \Delta \cap \chi^{-1}(\delta, \dots, \delta)$ is Lagrangian isotopic to $\eL_{\ul{c}_{inv}}$ in $\chi^{-1}(\delta, \dots, \delta)$.
\end{lemma}
\begin{proof}
The fibre of $\chi$ over $\{0, \ldots, 0\} \in \Sym^{2n}(\CC)$ is singular and corresponds to degenerating the $2n$ eigenvalues in pairs. From \cite[Lemma 28]{SS06}, we know that there are holomorphic local coordinates $(a_1, b_1, c_1, \ldots, a_n, b_n, c_n)$ in $\chi^{-1}(D_\epsilon)$ centered around this singular fibre with respect to which the adjoint quotient $\chi$ has the form

\begin{equation}\label{eqn:polydisc}
	\begin{tikzcd}
	\chi^{-1}(D_\epsilon)\cap S_n \arrow[r,"\sim"]\arrow[d,"\chi"] &\prod_{r=1}^{n} 
	\mathbb{C}^3 \arrow{d}{(a_{1}^{2}+b_{1}^{2}+c_{1}^{2}, \dots, a_{n}^{2}+b_{n}^{2}+c_{n}^{2})}\\
	D_\epsilon \arrow[r] &\mathbb{C}^n
	\end{tikzcd}
\end{equation}
   Let $\Delta$ be the real locus in these coordinates, i.e., $\Delta$ is the fixed set of the complex conjugation of the complex structure. Seidel and Smith show that if $(z_1,\dots, z_n)= (\delta, \dots, \delta)$ in the disc $D_{\epsilon}$ then the intersection $\eL = \Delta \cap \chi^{-1}(\delta, \dots, \delta)$ is Lagrangian isotopic in the fibre to $\eL_{\ul{c}_{inv}}$ \cite[Lemmas 32 and 36]{SS06}.
\end{proof}
Let the configuration $[\mu_1 - \sqrt{\delta}, \mu_1 +\sqrt{\delta}, \dots, \mu_i - \sqrt{\delta}, \mu_i +\sqrt{\delta}]$ be $\tau_{\delta}$.

\begin{lemma}\cite[Lemma 31]{SS10} \label{lemma: rel stable trivialization}
    The relative vector bundle $(T\cY_{n,\tau_{\delta}}, T\eL_{\ul{c}_{inv}})$ is relatively stably trivial for the crossingless matching $\ul{c}_{inv}$. Moreover, the same is true of $(T \Hilb^{n,hor}(A_{\tau_{\delta}}), \Sigma_{\ul{c}_{inv}})$, where $\Sigma_{\ul{c}_{inv}}$ denotes the Manolescu product Lagrangian.
\end{lemma}
\begin{proof}
    Consider the point $(z_1,\dots, z_n)= (\delta, \dots, \delta)$ in the disc $D_{\epsilon}$ and the intersection $\eL = \Delta \cap \chi^{-1}(\delta, \dots, \delta)$. We obtain a short exact sequence of real vector bundles
    \[
        \begin{tikzcd}
    0 \arrow[r] & T\eL  \arrow[r, "dj"] & T\Delta|_{\eL}  \arrow[rr, "id\times d\chi"] & & \eL \times T_{\tau_{\delta}}(\Conf^{2n}_0(\mathbb{C}))  \arrow[r] & 0 
        \end{tikzcd}
    \]
where $j$ is the inclusion map $j \colon \eL \to \Delta$ and $\chi$ is the restriction of the adjoint quotient map $\chi \colon \Delta \to D_\epsilon$. This shows that the tangent space $T\eL$ is stably trivial as a real vector bundle. Complexifying this short exact sequence, we obtain the following diagram.
\begin{equation}\label{eq: rel stable trivialization}
	\begin{tikzcd}[row sep=large, column sep=large]
    0  \arrow[r] & T\eL \otimes_\mathbb{R}\mathbb{C} \arrow [d, "\sim", sloped] \arrow[r, "dj"] & T\Delta|_{\eL} \otimes_\mathbb{R}\mathbb{C} \arrow [d, "\sim", sloped] \arrow[r, "id\times d\chi"] & \eL \times \mathbb{C}^{k}  \arrow [d, "\sim", sloped]\arrow[r] & 0 \\
     0 \arrow[r] & T\chi^{-1}(\delta, \ldots, \delta)|_{\eL}  \arrow[r, "dj"] & T\chi^{-1}(D_{\epsilon})|_{\eL}  \arrow[r, "id\times d\chi"] & \eL \times \mathbb{C}^{k}  \arrow[r] & 0 
	\end{tikzcd}
\end{equation}
      
    As $\eL$ is totally real in $\chi^{-1}(\delta, \ldots, \delta)$, we see that the complexification $T\eL \otimes_\mathbb{R}\mathbb{C}$ is isomorphic to $T\chi^{-1}(\delta, \ldots, \delta)|_{\eL}$. The vertical isomorphisms are as described in Lemma~\ref{lemma: SES stable trivialization} , and since the hypotheses of Lemma ~\ref{lemma: SES stable trivialization} are satisfied, we see that the relative vector bundle $(T\chi^{-1}(\delta, \ldots, \delta), T\eL)$ is stably trivial.
    
Now, recall that $\eL$ is isotopic to $\eL_{\ul{c}_{inv}}$ inside $\chi^{-1}(\delta, \ldots, \delta)$ . Hence, the relative vector bundle $(T\chi^{-1}(\delta, \ldots, \delta) , T\eL)$ is carried by parallel transport to the relative bundle $(T\cY_{n,\tau_{\delta}}, T\eL_{\ul{c}_{inv}})$. This proves the first statement. Moreover, $L_{\ul{c}_{inv}}$ is Lagrangian isotopic to the product Lagrangian $\Sigma_{\ul{c}_{inv}}$. Thus, we see that the relative vector bundle for the product Lagrangian $(T\cY_{n,\tau}, T\Sigma_{\ul{c}_{inv}}) = (T\Hilb^{n,hor}(A_{\tau_{\delta}}), T\Sigma_{\ul{c}_{inv}})$ is also stably trivial. %
 %
\end{proof}

\begin{lemma} Let ${\ul{c}}$ and $\ul{d}$ be two crossingless matchings for the configuration $\tau_{\delta}$.
The tangent polarization $(T\cY_{n,\tau_{\delta}}, T\eL_{\ul{c}}, T\eL_{\ul{d}})$ is stably trivial, as is the tangent polarization $(T\Hilb^{n,hor}(A_{\tau_{\delta}}), T\Sigma_{\ul{c}}, T\Sigma_{\ul{d}})$.
\end{lemma}
\begin{proof} Up to symplectomorphism, we may take $\ul{c}=\ul{c}_{inv}$. From the previous lemma, we see that $(T\cY_{n,\tau_{\delta}}, TL_{\ul{c}_{inv}})$ is stably trivial. Moreover, $\eL_{\ul{d}}$ can be obtained from $L_{\ul{c}_{inv}}$ by symplectic parallel transport in the fibration $S_n \to \Sym^{2n}(\mathbb{C})$. Lemma~\ref{lemma:monodromy} now implies that  $(T\cY_{n,\tau_{\delta}}, T\eL_{\ul{c}}, T\eL_{\ul{d}})$ is stably trivial. The second statement follows from the first via the usual isotopy of the Seidel-Smith vanishing cycle Lagrangian and the Manolescu product Lagrangian.
\end{proof}

We now turn our attention to equivariant trivializations. Consider a smooth vector bundle $\pi \colon V \to B$ such that the total space $V$ has a $\Z/2\Z$ action $\iota \co V \to V$ which is an order-two fibrewise linear automorphism of $V$. Then $V$ splits as $V \cong V^{inv} \oplus V^{anti}$ where $V^{inv}$ is the invariant subbundle  of $V$ preserved by the $\Z/2\Z$ action, corresponding to the eigenvalue $+1$ in each fibre, and $V^{anti}$ is the anti-invariant subbundle of $V$, corresponding to the eigenvalue $-1$ in each fibre.

\begin{definition}[Equivariant stable trivialization]\label{def: eqvt stable trivialization}
We say that a stable trivialization of $V$ is $\iota$-equivariant (or $\Z/2\Z$ -equivariant) if there exists an order-two linear automorphism $T \co \mathbb{C}^n \to \mathbb{C}^n$ which restricts to an order-two linear automorphism $T \co \mathbb{C}^k \to \mathbb{C}^k$, with $k \leq n$, such that the following diagram, in which the two rows are identical and correspond to the stable trivialization, commutes
\[
\begin{tikzcd}
    0  \arrow[r] & V \arrow [d, "\iota"] \arrow[r, "\rho"] &B \times \mathbb{C}^n  \arrow [d, "id \times T"] \arrow[r] & B \times \mathbb{C}^k  \arrow [d, "id \times T"]\arrow[r, "\xi"] & 0 \\
     0 \arrow[r] &V \arrow[r, "\rho"] &B \times \mathbb{C}^n  \arrow[r] &B \times\mathbb{C}^k \arrow[r, "\xi"] & 0.
\end{tikzcd}
\]
Equivariant relative stable trivializations are defined similarly.
\end{definition}

\begin{lemma}\label{lemma: eqvt stable trivialization}
    If a smooth complex vector bundle $V \to B$ has a $\Z/2\Z$ equivariant stable trivialization, then there is an induced a stable trivialization of the invariant subbundle $V^{inv}$ and anti-invariant subbundle $V^{anti}$. Similarly, if $V$ is a complex vector bundle with a Lagrangian subbundle $W$, and the pair $(V,W)$ has an $\iota$-equivariant relative stable trivialization, then this induces a relative stable trivialization of the invariant pair $(V^{inv}, W^{inv})$ and anti-invariant pair $(V^{anti}, W^{anti})$.
\end{lemma}

\begin{proof}
Consider the $\Z/2$ equivariant stable trivialization for $V$ with notation as in Definition ~\ref{def: eqvt stable trivialization}.
\[
\begin{tikzcd}
    0  \arrow[r] & V \arrow [d, "\iota"] \arrow[r,"\rho"] &B \times \mathbb{C}^n  \arrow [d, "id \times T"] \arrow[r, "\xi"] & B \times \mathbb{C}^k  \arrow [d, "id \times T"]\arrow[r] & 0 \\
     0 \arrow[r] &V \arrow[r, "\rho"] &B \times \mathbb{C}^n  \arrow[r, "\xi"] &B \times\mathbb{C}^k \arrow[r] & 0 
\end{tikzcd}
\]
Since the involutions $\iota$ and $id \times T$ intertwine the inclusion map $\rho$ and the quotient map $\xi$, it follows that the maps $\rho$ and $\xi$ must also split. In particular we have $\rho = \rho^{inv} \oplus \rho^{anti}$ such that $\rho^{inv} \colon V^{inv} \to B \times (\mathbb{C}^n)^{inv}$ maps the invariant subbundle of $V$ to the invariant subbundle of $B \times \mathbb{C}^n$, and $\rho^{anti} \colon V^{anti} \to B \times (\mathbb{C}^n)^{anti}$ maps one anti-invariant subbundle to the other. We similarly have $\xi = \xi^{inv} \oplus \xi^{anti}$. Then we see that 
\[
\begin{tikzcd}
    0  \arrow[r]  &V^{inv}  \arrow[r,"\rho^{inv}"] &B \times (\mathbb{C}^n)^{inv}  \arrow[r, "\xi^{inv}"]  &B \times (\mathbb{C}^k)^{inv}  \arrow[r] & 0
\end{tikzcd}
\]
is a stable trivialization for the invariant subbundle $V^{inv}$ and similarly
\[
\begin{tikzcd}
    0  \arrow[r] & V^{anti}  \arrow[r,"\rho^{anti}"] &B \times (\mathbb{C}^n)^{anti}  \arrow[r, "\xi^{anti}"] & B \times (\mathbb{C}^k)^{anti}  \arrow[r] & 0
\end{tikzcd}
\]
is a stable trivialization for $V^{anti}$. The statement for relative vector bundles is similar.
\end{proof}

We now apply this discussion to our situation. For notation, recall that we have considered three involutions on $\chi^{-1}(\tau) = \cY_{n,\tau}$ for $\tau \in \Conf^{2n}(\CC) \subseteq \Sym^n(\CC)$ in this paper, to wit
\[ \iota_A(u,v,z) = (u,z,-z)\qquad \iota_B(u,v,z)=(-u,-v,-z) \qquad \iota_{D}(u,v,z) = (u,-v,z).\]
In the Appendix, we rephrase each of these in terms of the nilpotent slice coordinates on $\cY_{n,\tau}$. With respect to this rephrasing, it follows that all three involutions have a natural extension to $S_n$. We observe that $\tau_{\delta}$ and $\ul{c}_{inv}$ have been chosen so that $\cY_{n, \tau_{\delta}}$ and $L_{\ul{c}_{inv}}$ are preserved by these actions. We have the following mild extension of \cite[Lemma 29]{SS10}.

\begin{lemma} The submanifold $\Delta \subseteq \chi^{-1}(D_{\epsilon})$ of Lemma~\ref{lemma: trivial neighborhood} is preserved by the natural extension of any of the three involutions induced on $\cY_{n,\tau}$ by $\iota_A$, $\iota_B$ or $\iota_D$. Moreover, the isotopy between $\chi^{-1}(\delta, \dots, \delta) \cap \Delta$ and $\eL_{\ul{c}_{inv}}$ may be taken to be $\iota$-equivariant. \end{lemma}

\begin{proof} As previously, let $\Delta$ be the real locus in these coordinates given in \eqref{eqn:polydisc} the $\chi^{-1}(D_{\epsilon})$ so that, $\Delta$ is the fixed set of the complex conjugation of the complex structure. Let $\iota$ be the extension of any of the involutions on $Y_{n, \tau}$ considered in this paper. As $\iota$ preserves each fibre of $\chi$, it preserves $\chi^{-1}(D_{\epsilon})$. Moreover, as $\iota$ is a symplectic involution, it preserves the complex structure and hence the real locus of the complex structure in this neighborhood. Thus, the neighborhood $\Delta$ constructed above is preserved by $\iota$. The second assertion follows immediately from the construction of the isotopy in \cite[Lemma 32]{SS06}.\end{proof}

Moreover, the isotopy between the vanishing cycle Lagrangians and Manolescu's product Lagrangians \cite[Lemmas 4.2 and 4.3]{Manolescu:nilpotent} may be made equivariant with respect to any of $\iota_A$, $\iota_B$, or $\iota_D$ as follows.

\begin{lemma}\label{lemma: Manolescu lagrangians}
	The Lagrangian isotopy relating the Seidel-Smith Lagrangian $\eL_{\ul{c}_{inv}}$ and the Manolescu product Lagrangian $\Sigma_{\ul{c}_{inv}} = \prod_{i=1}^{n} \Sigma_{c}$ associated to the crossingless matching $\ul{c}$ is equivariant with respect to any of the three involutions $\iota_A$, $\iota_B$ and $\iota_D$.
\end{lemma}

\begin{proof} In the Appendix, we translate each of the three involutions into coordinates on the nilpotent slice. With
respect to this description, the Seidel-Smith K\"ahler form $\omega$ on $S_n$ is $\iota$-equivariant for $\iota$ any of $\iota_A$, $\iota_B$, and $\iota_D$. Meanwhile, Manolescu constructs an extension of the form $\omega_{\cY}$ to a form $\Omega_{\cY}$ $S_n$ which agrees with the product form induced by $\omega_{A_{\tau}}$ on each fibre on a large compact set, which contains an open subset $U$ of the two Lagrangians which itself lies inside the set
\[\mathcal V_{m,\tau} = \{[(u_i,v_i,z_i)] \in \Conf^{2n}(A_{\tau}) : z_i \neq z_j \text{ for all } i\neq j\}.\]
This form is clearly $\iota$-equivariant on $U$ and may be made universally so via averaging. In \cite[Lemma 4.2]{Manolescu:nilpotent}, Manolescu interpolates linearly between these two symplectic forms; this interpolation is $\iota$-equivariant, implying that the Lagrangians constructed by iterated parallel transport for both the symplectic forms are also $\iota$-equivariantly isotopic. In \cite[Lemma 4.3]{Manolescu:nilpotent}, Manolescu uses the torus action $\tau$ on the fibration
\[ \mathcal V_{m,\tau} \rightarrow \Sym^{2n}(\CC)\]
and associated moment map $f$ on the restriction to $U$ to show that the iterated vanishing cycle construction for $\Omega_{\cY}$ yields a product Lagrangian $\Sigma_{\ul{c}}$. The involutions $\iota_A$ and $\iota_B$ commute with the torus action $\tau$ and hence preserve the moment map $f$ and the fibre $f^{-1}(0)$. More subtly, the involution $\iota_D((u,v,z) = (u, -v, z)$ does not commute with the torus action. Nonetheless the moment map $f_\tau = - f_{\iota \circ \tau}$, and hence the fibre $f^{-1}(0)$ is still preserved. Therefore the construction of the further isotopy of \cite[Lemma 4.3]{Manolescu:nilpotent} carries through $\iota$-equivariantly. In particular, in any step when the istopy class of a Lagrangian is necessarily that of a product Lagrangian, it is in particular equivariantly isotopic to a product Lagrangian preserved by $\iota$.\end{proof}

Now, let $\iota$ be any of the three actions $\iota_A$, $\iota_B$, and $\iota_D$, and let $\ul{c}$ and $\ul{d}$ be any two crossingless matchings with endpoints on $\tau_{\delta}$.

\begin{lemma}\label{lemma: polarization eqvt}
The restriction of the polarization data $(T\cY_n, T\eL_{\ul{c}}, T\eL_{\ul{d}})$ to $(\cY_{n, \tau_{\delta}}^{\fix}, \eL_{\ul{c}}^{fix}, \eL_{\ul{d}}^{fix})$ has an $\iota$-equivariant stable trivialization, as does the restriction of the $(T\Hilb^{n,hor}(A_{\tau,\delta}), T(\Sigma_{\ul{c}}), T(\Sigma_{\ul{d}}))$ polarization data to $(\Hilb^{n,hor}(A_{\tau,\delta})^{fix}, \Sigma_{\ul{c}}^{fix}, \Sigma_{\ul{d}}^{fix})$.
\end{lemma}
\begin{proof} It follows from Lemma ~\ref{lemma: eqvt stable trivialization} that the polarization data $(T\cY_{n,\bar{\mu}}, T\eL_{\ul{c}}, T\eL_{\ul{d}})$ is stably trivial. Moreover, as the neighborhoods $\Delta$ and $\chi^{-1}(D_\epsilon)$ constructed in Lemma ~\ref{lemma: trivial neighborhood} which realize the relative stable trivialization are $\iota$-equivariant, we see that $\iota$ will intertwines the commutative diagram of equation ~\ref{eq: rel stable trivialization}. Thus, the relative stable trivialization from Lemma ~\ref{lemma: rel stable trivialization} is $\iota$-equivariant. Lemma~\ref{lemma: Manolescu lagrangians} implies that the same is true if we consider the product Lagrangians.
\end{proof}
    
\begin{theorem}\label{thm: stable normal trivialization}
If $\iota$ is any of the three actions $\iota_A$, $\iota_B$, and $\iota_D$, then polarization data $(\cY_{n, \tau_{\delta}}^{\fix}, \eL_{\ul{c}}^{fix}, \eL_{\ul{d}}^{fix})$ admits a stable normal trivialization in the sense of Definition~\ref{def:stable-normal}, and likewise if we substitute the product Lagrangian triple $(\Hilb^{n,hor}(A_{\tau,\delta})^{fix}, \Sigma_{\ul{c}}^{fix}, \Sigma_{\ul{d}}^{fix})$.
\end{theorem}
\begin{proof} It follows from Lemma ~\ref{lemma: polarization eqvt} and an equivariant version of the arguments of Lemma ~\ref{lemma: eqvt stable trivialization} that the relative normal bundle $(N\cY_{n,\tau_{\delta}}, N\eL_{\ul{c}}, N\eL_{\ul{d}})$ has a stable trivialization. By Lemma~\ref{lemma: Manolescu lagrangians} the same holds if we instead use the product Lagrangians.
\end{proof}

\appendix

\section{Comparison of coordinates}


We conclude by describing how the involutions 
\[\iota_A(u,v,z)=(u,v,-z), \quad \iota_B(u,v,z)=(-u,-v,-z), \quad \iota_D(u,v,z)=(u,-v,z)\]  
considered in this paper act on $\cY_{m,\tau}$ in the original nilpotent slice picture of Seidel and Smith. As explained in Remark~\ref{rem:asymmetry}, in the case of $\iota_A$ and $\iota_B$ this resolves an apparent asymmetry in the actions used to recover the symplectic analogs of the Stoffregen-Zhang and Lipshitz-Sarkar spectral sequences, which on the combinatorial side are extremely similar.

Recall from Section~\ref{subsec:nilpotent} that if $G=GL_{2m}(\CC)$ and $ \mathfrak{g}$ be its Lie algebra, we consider the nilpotent slice $S_{m} \subseteq \kg$ which is the affine subspace consisting of matrices
\[  \left( \begin{array}{cccccccccccccccccccccccc}
a_1   & 1 &         &  & b_1  &    &  & \\
a_2   & 0 & 1       &  & b_2  &    &  & \\
\dots &   & \dots   &   & \dots &  &    \\
a_{m-1}& &         & 1 & b_{m-1}&  &  & \\
a_m    & &         & 0 & b_m    &  &  & \\
c_1      &   &      &    & d_1     & 1  &   & \\
c_2 &     &     &    & d_2    &  0 &1    & \\
\dots & &      &     & \dots & & &\dots & 1 \\
c_m   & &       &     & d_{m}& &  & &0
\end{array} 
\right)\] 
and that for $ \tau \in \Conf^{2m}(\CC)$, we define
\begin{align}
\cY_{m,\tau}:= \chi^{-1}(\tau)
\end{align}
where $\chi$ is the adjoint quotient map $\chi \co S_m \to \Sym^{2m}(\CC)$, which sends a matrix to its set of coefficients of its characteristic polynomial. We group the elements of this matrix into matrices $A_i \in \mathfrak{gl}_2{\mathbb C}$ where

\[ A_i = \left( \begin{matrix} a_i & b_i \\ c_i & d_i \end{matrix} \right).\]
In \cite[Section 4.1]{SS10}, Seidel and Smith describe the nilpotent slice in terms of these matrices, so that
\[
A = \begin{pmatrix}
&A_1 &I &0 &0 &\cdots\\
&A_2 &0 &I &0 &\cdots\\
&\cdots\\
&A_{m-1} &0 &0 &\cdots &I\\
&A_m &0 &0 &\cdots &0
\end{pmatrix}.
\]
(In the original formulation, the condition $a_1=-d_1$ is also imposed, with the result that we may insist that $A_1 \in \mathfrak{sl}_2(\CC)$, but this difference is not important to what follows.)

Let  
\[
A(x) = x^mI - x^{m-1}A_1 -x^{m-2}A_2 - \cdots - A_m \in \mathfrak{gl}_2(\mathbb{C}[x])
\]
so that
\[
\det(xI-A) = \det(A(x)) = (x-\tau_1)\dots(x-\tau_{2m}) = p(x).
\]
The involution considered by Seidel and Smith for 2-periodic links in \cite[Section 5]{SS10} is as below, for $m$ even.
\begin{align*}
\iota \colon \cY_{m,\tau} &\to \cY_{m,\tau} \\
(A_1, A_2, \cdots, A_{m-1}, A_m) &\mapsto (-A_1, A_2, \cdots, -A_{m-1}, A_m)
\end{align*}
For their case, this corresponds to the involution induced by $\iota_A$ on the Hilbert scheme. We will now describe how this and the involution induced by $\iota_B$ on the Hilbert scheme translate into the nilpotent slice picture more generally. For a matrix $A \in M_{2\times 2}(\CC)$, let
\[
A =\begin{pmatrix}
	a &b\\
	c &d
\end{pmatrix}, \; \; \;
\widehat{A} = \begin{pmatrix}
				a &-b\\
				-c &d
\end{pmatrix}.
\]

\begin{lemma}\label{lemma:change_coord} The involutions induced by $\iota_A$ and $\iota_B$ on the horizontal Hilbert scheme correspond to the following maps on the nilpotent slice.
	\begin{enumerate}
	\item When $m$ is even, 
	\[
	\iota \colon	(A_1, A_2, \cdots, A_m) \mapsto (-A_1, A_2, \cdots, -A_{m-1}, A_m)
	\]
	corresponds to the involution induced by $\iota_A \co (u,v,z) \mapsto (u,v,-z)$.
	\item When $m$ is even
	\[
	\iota' \colon (A_1, A_2, \cdots, A_m) \mapsto (-\widehat{A}_1, \widehat{A}_2, \cdots, -\widehat{A}_{m-1}, \widehat{A}_m)
	\]
	corresponds to the involution induced by $\iota_B\co (u,v,z) \mapsto (-u,-v,-z).$
	\item When $m=2k+1$ is odd,
	\[
	\iota \colon (A_1, A_2, \cdots, A_{2k-1}, A_{2k}, A_{2k+1}) \mapsto (-A_1, A_2, \cdots, -A_{2k-1}, A_{2k}, - A_{2k+1})
	\]
	corresponds to the involution induced by $\iota_B \co (u,v,z) \mapsto (-u, -v, -z).$
	\item When $m=2k+1$ is odd,
	\[
	\iota' \colon (A_1, A_2, \cdots, A_m) \mapsto (-\widehat{A}_1, \widehat{A}_2, \cdots, -\widehat{A}_{2k-1}, \widehat{A}_{2k}, -\widehat{A}_{2k+1})
	\]
	corresponds to the involution induced by $\iota_A\co(u, v, z) \mapsto (u,v, -z).$
	\end{enumerate}
\end{lemma}

\begin{proof} We will sketch a proof of this lemma by understanding the change of coordinates between the Hilbert scheme picture and the nilpotent slice picture more clearly. We will work on $\mathcal{U}_m = \Sym^m(A_\tau)\setminus \Delta$, which is open in $\Hilb^{m,hor}(A_\tau)$. Note that any  $\overrightarrow{\xi} \in \Sym^m(A_\tau)\setminus \Delta$ is given by coordinates
\[
\overrightarrow{\xi} = \left \{(u_1,v_1, z_1), \dots , (u_m, v_m, z_m):  (u_i,v_i,z_i) \neq (u_j,v_j,z_j) \text{ for } i\neq j \right\}.
\]
We claim that it is sufficient to show the correspondence of the involutions over the open set $\mathcal{U}_m$. For, if the change of coordinates map from the Seidel-Smith to the Hilbert scheme coordinates is $\Phi$, let two involutions in the Seidel-Smith and Hilbert scheme coordinates be $\iota_{SS}$ and $\iota_{Man}$ respectively. Since the involutions $\iota_{SS}$ and $\iota_{Man}$ are holomorphic maps, and the change of coordinates map $\Phi$ is also holomorphic, then by the identity theorem for multiple complex variables, if $\iota_{Man} $ agrees with $\iota_{SS} \circ \Phi$ on the open set $\mathcal{U}_n$, then they must also agree on the entire complex manifold $\mathcal{Y}_{m,\tau}$.

In \cite[Section 2.1]{Manolescu:nilpotent}, the nilpotent slice is described using slightly different coordinates
\begin{align*}
A(t) &\defeq t^m -a_1' t^{m-1} + a_2' t^{m-2} \cdots + (-1)^m a_m'\\
B(t) &\defeq b_1' t^{m-1} - b_2' t^{m-2} + \cdots + (-1)^m b_m'\\
C(t) &\defeq c_1' t^{m-1} - c_2' t^{m-2} + \cdots + (-1)^m c_m'\\
D(t) &\defeq t^m -d_1' t^{m-1} + d_2' t^{m-2} \cdots + (-1)^m d_m'
\end{align*}
For  $\tau \in \Conf^{2m}(\CC)$, we have
\[
\mathcal{Y}_{m,\tau} = \{(A,B,C,D) \mid (AD - BC)(z) = (z-\tau_1) \cdots (z-\tau_{2m}) = p(z)\}
\]
(It is again the case that Manolescu imposes the slightly stricter condition $a_1=-d_1$.) Comparing this with the definition from \cite[Section 4]{SS10}, where $A_k^{ij}$ is the $ij$-th entry of the matrix $A_k$
\[
\det(A(x)) = \det \begin{pmatrix}
x^m - A_{1}^{11}x^{m-1} - A_{2}^{11}x^{m-2} - \cdots - A_m^{11}   &-A_{1}^{12}x^{m-1} - A_{2}^{12}x^{m-2} - \cdots - A_m^{12}\\
-A_{1}^{21}x^{m-1} - A_{2}^{21}x^{m-2} - \cdots + A_m^{21}		&x^m - A_{1}^{22}x^{m-1} - A_{2}^{22}x^{m-2} - \cdots - A_m^{22}
\end{pmatrix}
\]
we can identify 
\begin{align*}
A_1 &= \begin{pmatrix}
a_1' &b_1'\\
c_1' &d_1'
\end{pmatrix} \; \; \;
&A_2 &= \begin{pmatrix}
-a_2' &-b_2'\\
-c_2' &-d_2'
\end{pmatrix}\\
\cdots \\
A_{2k-1} &= \begin{pmatrix}
a_{2k-1}' &b_{2k-1}'\\
c_{2k-1}' &d_{2k-1}'
\end{pmatrix} \; \; \;
&A_{2k} &= \begin{pmatrix}
-a_{2k}' &-b_{2k}'\\
-c_{2k}' &-d_{2k}'
\end{pmatrix}
\end{align*}

From \cite[Remark 2.8]{Manolescu:nilpotent} we see that $\overrightarrow{\xi} = \{ (u_1,v_1, z_1), \cdots , (u_m, v_m, z_m)\}$ is obtained from the polynomials $A(t), B(t), C(t), D(t)$ as follows. The coordinates $z_k$ are the roots of $A(t)$, and for each $k$ we have 
\begin{align*}
u_k &= \frac{1}{2}(B+C)(z_k) \defeq U(z_k)\\
v_k &= \frac{1}{2 \sqrt{-1}}(B-C) (z_k) \defeq V(z_k)
\end{align*}
Now, when $m$ is even
\[
\iota \colon (A_1, A_2, \cdots, A_{m-1}, A_m) \mapsto (-A_1, A_2, \cdots, -A_{m-1}, A_m)
\]
corresponds to
\begin{align*}
A(t) &\mapsto A(-t)\\
B(t) &\mapsto B(-t) \;\; &U(t) \mapsto U(-t)\\
C(t) &\mapsto C(-t) \;\; &V(t) \mapsto V(-t)\\
D(t) &\mapsto D(-t)
\end{align*}
Thus, $\iota$ maps $z_k$ to $-z_k$. We see that
\begin{align*}
U(-(-z_k))= U(z_k) =u_k\\
V(-(-z_k))= V(z_k) = v_k
\end{align*}
so $\iota$ corresponds to the involution induced the Hilbert scheme by $(u,v,z) \mapsto (u,v,-z)$. Similarly, for $m$ odd, $m=2k+1$
\[
\iota \colon (A_1, A_2, \cdots, A_{2k-1}, A_{2k}, A_{2k+1}) \mapsto (-A_1, A_2, \cdots, -A_{2k-1}, A_{2k}, - A_{2k+1})
\]
corresponds to
\begin{align*}
A(t) &\mapsto -A(-t)\\
B(t) &\mapsto -B(-t) \;\; &U(t) \mapsto -U(-t)\\
C(t) &\mapsto -C(-t) \;\; &V(t) \mapsto -V(-t)\\
D(t) &\mapsto -D(-t)
\end{align*}
which is the involution induced by $(u,v,z) \mapsto (-u,-v,-z)$ on the Hilbert scheme. We can also check that when $m$ is even
\[
\iota' \colon (A_1, A_2, \cdots, A_m) \mapsto (-\widehat{A}_1, \widehat{A}_2, \cdots, -\widehat{A}_{m-1}, \widehat{A}_m)
\]
corresponds to 
\begin{align*}
A(t) &\mapsto A(-t)\\
B(t) &\mapsto -B(-t) \;\; &U(t) \mapsto -U(-t)\\
C(t) &\mapsto -C(-t) \;\; &V(t) \mapsto -V(-t)\\
D(t) &\mapsto D(-t)
\end{align*}
which is induced by the involution $(u,v,z) \mapsto (-u,-v,-z)$ on the Hilbert scheme, and for $m=2k+1$ odd, 
\[
\iota' \colon (A_1, A_2, \cdots, A_m) \mapsto (-\widehat{A}_1, \widehat{A}_2, \cdots, -\widehat{A}_{2k-1}, \widehat{A}_{2k}, -\widehat{A}_{2k+1})
\]
corresponds to
\begin{align*}
A(t) &\mapsto -A(-t)\\
B(t) &\mapsto B(-t)\;\; &U(t) \mapsto U(-t)\\
C(t) &\mapsto C(-t)\;\; &V(t) \mapsto V(-t)\\
D(t) &\mapsto -D(-t)
\end{align*}
which is induced by the involution $(u,v,z) \mapsto (u,v,-z)$ on the Hilbert scheme. \end{proof}

We now consider the behavior of the final involution. For a matrix $A \in M_{2\times 2}(\CC)$, let $A^\intercal$ denote the transpose of $A$.

\begin{lemma}\cite[Section 4.1]{SS10}
The involution $\iota_D$ on $\Hilb^{m,hor}(A_\tau)$ induced by $(u,v,z)\mapsto (u, -v, z)$ from Section \ref{sec:knotfloer} corresponds in the Seidel-Smith coordinates to the involution
\[
(A_1, A_2, \cdots, A_{m-1}, A_m) \mapsto (A_1^{\intercal},A_2^{\intercal}, \cdots, A_{m-1}^{\intercal}, A_m^{\intercal}).
\]
\end{lemma}

\begin{proof}
Going through the intermediate coordinates $A(t)$, $B(t)$, $C(t)$, and $D(t)$ as in Lemma \ref{lemma:change_coord}, we see that
\[
(A_1, A_2, \cdots, A_{m-1}, A_m) \mapsto (A_1^{\intercal},A_2^{\intercal}, \cdots, A_{m-1}^{\intercal}, A_m^{\intercal})
\]
corresponds to 
\begin{align*}
A(t) &\mapsto A(t)\\
B(t) &\mapsto C(t) \;\; &U(t) \mapsto U(t)\\
C(t) &\mapsto B(t) \;\; &V(t) \mapsto -V(t)\\
D(t) &\mapsto D(t)
\end{align*}
As before, from \cite[Remark 2.8]{Manolescu:nilpotent}, $\overrightarrow{\xi} = \{ (u_1,v_1, z_1), \cdots , (u_m, v_m, z_m) \}$ is obtained from $A(t), B(t), U(t)$, and $V(t)$ such that the elements $z_k$ are the roots of $A(t)$ and

\begin{align*}
u_k &= \frac{1}{2}(B+C)(z_k) \defeq U(z_k)\\
v_k &= \frac{1}{2 \sqrt{-1}}(B-C) (z_k) \defeq V(z_k)
\end{align*}

Thus we see that the involution defined above on the Seidel-Smith nilpotent slice is induced by the involution $(u,v,z) \mapsto (u, -v, z)$ on the Hilbert scheme in the Manolescu reformulation. Again, we have only shown this correspondence for the open set $\mathcal{U}_n \subset \mathcal{Y}_n$, but this is sufficient as we are working with holomorphic involutions and a holomorphic change of coordinates, so the identity theorem applies.
\end{proof}

\bibliographystyle{amsalpha}
\bibliography{biblio}

\end{document}